\documentclass{amsart}
\usepackage{amssymb}
\usepackage{amsbsy, amsthm, amsmath,amstext, amsopn}
\usepackage[all]{xy}
\usepackage{amsfonts}
\usepackage{amscd}

\usepackage{graphicx}
\usepackage[T1]{fontenc}
\usepackage[latin1]{inputenc}
\usepackage{amssymb}
\usepackage{latexsym}
\usepackage{mathrsfs}
\usepackage[all]{xy}
\newtheorem{thm}{Theorem}[section]
\newtheorem{cor}[thm]{Corollary}
\newtheorem{lem}[thm]{Lemma}
\theoremstyle{definition}
\newtheorem{ex}[thm]{Example}
\theoremstyle{definition}

\newtheorem{prop}[thm]{Proposition}
\theoremstyle{definition}
\newtheorem{defn}[thm]{Definition}
\theoremstyle{remark}
\newtheorem{rem}[thm]{Remark}
\numberwithin{equation}{section}

\DeclareMathOperator{\Fred}{Fred}
\DeclareMathOperator{\im}{i}
\DeclareMathOperator{\Tr}{Tr}
\DeclareMathOperator{\Lie}{Lie}
\DeclareMathOperator{\pr}{pr}
\DeclareMathOperator{\Gr}{Gr}

\DeclareMathOperator{\id}{id}
\DeclareMathOperator{\diag}{diag}
\DeclareMathOperator{\ind}{ind}
\DeclareMathOperator{\coker}{coker}

\DeclareMathOperator{\DET}{DET}
\DeclareMathOperator{\Eig}{Eig}
\DeclareMathOperator{\Pf}{Pf}
\DeclareMathOperator{\Spin}{Spin}
\DeclareMathOperator{\Cliff}{Cliff}
\DeclareMathOperator{\ad}{ad}
\DeclareMathOperator{\Ad}{Ad}
\DeclareMathOperator{\End}{End}
\DeclareMathOperator{\Aut}{Aut}
\DeclareMathOperator{\SO}{SO}

\DeclareMathOperator{\Stab}{Stab}
\DeclareMathOperator{\Hom}{Hom}

\DeclareMathOperator{\Cl}{C\ell}
\DeclareMathOperator{\CCl}{\C\ell}

\DeclareMathOperator{\sign}{sign}
\DeclareMathOperator{\spec}{spec}
\DeclareMathOperator{\Length}{Length}
\DeclareMathOperator{\spa}{span}

\DeclareMathOperator{\slim}{s-\lim}
\newcommand{\norm}[1]{\left\Vert#1\right\Vert}
\newcommand{\trinorm}[1]{\left\vert\left\Vert#1\right\Vert\right\vert}
\newcommand{\abs}[1]{\left\vert#1\right\vert}
\newcommand{\set}[1]{\left\{#1\right\}}
\newcommand{\ip}[2]{\langle#1,#2\rangle}
\newcommand{\Real}{\mathbb R}
\newcommand{\Z}{\mathbb Z}
\newcommand{\N}{\mathbb N}

\newcommand{\To}{\longrightarrow}

\newcommand{\Aa}{\mathbb A}
\newcommand{\B}{\mathcal{B}}
\newcommand{\Bmn}{\mathbf{B}}
\newcommand{\Bb}{\mathbb B}
\newcommand{\C}{\mathbb C}
\newcommand{\D}{\mathcal{D}}
\newcommand{\Ps}{\mathbb P}
\newcommand{\F}{\mathcal{F}} 
\newcommand{\J}{\mathcal{J}}
\newcommand{\Lp}{\mathcal{L}}
\newcommand{\Hilb}{\mathcal{H}} 
\newcommand{\tensor}{\otimes}

\newcommand{\vac}{\vert 0 \rangle}

\newcommand{\CAR}{\mathfrak{A}}
\newcommand{\gl}{\mathfrak{gl}}
\newcommand{\RGr}{\Gr_\mathrm{res}}
\newcommand{\Ures}{U_\mathrm{res}}
\newcommand{\EUres}{\widehat{U}_\mathrm{res}}

\newcommand{\ures}{\mathfrak{u}_{\mathrm{res}}}

\newcommand{\ores}{\mathfrak{o}_{\mathrm{res}}}
\newcommand{\Eores}{\widehat{\mathfrak{o}}_{\mathrm{res}}}
\newcommand{\Eglres}{\widehat{\mathfrak{gl}}_{\mathrm{res}}}
\newcommand{\Eures}{\widehat{\mathfrak{u}}_{\mathrm{res}}}
\newcommand{\GLres}{GL_\mathrm{res}}
\newcommand{\EGLres}{\widehat{GL}_\mathrm{res}}
\newcommand{\EOres}{\widehat{O}_\mathrm{res}}
\newcommand{\glres}{\mathfrak{gl}_{\mathrm{res}}}

\newcommand{\isom}{\cong}
\newcommand{\hG}{\Gamma_{\!\mathcal{O}}}
\newcommand{\Olin}{\mathcal{O}_{\mathrm{lin}}}
\newcommand{\Jres}{\J_{\mathrm{res}}}
\newcommand{\Ores}{O_{\mathrm{res}}}
\newcommand{\Spinc}{\Spin^c}
\newcommand{\spib}{\mathbb S}

\newcommand{\lie}[1]{\mathfrak{#1}}
\newcommand{\rh}{\hat{r}}
\newcommand{\Dirac}{\rlap{$D$}{\,/}}
\newcommand{\KDirac}{\rlap{$\partial$}{/}}

\newcommand{\KDiracc}{\KDirac_{(N)}}
\newcommand{\spin}[1]{\mathfrak{spin}(\mathfrak{#1})}

\newcommand{\so}[1]{\mathfrak{so}(\mathfrak{#1})}
\newcommand{\adt}{\widetilde{\ad}\,}
\newcommand{\Adt}{\widetilde{\Ad}\,}

\newcommand{\Do}{\partial}
\newcommand{\vf}{\mathfrak{X}}
\newcommand{\Mat}[4]{
\left(
\begin{array}{cc}
#1 & #2 \\
#3 & #4 
\end{array} \right)
}
\newcommand{\Sk}{\mathscr{S}}

\newcommand{\FPol}{\F^{\mathrm{pol}}}
\newcommand{\normord}{\,\lower.4ex \hbox{$\circ$} \llap{\raise.8ex\hbox{$\circ$}} \,}
\newcommand{\one}{\textbf{1}}
\newcommand{\noro}[1]{\normord #1\Psi^\ast\Psi\normord}
\newcommand{\CDirac}{\mathscr{D}}
\newcommand{\domain}{\mathcal{D}}
\newcommand{\ainfb}{\overline{\mathfrak{a}}_\infty}
\newcommand{\ainf}{\mathfrak{a}_\infty}
\newcommand{\hinf}{\mathfrak{h}_\infty}
\newcommand{\Kt}[1]{\widetilde{K}_{#1}}
\newcommand{\KN}[1]{\widetilde{K}_{#1}^{(N)}}
\newcommand{\Hn}[1]{H_{#1}^{(N)}}
\newcommand{\Kn}[1]{K_{#1}^{(N)}}
\newcommand{\SPol}[1]{\spib^{\mathrm{pol},(#1)}}
\newcommand{\stlim}[1]{s\!\!-\!\!\!\!\lim_{#1\to\infty}}
\begin{document}

\title{Dirac Operator on the Restricted Grassmannian manifold}%
\author{Vesa Tähtinen}%
\address{Department of Mathematics and Statistics, University of Helsinki,  P.O. Box 68 (Gustaf Hällströmin katu 2b),
FI-00014 Helsinki, Finland}%
\email{vesa.tahtinen@helsinki.fi}%

\subjclass[2000]{17B65, 22E65, 81R10}%

\date{July 25th, 2010}%
\begin{abstract}
In section \S 6.4. of his book on current algebras \cite{Mic} Mickelsson 
notices that the infinite-dimensional Grassmannian manifold $\RGr(\Hilb,\Hilb_+)$ of Segal and Wilson,
associated to a polarized complex separable Hilbert space $\Hilb=\Hilb_+\oplus\Hilb_-$,
admits a (restricted) $\Spinc$ structure. By definition, this is a lift of the $\Ores$ bundle
of restricted orthonormal frames on $\RGr$ to an $\EOres$ bundle, where
$\EOres$ is a certain $\C^\ast$-central extension of the Banach Lie group $\Ores$. Thus one considers 
the Banach Lie group $\EOres$ as an 
infinite-dimensional
analogue of the 
finite-dimensional Lie group
$\Spinc$.

With the $\Spinc$ structure at hand, Mickelsson then moves on to consider the problem how one should define a Dirac type operator
on this infinite-dimensional manifold. He gives a candidate for a Dirac operator, but unfortunately
it turns out to be  badly diverging and he states that the construction has to be modified 
and further developed in order to yield a well-behaved operator and leaves this as an open problem. 

Motivated by the fact that for a compact Lie group $G$, the space of $L^2$-spinors on a homogeneous space
$G/H$ can be described purely in terms of representation theory by the
classical Peter-Weyl Theorem, and
noticing that the restricted Grassmannian manifold is an infinite-dimensional homogeneous Kähler manifold 
of the form $\Ures/(U(\Hilb_+)\times(\Hilb_-))$,
we use the basic representation of 
the Banach Lie group
$\Ures(\Hilb,\Hilb_+)$
on the fermionic Fock space $\F(\Hilb,\Hilb_+)$, introduced in \cite{PrSe},
to construct a well-defined Dirac type operator acting
on a relevant ``space of spinors''
of the restricted Grassmannian manifold.
As our main result we show that our Dirac type operator is an unbounded symmetric
operator with finite-dimensional kernel.

\end{abstract}
\maketitle
\footnotesize
\tableofcontents
\normalsize
\section{Introduction}
In this work we give a construction for a Dirac type 
operator on the infinite-dimensional \emph{restricted Grassmannian manifold} $\RGr(\Hilb,\Hilb_+)$ of Segal and Wilson (cf. \cite{PrSe,Mic})
using the (projective) fermionic Fock space representations of the infinite dimensional \emph{restricted unitary group}
$\Ures(\Hilb,\Hilb_+)$ of a polarized Hilbert space $\Hilb=\Hilb_+\oplus\Hilb_-$,
$$
\Ures(\Hilb,\Hilb_+):=\set{U\in U(\Hilb)\mid [\epsilon,U]\in\mathcal{L}^2},\qquad \epsilon:=\pr_{\Hilb_+}-\pr_{\Hilb_-},
$$
as well as certain infinite-dimensional spinor-representations of an infinite-di\-men\-sio\-nal complexified Clifford algebra. 

The history of the problem dates back to the
 book \cite{Mic} written by Mickelsson, where he argues that the restricted Grassmannian manifold $\RGr$ admits a natural
(restricted) $\Spinc$-structure and an attempt was made there towards a construction of a Dirac operator on
$\RGr$. However the candidate for the Dirac operator turned out to be badly diverging and Mickelsson
wrote that the construction has to be modified by at least introducing some relevant 
normal orderings for the operator. 

The first step in the construction of the Dirac type operator on the restricted Grassmannian manifold
is to notice that it is actually a (Hermitean symmetric) homogeneous space which is also a Kähler manifold. Our (purely algebraic) definition of 
the relevant space of $L^2$ spinors is then
motivated by the famous \emph{Peter-Weyl theorem} in the theory of finite-dimensional compact Lie groups, whose application in the finite-dimensional case
gives that for a homogeneous space $G/H$ which is spin, the space of $L^2$-spinors on $G/H$ is given in purely algebraic terms by the isomorphism
$$
L^2(G/H,S)\isom L^2(G\times_H\mathbb{S}_{\mathfrak{p}})\isom\widehat{\bigoplus_\lambda} V_{\lambda}\tensor(V_\lambda^\ast\tensor\mathbb{S}_{\mathfrak{p}})^H,
$$
where the Hilbert space direct sum is taken over all irreducible representations $V_\lambda$ of the compact Lie group $G$ and
$\mathbb{S}_{\mathfrak{p}}$ is a Clifford module for $\mathrm{Cliff}(\mathfrak{p})$, where 
$$
\mathfrak{p}\isom\mathfrak{g}/\mathfrak{h}\isom T_{[eH]}(G/H)
$$
is the orthogonally complementary vector subspace of the Lie subalgebra 
$\mathfrak{h}:=\Lie(H)$ in $\mathfrak{g}:=\Lie(G)$.

Since in
the finite-dimensional case a relevant Dirac operator $\Dirac$ acts separately on each summand of
$\widehat{\bigoplus_\lambda} V_{\lambda}\tensor(V_\lambda^\ast\tensor\mathbb{S}_{\mathfrak{p}})^H$ and acts
trivially on $V_\lambda$ appearing on the left hand side of each tensor product 
$V_{\lambda}\tensor(V_\lambda^\ast\tensor\mathbb{S}_{\mathfrak{p}})^H$, 
it is enough to study the behaviour of $\Dirac$ on each $H$-invariant piece $(V_\lambda^\ast\tensor\mathbb{S}_{\mathfrak{p}})^H$.
If the subgroup $H$ of $G$ happens to be connected one knows that
$$
(V_\lambda^\ast\tensor\mathbb{S}_{\mathfrak{p}})^H\isom(V_\lambda^\ast\tensor\mathbb{S}_{\mathfrak{p}})^{\lie{h}}=
\set{w\in V_\lambda^\ast\tensor\mathbb{S}_{\mathfrak{p}}\mid Xw=0\textrm{ for all }X\in\lie{h}}.
$$ Here the Lie algebra
$\lie{h}$ acts on $V_\lambda^\ast$ by restriction of the (differentiated) $\lie{g}$-re\-pre\-sen\-ta\-tion on $V_\lambda^\ast$. On
$\mathbb{S}_{\mathfrak{p}}$ the action of $\lie{h}$ is given by the composition of the \emph{spin lift} $\adt_{\lie{p}}:\lie{h}\To\spin{\lie{p}}\subset\Cliff(\lie{p})$ of the
\emph{isotropy representation} $\ad_{\lie{p}}:\lie{h}\To\so{p}$ and the action of the Clifford algebra $\Cliff(\lie{p})$ on the spin module
$\mathbb{S}_{\mathfrak{p}}$. And finally, on the tensor product $V_\lambda^\ast\tensor\mathbb{S}_{\mathfrak{p}}$ one considers
the diagonal representation of $\lie{h}$ of the two representations described above.

The Dirac operator with respect to the Levi-Civita connection acting on the $\lie{h}$-invariant space
$(V_\lambda^\ast\tensor\mathbb{S}_{\mathfrak{p}})^{\lie{h}}$ can then also be given in purely algebraic terms,
$$
\Dirac=\sum_{i=1}^{\dim\lie{p}}\left(r(X_i)\tensor X_i^\ast+1\tensor\frac{1}{2}X_i^\ast\cdot\adt_{\lie{p}} X_i\right),
$$
where the sum is taken over a basis $\set{X_i}$ of $\lie{p}$. Here we have used the identification
$\bigwedge^1(\lie{p^\ast})\longleftrightarrow\lie{p}\subset\Cliff(\lie{p})$ that follows from the well-known Chevalley isomorphism and the
dot $\cdot$ refers to Clifford algebra multiplication. If the homogeneous space $M=G/H$ happens to be a Riemannian symmetric space this further
simplifies to
\begin{equation}\label{DiracSym}
\Dirac=\sum_{i=1}^{\dim\lie{p}}r(X_i)\tensor X_i^\ast.
\end{equation}
Moreover, in this case the square $\Dirac^2$ can be written as
\begin{equation}
\Dirac^2=\Delta_{\lie{g}}+\frac{1}{8}r_M,
\end{equation}
where $\Delta_{\lie{g}}$ denotes for the Casimir operator of $\lie{g}$ and $r_M\in\Real$ is the scalar curvature of $M$.

Since the restricted Grassmannian manifold can be written as a homogeneous space 
$G/H$, where 
$G:=\Ures(\Hilb,\Hilb_+)$ and $H:=U(\Hilb_+)\times U(\Hilb_-)$ with
$U(\Hilb_+)\times U(\Hilb_-)$ connected,
we make the  simplest non-trivial choice for our Hilbert space of spinors in
the infinite-dimensional case  at hand, and set
\begin{equation}\label{L2forGrres}
L^2(Gr_{\mathrm{res}},S):=(\mathcal{F}_0\tensor\mathbb{S}_{\mathfrak{p}}^c)^{\mathfrak{h}^\C},
\end{equation}
where $\mathcal{F}_0\subset\mathcal{F}$ denotes the \emph{charge} $0$-sector of the
\emph{fermionic Fock space} $\mathcal{F}:=\F(\Hilb,\Hilb_+)$ and $\mathbb{S}_{\mathfrak{p}}^c$ is the complexification of $\mathbb{S}_{\mathfrak{p}}$
(to be precise, we shall actually take the invariant sector of the tensor product with respect to
a certain subalgebra $\hinf^\C$ of $\lie{h}^\C$, but at the moment we omit this detail). 

Recall that the fermionic Fock space $\F$ (which gives an irreducible projective unitary representation of the Lie group $\Ures$) decomposes into
charge $m$-sectors
$\F=\bigoplus_{m\in\Z}\F_m$, where
each $\F_m$ is a representation space for an irreducible projective unitary representation of the Lie algebra $\ures:=\Lie(\Ures)$ as well
as for its complexification $\glres$.
Notice also that
the Clifford algebra $\Cliff(\lie{p})$ is really well-defined in this case since the infinite-dimensional vector space
$$
\lie{p}\isom T_{[eH]}(G/H)
$$
is a Hilbert space for $\RGr$ being a Hilbert manifold. Thus there exists a natural inner product structure on $\lie{p}$, which is of course necessary for one to be able to define
an (infinite-dimensional) Clifford algebra $\Cliff(\lie{p})$. Again, there exists a direct sum decomposition
$\lie{g}=\lie{h}\oplus\lie{p}$, but the reader should notice that this is no longer an orthogonal decomposition 
like in the finite-dimensional case, since we only have a
natural inner product structure on $\lie{p}$ (given essentially by a real part of a trace of a product of two 
Hilbert-Schmidt operators) and \emph{not} on the whole space $\lie{g}$. Luckily, this is enough for our purposes.

One should also pay attention to the fact that we have complexified everything in (\ref{L2forGrres}) in comparison with its finite-dimensional analogue. This is related to
a certain tensor product \emph{vacuum structure} on $(\mathcal{F}_0\tensor\mathbb{S}_{\mathfrak{p}}^c)^{\lie{h}^\C}$, which is only applicable in the complexified setting and is the key element in making the Dirac operator
a well-defined (non-divergent) unbounded operator  on $(\mathcal{F}_0\tensor\mathbb{S}_{\mathfrak{p}}^c)^{\lie{h}^\C}$ with
a dense domain consisting of certain \emph{polynomial elements} in this space. Moreover, the diagonal action of
$\lie{h}^\C$ on the tensor product $\mathcal{F}_0\tensor\mathbb{S}_{\mathfrak{p}}^c$ requires a proper normal
ordering description which we solve.

By the results of Spera and Wurzbacher \cite{SpeWu}, one knows that 
the restricted Grassmannian manifold is an (infinite-dimensional) Hermitean symmetric space, which
suggests mimicking equation (\ref{DiracSym}). Once again, one complexifies and then goes to the limit. Our definition
is
$$
\KDirac:=\frac{1}{2}\sum_{\substack{i,j\in\Z\\ ij<0}}\hat{r}(E_{ij})\tensor\gamma(E_{ji}),
$$
where $\set{E_{ij}\mid i,j\in\Z,\, ij<0}$ is a certain natural Hilbert basis for the complexification $\lie{p}^\C$. Actually, this
definition gives a well-defined (unbounded) symmetric operator on the \emph{whole} tensor product
$\mathcal{F}_0\tensor\mathbb{S}_{\mathfrak{p}}^c$ and not just in the $\lie{h}^\C$-invariant  sector. The reason
that we restrict to look at the Dirac operator $\KDirac$ only in the $\lie{h}^\C$-invariant sector is that in general
one may hope to prove the finite-dimensionality of the kernel of a linear operator $T$ by looking at its square and showing that
$\dim\ker(T^2)<\infty$, and on the other hand
the expression for the square $\KDirac^2$ gets a more simple and more tractable form in the $\lie{h}^\C$-equivariant sector since
a certain diagonal Casimir operator of $\lie{h}^\C$ appearing in the expression for $\KDirac^2$ becomes very easy
to understand in the $\lie{h}^\C$-invariant sector.

In order to show that $\dim\ker(\KDirac^2)<\infty$ we introduce a concrete
diagonalization for the square of the Dirac operator.
What this all boils down is that on the $\mathfrak{h}^\C$-\emph{invariant} part
$(\mathcal{F}_0\tensor\mathbb{S}_{\mathfrak{p}}^c)^{\mathfrak{h}^\C}$ of the tensor product, the square of the Dirac operator
turns out to be essentially  a sum of a fermion number operator $F$
acting on the spinor part $\mathbb{S}_{\mathfrak{p}}^c$ of the tensor product and a certain normal ordered \emph{Casimir operator} $\Delta_{\lie{g}^\C}$ 
for the central extension of the Lie algebra 
$\mathfrak{gl}_{\mathrm{res}}=\mathfrak{u}_{res}\tensor\C$ acting on the
fermionic Fock space $\mathcal{F}_0$. As a Casimir operator
for the central extension it does \emph{not} commute with the representation, but
we are able to compute explicitely commutator laws that imply the wanted diagonalization. 

However, to obtain information on the diagonalization of $\KDirac^2$ in the 
$\lie{h}^\C$-invariant sector, there is still a technical complication
preventing one to do this directly: the $\lie{h}^\C$-invariant subspace is not closed with respect the various commutator operators
one needs to apply if for example one diagonalizes the Casimir operator $\Delta_{\lie{g}^\C}$ on $\F_0$ 
(see Appendix C). To overcome this technical difficulty we proceed as follows:
\begin{enumerate}
\item The square $\KDirac^2$ has an algebraic expression as an infinite sum of operators in the $\lie{h}^\C$-invariant sector, which
is \emph{not} the same expression as the algebraic expression for $\KDirac^2$ on the whole tensor product
$\mathcal{F}_0\tensor\mathbb{S}_{\mathfrak{p}}^c$ since certain terms of $\KDirac^2$ become much more simple
in the $\lie{h}^\C$-invariant sector.
\item The algebraic expression for $\KDirac^2$ in the $\lie{h}^\C$-invariant sector makes also sense as an operator when
acting on the whole tensor product $\mathcal{F}_0\tensor\mathbb{S}_{\mathfrak{p}}^c$. This trivial lift is denoted by
$T_{\KDirac^2}$. We emphasize that 
$T_{\KDirac^2}:\mathcal{F}_0\tensor\mathbb{S}_{\mathfrak{p}}^c\To \mathcal{F}_0\tensor\mathbb{S}_{\mathfrak{p}}^c$ is \emph{not} the
same operator anymore as $\KDirac^2:\mathcal{F}_0\tensor\mathbb{S}_{\mathfrak{p}}^c\To\mathcal{F}_0\tensor\mathbb{S}_{\mathfrak{p}}^c$
by what we just said in (1). However, $T_{\KDirac\,^2}$ restricts to the $\lie{h}^\C$ invariant sector of the tensor product and
inside $(\mathcal{F}_0\tensor\mathbb{S}_{\mathfrak{p}}^c)^{\mathfrak{h}^\C}$ it \emph{is} trivially true that the operators $T_{\KDirac^2}$
and $\KDirac^2$ coincide.
\item The commutator arguments used in the Appendix C
are legal for $T_{\KDirac^2}:\mathcal{F}_0\tensor\mathbb{S}_{\mathfrak{p}}^c\To \mathcal{F}_0\tensor\mathbb{S}_{\mathfrak{p}}^c$,
and produce a concrete diagonalization for it, i.e. there exists an orthonormal basis of eigenvectors, whose eigenvalues we know.
\item Knowing concretely a diagonalization for $T_{\KDirac^2}:\mathcal{F}_0\tensor\mathbb{S}_{\mathfrak{p}}^c\To \mathcal{F}_0\tensor\mathbb{S}_{\mathfrak{p}}^c$
tells something about the spectral properties of the restriction of $T_{\KDirac^2}$ to the subspace
$(\mathcal{F}_0\tensor\mathbb{S}_{\mathfrak{p}}^c)^{\mathfrak{h}^\C}$. But once again this is 
$$
\KDirac^2:(\mathcal{F}_0\tensor\mathbb{S}_{\mathfrak{p}}^c)^{\mathfrak{h}^\C}
\To(\mathcal{F}_0\tensor\mathbb{S}_{\mathfrak{p}}^c)^{\mathfrak{h}^\C}.
$$
This way we are able to prove that in the $\lie{h}^\C$-\emph{invariant sector} $\dim\ker\KDirac^2<\infty$, which
implies that in the same sector also $\dim\ker\KDirac<\infty$.
\end{enumerate}
\section*{Acknowledgments}
The author would like to express his gratitude Jouko Mickelsson for introducing the 
research problem and giving many helpful comments during the working process. The author
also thanks Tilmann Wurzbacher for many detailed comments that had the impact of improving the
manuscript significantly.
\section{Dirac operators on homogeneous spaces}
In this section we follow \cite{La}.

Let $G$ be a compact, semisimple, connected Lie group, and let $H$ be a closed connected subgroup of $G$. 
We consider the quotient space $G/H$ of left (resp. right) cosets with the quotient topology.
The resulting space is then known to be a smooth (connected orientable) manifold. If we denote by
$\pi:G\To G/H$ the projection map sending $g\in G$ to the left coset $Hg\in G/H$ (resp. right coset $gH\in G/H$) then
\begin{enumerate}
\item $\pi:G\To G/H$ is a smooth map.
\item $\pi:G\To G/H$ admits the structure of a (smooth) principal $H$-bundle, where $H$ acts on $G$ via the Lie group multiplication
map $m: G\times G\To G$, when restricted to a map $H\times G\To G$ (resp. to a map $G\times H\To G$).
\end{enumerate}

Suppose $G$ acts smoothly on a $C^\infty$ manifold $X$. If the action is \emph{transitive} 
(i.e. for each $x,y\in X$ we have $g\cdot x=y$ for some $g\in G$), then we say that $X$ is a \emph{homogeneous space}. Moreover,
if we let $H$ be the isotropy or stabilizer subgroup (i.e. the subgroup fixing a point $x\in X$), then we have an isomorphism
$X\isom G/H$.

Below we 
give a list of the most basic properties of homogeneous spaces $M:=G/H$. 

\subsection{Homogeneous vector bundles} 
\begin{defn}
A \emph{homogeneous vector bundle} over $M$ is a vector
bundle $E\To M$ together with a lifting of the $G$-action on the base manifold $M$ to
a fibrewise linear $G$-action on the fibres.
\end{defn}

\begin{ex}
Let $V$ be a representation of $H$. Viewing $\pi:G\To M$ as a principal $H$-bundle we may consider
the associated homogeneous vector bundle $G\times_H V\To M$ given by
$$
G\times_H V:=G\times V/\set{(g,v)\sim (gh,h^{-1}v)}.
$$
The $G$-action on $G\times_H V$ is just left multiplication on the $G$-component
and the projection map takes $(g,v)\mapsto gH$.
\end{ex}
On the other hand it is known that all homogeneous vector bundles on $M$ are induced by representations
of $H$.

\subsection{Functions} Functions on the homogeneous space $M$ correspond to right $H$-\emph{invariant}
functions on $G$ i.e., those functions that satisfy $f(gh)=f(g)$ for all $g\in G,\, h\in H$.

\subsection{Sections of homogeneous vector bundles} 
Let $\rho$ be an action of $H$ on a vector space $V$.
Sections of the homogeneous vector bundle
$G\times_H V$ correspond to right $H$-\emph{equivariant} maps $s:G\To V$ satisfying
$s(gh)=\rho(h)^{-1}s(g)$. The associated section is given by $s(gH)=(g,s(g))\in(G\times_H V)_{gH}$.

Now according to the Peter-Weyl theorem, we may decompose the space of $L^2$-maps from $G$ to
$V$ into a sum over the irreducible representations $V_\lambda$ of $G$,
\begin{equation}
L^2(G)\tensor V\isom\widehat{\bigoplus}_\lambda V_\lambda\tensor V_\lambda^\ast\tensor V,
\end{equation}
with respect to the action $l\tensor r\tensor\rho$ of $G\times G\times H$. The
$H$-equivariance condition is equivalent to requiring that the functions be invariant
under the diagonal $H$-action $h\mapsto r(h)\tensor\rho(h)$. Hence restricting
the Peter-Weyl decomposition above to the $H$-invariant subspace, we get
\begin{eqnarray}\label{Ltworep}
L^2(G\times_H V) &\isom& \widehat{\bigoplus}_\lambda V_\lambda\tensor(V_\lambda^\ast\tensor V)^H\nonumber\\
&\isom& \widehat{\bigoplus}_\lambda V_\lambda\tensor\Hom_H(V_\lambda,V).
\end{eqnarray}
The Lie group $G$ acts on the space of $L^2$-sections of the associated vector bundle
$G\times_H V\To G/H$ by $l(g)$, making the Hilbert space $L^2(G\times_H V)$ an infinite dimensional
representation of $G$.

\subsection{Tangent bundle} Let $\ip{\cdot}{\cdot}$ be the Killing form on $\lie{g}=\Lie(G)$. 
Since $G$ is compact semisimple, this
is an $\Ad$-invariant inner product on $\lie{g}$. Since $\lie{h}\subset\lie{g}$, where
$\lie{h}:=\Lie(H)$, we may use the Killing form to produce an orthogonal decomposition
$\lie{g}=\lie{h}\oplus\lie{p}$, where $\lie{p}:=\lie{h}^\perp\isom\lie{g}/\lie{h}$. This makes
it possible to identify the tangent space at the identity coset with $TM_{eH}\isom\lie{p}$. The group
$H$ acts on $\lie{p}=\lie{h}^\perp$ by the adjoint action (the inner product is $\Ad$-invariant, so
$\Ad H$ fixes both $\lie{h}$ and $\lie{p}$). This makes it possible to identify
the tangent bundle with $TM\isom G\times_H\lie{p}$. Since this is a homogeneous vector bundle, vector
fields on $M$ (i.e. sections of $TM$) can be identified by the above with
maps $\xi:G\To\lie{p}$ satisfying the $H$-equivariance condition
$\xi(gh)=(\Ad h^{-1})\xi(g)$.

Let $\pi:G\To M$ be the projection map. The bi-invariant metric on $G$ makes it possible to identify
$\pi^\ast TM$ with a subbundle of $TG$. Hence for any vector field $\xi\in\mathfrak{X}(M)$ the pullback
$\pi^\ast\xi$ gives a lifting of $\xi$ to a vector field on $G$. Conversely, given a
vector field $\tilde{\xi}$ on $G$, the push forward $\pi_\ast\tilde{\xi}$ is well-defined provided
that $\tilde{\xi}$ is right $H$-equivariant.

\subsection{Decomposition of isotropy representations}
Let $\lie{p}:=\lie{h}^\perp\isom\lie{g}/\lie{h}$ be the
orthogonal complement of $\lie{h}$ in $\lie{g}$ with respect to the Killing form
$\ip{\cdot}{\cdot}$ of $\lie{g}$. 

The adjoint action of $H$ on $\lie{g}\isom\lie{h}\oplus\lie{p}$
respects this sum decomposition, so that we obtain
two representations $\Ad_{\lie{h}}$ and $\Ad_{\lie{p}}$ of $H$ on $\lie{h}$ and $\lie{p}$ respectively.
Differentiating these gives us the corresponding Lie algebra representations
$\ad_{\lie{h}}$ and $\ad_{\lie{p}}$ of $\lie{h}$ on $\lie{h}$ and $\lie{p}$.

The Clifford algebra decomposes into the tensor product 
$\Cliff(\lie{g})\isom\Cliff(\lie{h})\tensor\Cliff(\lie{p})$ so that the
spin lift $\adt_{\lie{g}}:\lie{h}\To\spin{g}\subset\Cliff(\lie{g})$ of the adjoint action 
$\ad:\lie{h}\To\lie{so}(\lie{g})$
becomes the sum 
$\adt_{\lie{g}}=\adt_{\lie{h}}\tensor 1+1\tensor\adt_{\lie{p}}$ of two separate
spin actions $\adt_{\lie{h}}:\lie{h}\To\Cliff(\lie{h})$ and $\adt_{\lie{p}}:\lie{h}\To\Cliff(\lie{p})$.
This way the spin representations $\spib_{\lie{h}}$ and $\spib_{\lie{p}}$ of $\Cliff(\lie{h})$ and $\Cliff(\lie{p})$ become representations
of the Lie algebra $\lie{h}$. 
Moreover, if $\lie{h}$ is of maximal rank in $\lie{g}$, the complement $\lie{p}$ is even dimensional and
so the spin representation of $\lie{h}$ decomposes as the tensor product $\spib_{\lie{g}}\isom\spib_{\lie{h}}\tensor\spib_{\lie{p}}$
of $\lie{h}$-representations.

\subsection{Spin structures on $G/H$}
\begin{defn}
A Riemannian manifold $(M,g)$ of dimension $n$ is called a \emph{spin manifold} if its principal
$\SO(n)$ frame bundle $P$ admits a double cover $Q\To P$ that restricts to the map
$\Ad:\Spin(n)\To\SO(n)$ on each fibre. 
\end{defn}
The choice of such a principal $\Spin(n)$-bundle $Q$ on $M$ is called
a \emph{spin structure}. Given a spin structure on $M$ the tangent bundle is then the bundle
$TM\isom Q\times_{\Spin(n)}\Real^n$ associated to the double cover
$\Ad:\Spin(n)\To\SO(n)$. Associated to the spin representation $\spib$ of $\Spin(n)$, one has
the \emph{spin bundle} $S\isom Q\times_{\Spin(n)}\spib$.

\subsubsection{The homogeneous space case}
Recall that the tangent bundle to $G/H$ is the vector bundle $G\times_H \lie{p}\To G/H$ associated to the
isotropy representation $\Ad:H\To\SO(\lie{p})$. The corresponding frame bundle is the principal
$\SO(\lie{p})$-bundle $G\times_H\SO(\lie{p})$.

The homogeneous space $G/H$ is a spin manifold if and only if $\Ad$ can be lifted to a spin
representation $\Adt:H\To\Spin(\lie{p})$: 
$$
\xymatrix{
              & \Spin(\lie{p})\ar[d]^{\Ad}\\
H \ar@{-->}[ur]^{\Adt}\ar[r]^{\Ad}           & \SO(\lie{p})
}
$$
Given such a lift, the corresponding $\Spin(\lie{p})$-bundle $Q\isom G\times_H \Spin(\lie{p})$.
Letting then $\rho$ be the composition of $\Adt:H\To\Spin(\lie{p})$ 
with the spin representation of $\Spin(\lie{p})$ on the spin module $\spib_{\lie{p}}$ one obtains
the \emph{spin bundle} $S\To G/H$ as the associated vector bundle $G\times_H\spib_{\lie{p}}$.
Hence applying equation (\ref{Ltworep}) to the spinor bundle $S$ we obtain
\begin{equation}
L^2(G/H,S)\isom L^2(G\times_H\spib_{\lie{p}})\isom
\widehat{\bigoplus}_\lambda V_\lambda\tensor\Hom_H(V_\lambda,\spib_{\lie{p}}),
\end{equation}
where we sum over all irreducuble representations $V_\lambda$ of $G$.

\begin{rem}
Even if the homogeneous space $G/H$ is not a spin manifold, it might still admit
a $\Spinc$ \emph{structure}. In analogy with spin structures this happens if and only if
the isotropy representation $\Ad:H\To\SO(\lie{p})$ can be lifted to a map
$$
\Adt:H\To\Spinc(\lie{p})=\Big(\Spin(\lie{p})\times U(1)\Big)/\set{\pm(1,1)}.
$$
\end{rem}

\subsection{Levi-Civita connection}
The homogeneous space $M=G/H$ has a natural metric defined by
$$
\ip{\xi}{\eta}_M:=\ip{\pi^\ast\xi}{\pi^\ast\eta}_G.
$$
Moreover the vector field bracket on $M$ is given by
$$
[\xi,\eta]=\pi_\ast[\pi^\ast\xi,\pi^\ast\eta].
$$
The Levi-Civita connection on $M$ is then given by
$$
\nabla_\xi^M\eta=\pi_\ast\nabla_{\pi^\ast\xi}^G\pi^\ast\eta
$$
for all $\xi,\eta\in\mathfrak{X}(M)$. If we consider the vector fields
$\xi$ and $\eta$ as $H$-equivariant maps $\xi,\eta:G\To\lie{p}$, the Levi-Civita connection
then becomes
$$
(\nabla_\xi^M\eta)(g)=(\Do_\xi\eta)(g)+\frac{1}{2}[\xi(g),\eta(g)]_\lie{p},
$$
where $[\cdot,\cdot]_{\lie{p}}$ 
denotes the Lie algebra bracket on $\lie{g}$ projected onto $\lie{p}$ and
$\Do_\xi$ is the directional derivative.
This is written in shorthand notation
as $\nabla_\xi^M=\Do_\xi+\frac{1}{2}\ad_{\lie{p}}\xi$, where
$\ad_{\lie{p}}:\lie{p}\To\mathfrak{so}(\lie{p})$ is the projection of the adjoint representation
onto $\lie{p}$ given by $\ad_{\lie{p}} X:Y\mapsto [X,Y]_{\lie{p}}$. This has a spin lift
$\adt_{\lie{p}}:\lie{p}\To\lie{spin}(\mathfrak{p})$ given by
$$
\adt_{\lie{p}} X=\frac{1}{4}\sum_{i=1}^{\dim\lie{p}}X_i^\ast\cdot[X,X_i]_\lie{p},
$$
where $\{X_i\}$ is a basis of $\lie{p}$, $\{X_i^\ast\}$ its dual basis satisfying $\ip{X_i^\ast}{X_j}=\delta_{ij}$, and the
product $\cdot$ is taken inside $\Cliff(\lie{p})$ under the Chevalley identification.

From now on we assume that $M=G/H$ is a spin manifold so that it makes sense to talk about $L^2$-spinors. 
We may view the Hilbert space of $L^2$-spinors on $G/H$ as $L^2$-maps $s:G\To\spib_{\lie{p}}$ satisfying
the $H$-equivariance condition $s(gh)=(\Adt h^{-1})s(g)$.
It follows that the
Levi-Civita connection on spinors then becomes
$$
\nabla_\xi=\Do_\xi+\frac{1}{2}\adt_\lie{p}\xi,
$$
where $\xi:G\To\lie{p}$ is a vector field on $G/H$.

On a homogeneous space $G/H$, given a basis $\set{X_i}$ of $\lie{p}$, the left translates
$(g,X_i)\in G\times_H\lie{p}\isom TM$ give a frame for the tangent space to $G/H$ at the
coset $gH$.
We may again write the directional derivatives $\Do_{X_L}=r(X)$ in terms of the right action of
$\lie{g}$ on functions. The Dirac operator with respect to the Levi-Civita connection on the spin manifold $M=G/H$ thus becomes
\begin{equation}\label{DiracGH}
\KDirac=\sum_{i=1}^{\dim\lie{p}}\Big(r(X_i)\tensor X_i^\ast + 1\tensor\frac{1}{2} X_i^\ast\cdot\adt_\lie{p} X_i\Big),
\end{equation}
where the sum is taken over a basis $\set{X_i}$ of $\lie{p}$. Algebraically this Dirac operator can be viewed as an element
of $U(\lie{g})\tensor\Cliff(\lie{p})$.

Choose a basis $\set{Z_a}$ of $\lie{h}$. 
Landweber \cite{La} computes the square of the
Dirac operator $\KDirac$ on $G/H$ associated to the Levi-Civita connection to be
\begin{eqnarray}\label{awksqrGH}
\KDirac^2 &=& -\sum_i r(X_i)^2\tensor 1-\sum_i r(X_i)\tensor \adt_\lie{p} X_i+2\sum_a r(Z_a)\tensor\adt Z_a\nonumber\\
&-& 1\tensor\Big(\frac{9}{4}\sum_i(\adt_{\lie{p}} X_i)^2-\frac{3}{16}\Tr_{\lie{p}}\sum_i(\ad_{\lie{p} X_i})^2\Big).
\end{eqnarray}

\section{Dirac operators on compact Riemannian symmetric spaces}\label{DOonSymSpa}
In this section we don't assume our (finite-dimensional) Lie groups $G$ and $H$ to
be compact, unless otherwise stated.
\begin{defn}
Let $M$ be a Riemannian manifold. If for each point $p\in M$ there exists an isometry
$j_p:M\To M$ such that $j_p(p)=p$ and $(dj_p)_p=-\id_p$, then $M$
is called a \emph{Riemannian symmetric space}. The map $j_p$ is called
a (global) \emph{symmetry} of $M$ at $p$.
\end{defn}

\begin{thm}
Every Riemannian symmetric space $M$ is a homogeneous manifold, i.e.
$M\isom G/H$ with $G$ and $H$ here not necessarily compact.
\end{thm}
\begin{proof}
See e.g. \cite{He}
\end{proof}

\begin{defn}
The isotropy representation of a homogeneous space $M=G/H$ is the homomorphism
$\Ad^{G/H}:H\To GL(T_0(G/H))$ defined by
$h\mapsto (d\tau_h)_0$, where $0=eH$ and $\tau_a\,(a\in G)$ is the diffeomorphism $G/H\To G/H,\, gH\mapsto agH$.
\end{defn}

\begin{defn}
A homogeneous space $M\isom G/H$ is called \emph{reductive} if there exists a subspace $\lie{p}$ of $\lie{g}$ such that
$\lie{g}=\lie{h}\oplus\lie{p}$ and $\Ad(h)\lie{p}\subset\lie{p}$ for every $h\in H$, i.e. $\lie{p}$ is
$\Ad(H)$-invariant.
\end{defn}
\begin{ex}
If $G$ is a Lie group and $H$ is a (closed) compact subgroup of $G$, then $G/H$ is known to be a reductive
homogeneous space.
\end{ex}
\begin{rem}
If $G/H$ is a reductive homogeneous space, the isotropy representation of $G/H$ is equivalent with the
adjoint representation of $H$ on $\lie{p}$.
\end{rem}
\begin{prop}\label{RiemSymPro}
Let $M=G/H$ be a Riemannian symmetric space, and let $\lie{g}$ and $\lie{h}$ be the corresponding Lie algebras, respectively.
Then there exists an involutive automorphism $\sigma$ of $G$ for which $\lie{h}=\set{X\in\lie{g}\mid d\sigma(X)=X}$. Moreover,
\begin{enumerate}
\item $\lie{g}=\lie{k}\oplus\lie{p}$ with $\lie{p}=\set{X\in\lie{g}\mid d\sigma(X)=-X}$;
\item The subspace $\lie{p}$ is $\Ad(K)$-invariant (so that taking into account (1), $M=G/H$ is a reductive homogeneous space).
\item The following identities hold:
$$
[\lie{h},\lie{p}]\subset\lie{p},\quad [\lie{h},\lie{h}]\subset\lie{h}\quad [\lie{p},\lie{p}]\subset\lie{h}.
$$
\end{enumerate}
\end{prop}

We now make the assumption for the rest of this section that $M=G/H$ is a Riemannian symmetric space with $G$ a compact Lie group and
$K$ its closed subgroup. Moreover we assume that $M$ has a homogeneous spin structure and 
choose an orthonormal basis $\set{X_i}$ for $\lie{p}$. Then it follows
from Proposition \ref{RiemSymPro} part (iii), after recalling the definition of $\adt_{\lie{p}}$, that equation (\ref{DiracGH})
may now be written in the much simpler form

\begin{equation}
\KDirac=\sum_{i=1}^{\dim\lie{p}}r(X_i)\tensor X_i^\ast.
\end{equation}
It turns out that the awkward looking square in (\ref{awksqrGH}) obtains a really nice and conceptual form when we add
the extra condition that $M=G/H$ is symmetric:
\begin{equation}
\KDirac^2=\Delta_{\lie{g}}+\frac{1}{8}r_M,
\end{equation}
where $\Delta_{\lie{g}}$ denotes for the Casimir operator of $\lie{g}=\Lie(G)$ and $r_M\in\Real$ is
the scalar curvature of the Riemannian manifold $M$. For a proof, the reader should consult for instance \cite{Fri} \S 3.5. The
importance of this formula lies in the fact that it allows one to compute the eigenvalues of $\KDirac^2$ purely by
means of representation theory.

\section{CAR algebra representations}
In this section our reference is \cite{HuWu}.
\subsection{Fermionic Fock space}
Let first $\Hilb$ be a complex Hilbert space with inner product $\ip{\cdot}{\cdot}$ (which we will always assume to be $\C$-linear 
in the second argument). The \emph{full tensor product Hilbert space} $\hat{\tensor}\Hilb$ is then equipped with the inner product
$$
\ip{h_1\tensor\cdots\tensor h_n}{h_1'\tensor\cdots\tensor h_m'}=\delta_{n,m}\cdot\prod_{j=1}^n\ip{h_j}{h_j'}.
$$
Setting 
$$
h_1\wedge\cdots\wedge h_n=\frac{1}{\sqrt{n!}}\sum_{\sigma\in S_n}\sign(\sigma)\cdot h_{\sigma(1)}\tensor\cdots\tensor h_{\sigma(n)}
$$
one finds that
$$
\ip{h_1\wedge\cdots\wedge h_n}{h_1'\wedge\cdots\wedge h_n'}=\det\Big(\ip{h_j}{h_k'}_{1\leq j, k\leq n}\Big).
$$
Here and in the sequel we will always complete sums, tensor products and wedge products of Hilbert spaces without changing
the symbols.

Now let $\Hilb=\Hilb_+\oplus\Hilb_-$ be a complex separable Hilbert space with a given polarization and $C:\Hilb\To\Hilb$ an anti-unitary map called the
\emph{charge conjugation}. One defines following Hilbert space called
the \emph{fermionic Fock space} $\F=\F(\Hilb,\Hilb_+)$ in three steps as follows. First we define
$$
\F_{\pm}^{(0)}=\C,\quad\F_+^{(1)}=\Hilb_+\quad\textrm{and}\quad\F_-^{(1)}=C\Hilb_-
$$
and set $\F^{(n,m)}=\bigwedge^n\F^{(1)}_+\tensor\bigwedge^m\F^{(1)}_-$ and finally
$$
\F:=\bigoplus_{n,m\geq 0}\F^{(n,m)}\isom\F_+\tensor\F_-,
$$
where $\F_+:=\bigoplus_{n\geq 0}\bigwedge^n\F^{(1)}_+$ and $\F_-:=\bigoplus_{m\geq 0}\bigwedge^m\F^{(1)}_-$.

The vector $\vac:=1\in\F^{(0,0)}$ is the \emph{vacuum vector} and elements of $\F^{(n,m)}$ will be considered as describing a physical system
of $n$ particles and $m$ anti-particles.

\subsection{CAR-algebras $\CAR(\Hilb_+)$ and $\CAR(\overline{\Hilb_-})$}
For each $f\in\Hilb_+$ there exists operators $a^\ast(f), a(f)\in\B(\F)$ called
the \emph{particle creator} and the \emph{particle annihilator}, respectively. The first operator
is $\C$-linear in $f$ in $\Hilb_+$ while the second operator is $\C$-antilinear in $f$. They act on decomposable elements
$\psi^{(n,m)}=f_1\wedge\ldots\wedge f_n\tensor C g_1\wedge\ldots\wedge C g_m\in\F^{(n,m)}$ by
\begin{equation}
a^\ast(f)\psi^{(n,m)}=f\wedge f_1\wedge\ldots\wedge f_n\tensor C g_1\wedge\ldots\wedge C g_m
\end{equation}
and
\begin{equation}\label{afonpsi}
a(f)\psi^{(n,m)}=\sum_{j=1}^n (-1)^{j+1}\ip{f}{f_j}f_1\wedge\ldots\wedge\hat{f}_j\wedge\ldots\wedge f_n\tensor C g_1\wedge\ldots\wedge C g_m.
\end{equation}

These satisfy
$$
(a(f))^\ast=a^\ast(f)\quad\textrm{and}\quad \norm{a^\ast(f)}=\norm{a(f)}=\norm{f}.
$$
Moreover they obey the following \emph{canonical anticommutation relations} (CAR):

\begin{align}
\{a(f),a^\ast(f')\} &= a(f)a^\ast(f')+a^\ast(f')a(f)=\ip{f}{f'}_\Hilb\cdot 1_\F\\
\{a^\ast(f),a^\ast(f')\} &= 0=\{a(f),a(f')\}
\end{align}

\begin{defn}
The unital $C^\ast$-algebra generated by the operators $a^\ast(f)$ and $a(f')$ in $\B(\F)$ (the unital Banach algebra of bounded linear operators in the Fock space $\F$)
is called the \emph{CAR-algebra} of $\Hilb_+$ and is denoted by $\CAR(\Hilb_+)$.
\end{defn}

For each $g\in\Hilb_-$ there exists a second set of operators $b^\ast(g),b(g)\in\B(\F)$:
\begin{equation}
b^\ast(g)\psi^{(n,m)}=(-1)^n f_1\wedge\ldots\wedge f_n\tensor Cg\wedge Cg_1\wedge\ldots\wedge Cg_m
\end{equation}
and
\begin{equation}\label{bgonpsi}
b(g)\psi^{(n,m)}=(-1)^n f_1\wedge\ldots\wedge f_n\tensor
\Big(\sum_{j=1}^m (-1)^{j+1}\ip{Cg}{Cg_j} Cg_1\wedge\ldots\wedge\widehat{Cg_j}\wedge\ldots\wedge Cg_m\Big)
\end{equation}

These operators, called the \emph{anti-particle creator} and
\emph{anti-particle annihilator} respectively, satisfy
\begin{align}
(b(g))^\ast &= b^\ast(g),\quad \norm{b^\ast(g)}=\norm{b(g)}=\norm{g}\\
\{b(g),b^\ast(g')\} &=\ip{Cg}{Cg'}_\Hilb\cdot 1_\F =\ip{g'}{g}_\Hilb\cdot 1_\F\label{ipandC}\\
\{b^\ast(g),b^\ast(g')\} &= 0=\{b(g),b(g')\}
\end{align}
The operators $b^\ast(g)$ and $b(g)$ generate the \emph{complex conjugate} $\CAR(\overline{\Hilb_-})$ of the CAR-algebra of $\Hilb_-$ in $\B(\F)$.

Now the vacuum vector $\vac$ is, up to scalar multiples, the only vector in $\F$ satisfying
\begin{eqnarray}
a(f)\vac &=& 0\quad\textrm{ for all } f\in\Hilb_+,\textrm{ and }\\
b(g)\vac &=& 0\quad\textrm { for all } g\in\Hilb_-.
\end{eqnarray}

\subsection{The space $\FPol$ of polynomial elements in $\F$}
We want to consider an algebraic version of the fermionic Fock space living inside $\F$, which we call here the \emph{space of
polynomial elements} in $\F$ and denote by $\FPol$. The name refers to the defining property that elements of $\FPol$ consist of all \emph{finite} linear combinations of elements
in a fixed Hilbert basis for $\F$.

\begin{defn}
Let $\Hilb=\Hilb_+\oplus\Hilb_-$ be a polarized separable complex Hilbert space. Let $\set{u_k}$ be a countable orthonormal basis of $\Hilb_+$ and $\set{v_k}$
be a countable orthonormal basis of $\Hilb_-$. Let $\B^{(n,m)}$ be the space of all finite $\C$-linear combinations of the decomposable
elements $\psi^{n,m}\in\F^{(n,m)}$, i.e.
\begin{eqnarray}
\Bmn^{(n,m)} &:=&\spa\Big\{\psi^{(n,m)}=u_{i_1}\wedge\ldots\wedge u_{i_n}\tensor Cv_{j_1}\wedge\ldots\wedge Cv_{j_m}\Big\vert\nonumber\\
& & i_1<i_2<\cdots<i_n\textrm{ and } j_1<j_2<\cdots< j_m
\Big\}.
\end{eqnarray}
\end{defn}
\begin{defn}
The space of polynomial elements in $\F=\F(\Hilb,\Hilb_+)$ is defined as the algebraic direct sum
$$
\FPol:=\bigoplus_{\substack{\mathrm{alg}\\n,m\geq 0}}\Bmn^{(n,m)}.
$$
\end{defn}
\begin{rem}
We want to emphasize the fact that the definition of $\FPol$ depends on a chosen basis!
Notice also that the space $\FPol$ inherits naturally the inner-product structure from $\F\supset\FPol$ and  that
$\FPol$ in dense in $\F$.
\end{rem}

\begin{lem}\label{finite}
Let $\Hilb=\Hilb_+\oplus\Hilb_-$ be a polarized separable complex Hilbert space. Let $\set{u_k}$ be a countable orthonormal basis of $\Hilb_+$ and $\set{v_k}$
be a countable orthonormal basis of $\Hilb_-$. Choose a decomposable element
$\psi^{n,m}\in\F^{(n,m)}$,
$$
\psi^{(n,m)}=u_{i_1}\wedge\ldots\wedge u_{i_n}\tensor Cv_{j_1}\wedge\ldots\wedge Cv_{j_m},
$$
where $i_1<i_2<\cdots<i_n$ and $j_1<j_2<\cdots< j_m$.
Then there exists only finitely many indices $k$ and $l$ such that
$$
a(u_k)\psi^{(n,m)}\neq 0\quad\textrm{ and }\quad  b(v_l)\psi^{(n,m)}\neq 0.
$$
\end{lem}
\begin{proof}
Follows directly from the definitions (\ref{afonpsi}) and (\ref{bgonpsi}), and orthogonality of the basis $\set{u_k}$ and $\set{v_k}$.
One also needs the formula (\ref{ipandC}).
\end{proof}

Since each element of $\FPol$ is by definition a finite linear combination of decomposable elements, one has the following corollary.
\begin{cor}\label{finitepol}
Let $\Hilb=\Hilb_+\oplus\Hilb_-$ be a polarized separable complex Hilbert space. Let $\set{u_k}$ be a countable orthonormal basis of $\Hilb_+$ and $\set{v_k}$
be a countable orthonormal basis of $\Hilb_-$. Then for any fixed $\psi\in\FPol$ 
there exist only finitely many indices $k$ and $l$ such that
$$
a(u_k)\psi\neq 0\quad\textrm{ and }\quad  b(v_l)\psi\neq 0.
$$
\end{cor}

\subsection{CAR-algebra $\CAR(\Hilb,\Hilb_+)$}
\begin{defn}
For $f=f_+ + f_-$ in $\Hilb=\Hilb_+\oplus\Hilb_-$ set
\begin{align}
\Psi^\ast(f) &= a^\ast(f_+)+b(f_-)\quad\textrm{and}\\
\Psi(f) &= a(f_+)+b^\ast(f_-)
\end{align}
\end{defn}
The operators $\Psi^\ast(f)$ depend $\C$-linearly on $f\in\Hilb$ and one has
\begin{align}
\{\Psi(f),\Psi^\ast(f')\} &=\ip{f}{f'}_\Hilb\cdot 1_\F\quad\textrm{and}\\
\{\Psi^\ast(f),\Psi^\ast(f')\} &=0=\{\Psi(f),\Psi(f')\}.
\end{align}
The vacuum vector $\vac$ is characterized by
\begin{equation}
\Psi^\ast(f_-)\vac=0=\Psi(f_+)\vac
\end{equation}
for all $f_-\in\Hilb_-$ and $f_+\in\Hilb_+$.

\subsection{GNS construction}
\begin{defn}
Let $A$ be a complex unital $C^\ast$-algebra. A \emph{positive linear form} on $A$ is a $\C$-linear functional $\omega$
such that $\omega(a^\ast a)\geq 0$ for all $a\in A$.
\end{defn}
One can show that $\omega$ is bounded with $\norm{\omega}=\omega(1)$.
\begin{defn}
Suppose that $A$ is a complex unital $C^\ast$-algebra. A \emph{state} on $A$ is a positive linear form on $A$ of norm $1$.
\end{defn}
\begin{thm}[Gelfand-Naimark-Segal (GNS)] Let $\omega$ be a state on a complex unital $C^\ast$-algebra $A$. Then
there exists a Hilbert space $\Hilb_\omega^A$, a $C^\ast$-algebra representation 
$\pi_\omega^A:A\To\B(\Hilb_\omega^A)$ and
a ``vacuum vector'' $\xi$ of norm $1$ in $\Hilb_\omega^A$ such that
\begin{equation}
\omega(a)=\ip{\xi_\omega^A}{\pi_\omega^A(a)\xi_\omega^A}_{\Hilb_\omega^A}
\end{equation}
and $\pi_\omega^A(A)\cdot\xi_\omega^A$ is dense in $\Hilb_\omega^A$.
\end{thm}
We specialize to the case of CAR-algebra $A=\CAR(\Hilb,\Hilb_+)$ with unit denoted by $1$.

\begin{defn}
Let $q$ be a self-adjoint operator on a complex separable polarized Hilbert space 
$\Hilb=\Hilb_+\oplus\Hilb_-$ satisfying
$0\leq q\leq 1$. A \emph{quasi-free state} $\omega_q$ on $\CAR(\Hilb,\Hilb_+)$ 
associated to $q$ is defined by
\begin{equation}
\omega_q(1)=1\quad\textrm{and}\\
\end{equation}
\begin{equation}
\omega_q(\psi^\ast(k_1)\cdots\psi^\ast(k_n)\psi(l_1)\cdots\psi(l_m)) 
=\delta_{n.m}\det(\ip{l_i}{q(k_j)}_\Hilb).
\end{equation}
\end{defn}
\begin{prop}
The representation of the CAR-algebra $\CAR(\Hilb,\Hilb_+)$ on the fer\-mi\-onic Fock space $\F$ generated by the field operators
$\Psi^\ast$ and $\Psi$ is unitarily equivalent to the GNS-re\-pre\-sen\-ta\-tion associated to the gauge-invariant
quasi-free state $\omega_{p-}$ coming from the projection to $p_-:\Hilb\To\Hilb_-$.
\end{prop}

\subsection{The Banach Lie groups $\Ures$ and $\GLres$}\label{FermionsandUres}
Mathematically speaking one would like to lift certain unitary operators $U$ on a polarized ``one particle'' Hilbert space $\Hilb=\Hilb_+\oplus\Hilb_-$ to
unitary operators $\mathbb{U}$ on the ``many particle'' Hilbert space $\F(\Hilb,\Hilb_+)$ compatible with the so called
Bogoliubov transformations (cf. \cite{HuWu}). It turns out that this implementation problem is solvable if and only if we impose some relevant restrictions to our
unitaries in $\B(\Hilb)$. This leads to the following definition.
\begin{defn}
Define the \emph{restricted unitary group} of the 
separable complex polarized Hilbert space $\Hilb=\Hilb_+\oplus\Hilb_-$ to be
$$
\Ures=\Ures(\Hilb,\Hilb_+):=\set{U\in U(\Hilb)\mid [\epsilon,U]\in\Lp^2(\Hilb)},
$$
where $\epsilon=\pr_+-\pr_-$ is the grading operator corresponding to the polarization $\Hilb=\Hilb_+\oplus\Hilb_-$
and $\Lp^2(\Hilb)\subset\B(\Hilb)$ is the symmetrically normed two-sided $^\ast$-ideal of
the $C^\ast$-algebra $\B(\Hilb)$ consisting of all Hilbert-Schmidt operators in $\Hilb$.
\end{defn}

The topology of $\Ures$ is induced by the embedding
$$
\Ures(\Hilb,\Hilb_+)\hookrightarrow U(\Hilb)\times\Lp^2(\Hilb),\quad U\mapsto (U,[\epsilon,U]),
$$
where $U(\Hilb)$ is endowed with the operator norm topology induced from $\B(\Hilb)$,
$$
\norm{U}_{U(\Hilb)}=\sup\set{\norm{Uv}\mid v\in\Hilb,\quad\norm{v}\leq 1}
$$ 
and $\Lp^2(\Hilb)$ with the norm topology coming from the inner product 
$$
\ip{A}{B}_{\Lp^2(\Hilb)}:=\sum_{k\in\Z}\ip{Ae_k}{Be_k}_\Hilb,
$$ i.e.
\begin{equation}
\norm{A}^2_{\Lp^2(\Hilb)}=\ip{A}{A}_{\Lp^2(\Hilb)}=\sum_{k\in\Z}\norm{Ae_k}^2
\end{equation}
(here $\set{e_k}_{k\in\Z}$ is \emph{any} orthonormal basis of the Hilbert space $\Hilb$). Endowed with this topology
$\Ures$ becomes a Banach Lie group.

Let us recall the standard block-matrix description of a linear operator $A$ from a polarized Hilbert space $\Hilb=\Hilb_+\oplus\Hilb_-$ to itself:
$$
A=
\left(
\begin{array}{cc}
A_{+} & A_{+-} \\
A_{-+} & A_{-} 
\end{array} \right),
$$
where $A_\pm:\Hilb_\pm\To\Hilb_\pm$ and $A_{+-}$ (respectively $A_{-+}$) maps $\Hilb_-$ to $\Hilb_+$ (respectively $\Hilb_+$ to $\Hilb_-$).

In the block-matrix description the Lie group $\Ures(\Hilb,\Hilb_+)$ can be given by
$$
\set{A\in\B(\Hilb)\mid A^\ast=A^{-1},\quad A_{+-},A_{-+}\textrm{ are Hilbert-Schmidt operators}}.
$$
Using the fact that $A$ is unitary, an easy computation shows that $A_+$ and $A_-$ are invertible up to compact operators, thus Fredholm.
In particular, the index $\ind A_+$ (resp. $\ind A_-$) is well-defined.

Similarly, the Lie algebra $\ures(\Hilb,\Hilb_+):=\Lie(\Ures(\Hilb,\Hilb_+))$ can be described as
\begin{equation}
\set{A\in\B(\Hilb)\mid A^\ast=-A,\quad  A_{+-},A_{-+}\textrm{ are Hilbert-Schmidt operators}}
\end{equation}
with the topology given by the norm
\begin{equation}
\norm{A}_{\ures}:=\norm{A_+}+\norm{A_-}+\norm{A_{+-}}_2+\norm{A_{-+}}_2,
\end{equation}
where $\norm{\cdot}$ denotes the operator norm and $\norm{\cdot}_2$ is the Hilbert-Schmidt norm.

\begin{thm}[``Theorem of Powers and St\o rmer'' or ``\,Shale--Stinespring criterion'']
Let $\Hilb=\Hilb_+\oplus\Hilb_-$ be a polarized Hilbert space, $U:\Hilb\To\Hilb$ a unitary
operator and 
$\beta_U:\CAR(\Hilb,\Hilb_+)\To\CAR(\Hilb,\Hilb_+)$ 
the corresponding ``Bogoliubov transformation''
induced by
\begin{equation}
\beta_U(a^\ast(f)):=a^\ast(Uf).
\end{equation}
Consider the GNS-representation $\pi_-$ of the CAR algebra $\CAR(\Hilb,\Hilb_+)$ on
the Fock space $\F(\Hilb,\Hilb_+)$ associated to the gauge-invariant quasi-free state
$\omega_{p-}$ coming from the projection to $\Hilb_-$ on $\Hilb$.
Then there exists a unitary transformation $\mathbb{U}$ of the fermionic Fock space $\F(\Hilb,\Hilb_+)$
such that
\begin{equation}
\mathbb{U}\circ\pi_-(\alpha)\circ\mathbb{U}^{-1}=\pi_-(\beta_U(\alpha))\quad\textrm{ for all }
\alpha\in\CAR(\Hilb,\Hilb_-),
\end{equation}
if and only if $U\in\Ures(\Hilb,\Hilb_-)$.
\end{thm}
\begin{rem}
The unitary transformation 
$\mathbb{U}:\F(\Hilb,\Hilb_+)\To\F(\Hilb,\Hilb_+)$ in the previous theorem is said to 
\emph{implement} the Bogoliubov transformation 
$$
\beta_U:\CAR(\Hilb,\Hilb_+)\To\CAR(\Hilb,\Hilb_+)
$$ 
associated to the unitary operator
$U:\Hilb\To\Hilb$.
\end{rem}
\begin{lem}
The map $\Ures(\Hilb,\Hilb_+)\To U(\F(\Hilb,\Hilb_+))$, $U\mapsto\mathbb{U}$ given by the theorem of
Powers and St\o rmer yields a projective representation $\Theta:\Ures(\Hilb,\Hilb_+)\To\Ps U(\F)$.
\end{lem}
The map $\Theta:\Ures(\Hilb,\Hilb_+)\To\Ps U(\F)$ in the previous lemma is known to be \emph{discontinuous} when the infinite dimensional projective unitary
group $\Ps U(\F)$ is endowed with the  topology coming from the norm topology. However, $\Ps U(\F)$ is a metrizable \emph{topological} group
in the topology induced by the strong operator topology \cite{HuWu} and with respect to this topology the homomorphism
$\Theta$ \emph{is} continuous.

\begin{prop} 
Let $\Hilb=\Hilb_+\oplus\Hilb_-$ be a polarized Hilbert space. Then
\begin{enumerate}
\item the exponential map $\exp:\ures\To\Ures,\,\exp(A)=\sum_{n\geq 0}\frac{1}{n!}A^n$, gives real-analytic coordinates on $\Ures$,
\item the connected components of $\Ures$ are the sets
$$
\Ures^k:=\set{U\in\Ures\mid\ind(U_+)=k},
$$
where the \emph{index} $\ind(U_+):=\dim(\ker U_+)-\dim(\coker U_+)$ and $k\in\Z$.
\end{enumerate}
\end{prop}

The Banach  Lie group $\Ures$ has a complexification
\begin{equation}
\GLres(\Hilb,\Hilb_+):=\set{g\in GL(\Hilb)\mid [\epsilon,g]\in\Lp^2},
\end{equation}
where $GL(\Hilb)$ denotes the invertible bounded linear operators on $\Hilb$. The corresponding Lie algebra is given by
\begin{equation}
\glres(\Hilb,\Hilb_+):=\set{A\in\mathfrak{gl}(\Hilb)\mid [\epsilon,A]\in\Lp^2},
\end{equation}
where $\mathfrak{gl}(\Hilb)$ is the space $\B(\Hilb)$ of bounded linear operators on $\Hilb$.

Pressley and Segal prove in \cite{PrSe} that the Banach Lie group $\GLres(\Hilb,\Hilb_+)$ admits a \emph{central extension}
\begin{equation}
1\To \C^\ast\To\EGLres\To\GLres\To 1
\end{equation}
in the category of Banach Lie groups. This induces the central extension
\begin{equation}
1\To U(1)\To\EUres\To\Ures\To 1
\end{equation}
of $\Ures$. 

At the level of Lie algebras $\Lie(\EGLres)\isom\glres\oplus\C$, with the Lie bracket given by
\begin{equation}
\Big[(A,\alpha),(B,\beta)\Big]=\Big([A,B],s(A,B)\Big),
\end{equation}
where the Lie algebra two-cocycle $s$ is given by
\begin{equation}\label{Schwinger}
s(A,B)=\Tr(A_{-+}\cdot B_{+-}-B_{-+}\cdot A_{+-})
\end{equation}
also known in quantum field theory as the \emph{``Schwinger term''} in 1+1 dimensions.

\subsection{Fermionic second quantization and normal-ordered operators}
Next we would like to work concretely at the level of Lie algebras and lift operators 
$A\in\ures(\Hilb,\Hilb_+)$,
$A:\Hilb\To\Hilb$ to an
explicitly given projective action on the fermionic Fock space $\F:=\F(\Hilb,\Hilb_+)$. After this is done one naturally hopes to exponentiate the projective Lie algebra action to a projective
action of the Lie group $\Ures(\Hilb,\Hilb_+)$.

Let $\Hilb=\Hilb_+\oplus\Hilb_-$ be a complex separable polarized Hilbert space. We fix orthonormal Hilbert bases
$\set{e_j\mid j\geq 0}$ of $\Hilb_+$ and $\set{e_j\mid j<0}$ of $\Hilb_-$, respectively. The one-particle matrix description of
$A\in\B(\Hilb)$,
$$
Ae_j=\sum_{k\in\Z}\ip{e_k}{Ae_k}e_k,
$$
can be seen as the annihilation of $e_j$ then followed by the creation of all possible $e_k$'s weighted by the
matrix elements $\ip{e_k}{Ae_j}$. This motivates trying to use the field operators $\Psi^\ast$ and $\Psi$
acting on the fermionic Fock space $\F=\F(\Hilb,\Hilb_+)$ to lift $A\in\B(\Hilb)$ to
an operator $A\Psi^\ast\Psi$ on $\F$ by defining
$$
A\Psi^\ast\Psi:=\sum_{k\in\Z}\sum_{j\in\Z}\ip{e_k}{Ae_j}\Psi^\ast(e_k)\Psi(e_j)=\sum_{j\in\Z}\Psi^\ast(Ae_j)\Psi(e_j).
$$
Decomposing $A\Psi^\ast\Psi$ with respect to $\Hilb_+$ and $\Hilb_-$ gives
\begin{equation}\label{canquan}
A\Psi^\ast\Psi=A_+ a^\ast a +A_{+-} a^\ast b^\ast+A_{-+}ba+A_- bb^\ast.
\end{equation}
This shows that our candidate lift $A\Psi^\ast\Psi$ of $A$ is in general diverging since for instance the ``vacuum expectation value''
of $Abb^\ast$ is given by
$$
\ip{\vac}{Abb^\ast\vac}=\Tr(A_-),
$$
which is meaningless unless the operator $A_-$ fulfils a trace class condition, which is undesirable from the
perspective of physics.

The above problem has a brutal answer: Just subtract the possibly infinite trace term in (\ref{canquan}) and redefine the operators by the
so-called \emph{normal-ordering} process. Its prescription is given as follows: ``Each time a creator \emph{precedes} an annihilator,
reverse the order of the operators and factor by $-1$''. More concretely we set
\begin{equation}
\normord Abb^\ast\normord=-\sum_{k,\,l<0}\ip{e_k}{Ae_l}b^\ast(e_l)b(e_k)
\end{equation}
and give the following definition.

\begin{defn}
Let $A\in\B(\Hilb)$ with $\Hilb$ a polarized Hilbert space as above. Define the (formal) normal ordered operator
$\normord A\Psi^\ast\Psi\normord$ on $\F$ as
\begin{eqnarray}\label{NOrd}
\normord A\Psi^\ast\Psi\normord &=& Aa^\ast a+Aa^\ast b^\ast+ Aba+(\normord Abb^\ast\normord)\\
&=& A_{+}a^\ast a+A_{+-}a^\ast b^\ast+ A_{-\,+}ba-A_{-}b^\ast b
\end{eqnarray}
\end{defn}

For each $L\in\Z_{\geq 0}$ define projection operators $P_L$ on the fermionic Fock space $\F$ by
$$
P_L\left(\sum_{n,m\geq 0}\Psi^{(n,m)}\right)=\sum_{n+m\leq L}\Psi^{(n,m)}\qquad (\Psi^{(n,m)}\in\F^{(n,m)}\textrm{ for all }n,m).
$$

\begin{defn}
Define the subspace $\domain\subset\F$ of ``\emph{finite particle vectors}''  by
$$
\domain:=\set{\xi=\sum_{n,m\geq 0}\Psi^{(n,m)}\in\F\,\Big\vert\, \Psi^{(n,m)}\in\F^{(n,m)}\textrm{ and }\exists L\textrm{ such that }P_L(\xi)=\xi}.
$$
\end{defn}
Obviously $\domain$ is dense in $\F$ and it consists of physical states described as finite finite superpositions of states with bounded numer of particles
and anti-particles, hence the name. Notice also that
\begin{equation}
\domain=\bigoplus_{\substack{\mathrm{alg}\\n,m\geq 0}}\F^{(n,m)},
\end{equation}
the \emph{algebraic} direct sum of Hilbert spaces $\F^{(n,m)}$.

\begin{thm}
Let $A\in\ures(\Hilb,\Hilb_+)$ be arbitrary. Then the normal ordered product $\normord A\Psi^\ast\Psi\normord$ is
a well-defined operator on the dense domain $\domain\subset\F(\Hilb,\Hilb_+)$. In particular this domain
is independent of $A$. Furthermore, $\normord A\Psi^\ast\Psi\normord$ has a unique anti-selfadjoint extension $\Aa$
and $\exp(t\Aa)$ is unitary for all $t\in\Real$ and implements $U_t=\exp(tA)$, i.e. $\exp(t\Aa)$ is a representative of
$[\mathbb{U}_t]$.
\end{thm}

\begin{proof}[Sketch of proof]
Let $\varphi\in\domain$. Then $\varphi=P_L\varphi$ for large enough $L$, so that we may write the $k$-th power
\begin{equation}
\normord A\Psi^\ast\Psi\normord^k=\normord A\Psi^\ast\Psi\normord P_{L+2k-2}\cdots \normord A\Psi^\ast\Psi\normord P_L\varphi.
\end{equation}
One can show that there exists a bound
\begin{equation}
\norm{\normord A\Psi^\ast\Psi\normord P_L}\leq(L+2)\trinorm{A},
\end{equation}
where
$$
\trinorm{A}:=4\max\set{\norm{A_+},\norm{A_-}, \norm{A_{+-}}_2,\norm{A_{-+}}_2}.
$$
Thus
\begin{equation}
\norm{\normord A\Psi^\ast\Psi\normord^k\varphi}\leq\trinorm{A}^k(L+2k)\cdots(L+2)
\end{equation}
so that by the ratio test the power series
$
\sum_{k=0}^\infty\frac{\normord A\Psi^\ast\Psi\normord^k\varphi}{k!}z^k
$
converges for $\abs{z}<\delta:=(2\trinorm{A})^{-1}$. 

This proves that all elements of $\domain$ are analytic vectors for $\normord A\Psi^\ast\Psi\normord$.
Since the $\normord A\Psi^\ast\Psi\normord$ are anti-symmetric on $\domain$, Corollary \ref{NelsonCor}
(a consequence of the Nelson's analytic vector theorem) provides us with a unique
anti-selfadjoint extension $\Aa$ of $\normord A\Psi^\ast\Psi\normord$. The rest of the claim follows then
from the theorem of Stone--von Neumann and by direct calculations.
\end{proof}

\begin{ex}
Let $\one=\id_\Hilb:\Hilb\To\Hilb$ be the identity. Then one finds that
\begin{equation}
\normord\one\Psi^\ast\Psi\normord=\sum_{j\geq 0}a^\ast(e_j)a(e_j)-\sum_{k<0}b^\ast(e_k)b(e_k).
\end{equation}
Applying this to a vector $\psi^{(n,m)}\in\F^{(n,m)}=\bigwedge^n\Hilb_+\tensor\bigwedge^m(C\Hilb_-)$
gives
\begin{equation}
(\normord\Psi^\ast\Psi\normord)(\psi^{(n,m)})=(n-m)\cdot\psi^{(n,m)}.
\end{equation}
For this reason $\normord 1\Psi^\ast\Psi\normord$ is often called the \emph{charge operator} on $\F$ and is denoted by $Q$.
\end{ex}

It follows that we can write $\F=\oplus_{k\in\Z}\F_k$, where the \emph{charge-$k$ sector} is given as
\begin{equation}
\F_k:=\set{\psi\in\F\mid Q\psi=k\cdot\psi}=\bigoplus_{\substack{n,m\geq 0\\n-m=k}}\F^{(n,m)}.
\end{equation}

\begin{ex}\label{numberoperator}
Let $\pr_{\pm}:\Hilb\To\Hilb_{\pm}$ be the two projection operators associated to the polarization $\Hilb=\Hilb_+\oplus\Hilb_-$.
Then one computes that
\begin{equation}
(\normord(\pr_+-\pr_-)\Psi^\ast\Psi\normord)(\psi^{(n,m)})=(n+m)\cdot\psi^{(n,m)},
\end{equation}
i.e. $\normord(\pr_+-\pr_-)\Psi^\ast\Psi\normord$ is the \emph{number operator} $N$ on the fermionic Fock space $\F$.
\end{ex}
\begin{prop}
On the common dense domain $\D$ in $\F$, one has for elements $A,B\in\ures(\Hilb,\Hilb_+)$ with commutator
$C=[A,B]$ the following identity
\begin{equation}\label{basicrep}
[\Aa,\Bb]=[\noro{A},\noro{B}]=(\noro{C})+s(A,B)\cdot\one_\F,
\end{equation}
where $s(A,B)=\Tr(A_{-+}B_{+-}-B_{-+}A_{+-})$.
\end{prop}

Hence Proposition \ref{basicrep} gives us a representation of the \emph{central extension} $\Eures$ on
the fermionic Fock space $\F$. It can also be obtained by differentiating the Lie group
representation of the central extension $\EUres$ acting on $\F$ constructed in
\cite{PrSe} and called the \emph{basic representation}
or the \emph{fundamental representation}. It has the following basic properties.
\begin{prop}
\begin{enumerate}
\item All implementers $\mathbb{U}$ of a given $U\in\Ures(\Hilb,\Hilb_+)$ map
the charge-$k$-sector $\F_k$ to $\F_{k+q(U)}$, where $q(U)=\ind(U_+)\in\Z$.
\item The connected component of the identity element of $\EUres(\Hilb,\Hilb_+)$ acts irreducibly
on each $\F_k$, $k\in\Z$.
\item The representation of $\EUres(\Hilb,\Hilb_+)$ on $\F$ is irreducible.
\end{enumerate}
\end{prop}

We are also interested in how the closed Lie subgroup 
$H:=U(\Hilb_+)\times U(\Hilb_-)\subset\EUres(\Hilb,\Hilb_+)$ acts
on $\F$ via the basic representation. For this we first observe that if $\Hilb$ is a separable
complex Hilbert space and $k\in\N$, then the big unitary group
$U(\Hilb)$ acts irreducibly on $\bigwedge^k\Hilb$ by
$\bigwedge^k:U(\Hilb)\To\Aut(\bigwedge^k\Hilb)$,
\begin{equation}\label{wedgerep}
\quad\bigwedge^k(U)(v_1\wedge\ldots\wedge v_k)=(Uv_1)\wedge\ldots\wedge(Uv_k)
\end{equation}
for all $U\in U(\Hilb)$.
We also know that if $\pi_\alpha:G_\alpha\To U(\Hilb_\alpha)$ are irreducible unitary
representations of topological groups on Hilbert spaces $\Hilb_\alpha$ $(\alpha=1,2)$
then the tensor product representation 
\begin{equation}\label{tensorrep}
\pi_1\tensor\pi_2:G_1\times G_2\To U(\Hilb_1\tensor\Hilb_2)
\end{equation}
is irreducible as well (See for example \cite{Ga}, Theorem IV.3.15). We have thus arrived at:
\begin{lem}\label{nonprojlemma}
If $\Hilb_\alpha\, (\alpha=1,2)$ are separable complex Hilbert spaces and
$n,m\geq 0$ are integers, then the (non-projective) representation of
$U(\Hilb_1)\times U(\Hilb_2)$ on the tensor product $\bigwedge^n\Hilb_1\tensor\bigwedge^m\Hilb_2$
given by equations (\ref{wedgerep}) and (\ref{tensorrep}) is irreducible.
\end{lem}
Next, let $\lie{h}:=\Lie{H}=\lie{u}(\Hilb_+)\times\lie{u}(\Hilb_-)$ be the corresponding
Lie subalgebra of $\Eures(\Hilb,\Hilb_+)$.
\begin{cor}\label{restrictiontoh}
The restriction of the basic representation to a representation of the Lie subalgebra $\lie{h}\subset\Eures(\Hilb,\Hilb_+)$ on
$\F^{(n,m)}=\bigwedge^n\Hilb_+\tensor\bigwedge^m(C\Hilb_-)$ is the
derivative of the non-projective action of $H$ on $\F^{(n,m)}$ described in Lemma \ref{nonprojlemma}.
\end{cor}

\section{Basic differential geometry of the restricted Grassmannian manifold}
We follow in subsections \ref{SecBanHomSpa} -- \ref{SecKähStrRGr} again
closely \cite{HuWu}.
\subsection{Banach homogeneous spaces}\label{SecBanHomSpa}
Let $G$ be a Banach Lie group acting smoothly from the left on a Banach manifold $M$. Then
we can associate to each $X\in\lie{g}:=\Lie(G)=T_e G$ 
(with the commutator given by identifying $T_e G$ with the space of 
left-invariant vector fields on $G$) a 
\emph{fundamental vector field of the $G$-action} $\tau(X)$ by
\begin{equation}
\tau(X)_p:=\frac{d}{dt}\Big\vert_{t=0} (e^{tX}\cdot p)\quad\textrm{for all }p\in M.
\end{equation}
The map $\tau:\lie{g}\To\vf(M)$, where $\vf(M)$ denotes the space of vector fields on $M$,
is an anti-homomorphism of Lie algebras.

If the action of $G$ on $M$ is transitive the projection map 
$\pi:G\To G/H\isom M$, $g\mapsto gH$, where $H:=\Stab(p_0)$ is a stabilizer of a chosen
point $p_0\in M$, induces a projection $d\pi_e:\lie{g}=T_e G\To T_{eH}(G/H)$ which
we sometimes also denote by $\pi$ and which satisfies $\pi(X)=\tau(X)_{p_0}$ for all
$X\in\lie{g}$. Moreover, given a direct sum decomposition 
$\lie{g}=\lie{h}\oplus\lie{p}$, the restriction $\pi\vert_{\lie{p}}:\lie{p}\To\lie{g}/\lie{h}$
is an isomorphism.
\begin{defn}
A Banach manifold $M$ is called a \emph{Banach homogeneous space} if it is of the form $M=G/H$, where $G$ is a Banach-Lie group and $H$ is a
closed and connected Lie subgroup of $G$. 
\end{defn}

\subsection{The restricted Grassmannian as a homogeneous complex manifold}
\begin{defn}
The restricted Grassmannian $\RGr(\Hilb,\Hilb_+)$ is the set of all closed subspaces of $\Hilb$ such that
\begin{enumerate}
\item the orthogonal projection $\pr_+:W\To\Hilb_+$ is a Fredholm operator and
\item the orthogonal projection $\pr_-:W\To\Hilb_-$ is a Hilbert-Schmidt operator.
\end{enumerate}
\end{defn}

\begin{prop}\label{stabilizer}
\begin{enumerate}
\item The restricted Grassmannian $\RGr(\Hilb,\Hilb_+)$ is a Banach homogeneous space under
$\GLres(\Hilb,\Hilb_+)$ and $\Ures(\Hilb,\Hilb_+)$. The corresponding isotropy groups of the subspace $\Hilb_+\subset\Hilb$
are
$$
P=\set{
\left(
\begin{array}{cc}
a & b \\
c & d 
\end{array} \right)
\in\GLres(\Hilb,\Hilb_+)\mid c=0
}
$$
and
$$
H=\set{
\left(
\begin{array}{cc}
a & b \\
c & d 
\end{array} \right)
\in\Ures(\Hilb,\Hilb_+)\mid b=0\textrm{ and } c=0 
}\isom U(\Hilb_+)\times U(\Hilb_-),
$$
respectively.
\item The restricted Grassmannian $\RGr(\Hilb,\Hilb_+)$ is a complex analytic manifold modelled on the separable Hilbert
space $\Lp^2(\Hilb_+,\Hilb_-)$.
\item The actions of $\GLres(\Hilb,\Hilb_+)$ and $\Ures(\Hilb,\Hilb_+)$ on $\RGr(\Hilb,\Hilb_+)$ are 
complex analytic and real analytic, respectively.
\item The connected components of the restricted Grassmannian manifold $\RGr$ are given by the sets of subspaces in $\Hilb$ having
\emph{virtual dimension} $k$ $(k\in\Z)$:
$$
\RGr^k:=\set{W\in\RGr\mid\ind(\pr_+:W\To\Hilb_+)=k}.
$$
\end{enumerate}
\end{prop}
\subsection{The isotropy representation of $U(\Hilb_+)\times U(\Hilb_-)$}
\subsubsection{Isotropy representation}
We now take a closer look at the homogeneous space $\RGr(\Hilb,\Hilb_+)\isom G/H$, where
$G=\Ures(\Hilb,\Hilb_+)$ and $H=U(\Hilb_+)\times U(\Hilb_-)$. The latter group $H$ is known
to be connected and contractible. 

Let $\Ad H:H\To\Aut(\ures)$ be the usual adjoint representation of the Lie subgroup $H\subset\Ures(\Hilb,\Hilb_+)$ on the Lie algebra $\ures(\Hilb,\Hilb_+)$.
Set
\begin{equation}
\lie{h}=\Lie(H)=\set{\Mat{\alpha}{0}{0}{\beta}\in\B(\Hilb)\Big\vert\, 
\alpha^\ast=-\alpha,\,\beta^\ast=-\beta}\subset\ures,
\end{equation}
\begin{equation}\label{p}
\lie{p}=
\set{\Mat{0}{-\gamma^\ast}{\gamma}{0}\in\B(\Hilb)\,\Big\vert\,\gamma\textrm{ is Hilbert--Schmidt}}\subset\ures.
\end{equation}
Then $\lie{p}$ is known to be an $\Ad H$-invariant complement of $\lie{h}$ in $\ures$, i.e.
$\ures=\lie{h}\oplus\lie{p}$ and this decomposition respects the adjoint action $\Ad H$ of $H$.

Explicitly the isotropy representation $\Ad H:H\To\Aut(\lie{p})$  is given by
\begin{eqnarray}\label{explisotropy}
 \Ad\Mat{a}{0}{0}{b}\cdot\Mat{0}{-\gamma^\ast}{\gamma}{0} &=&\Mat{a}{0}{0}{b}\Mat{0}{-\gamma^\ast}{\gamma}{0}\Mat{a}{0}{0}{b}^{-1}\nonumber\\
&=&\Mat{0}{-a\circ\gamma^\ast\circ b^{-1}}{b\circ\gamma\circ a^{-1}}{0}.
\end{eqnarray}

\begin{cor}\label{AdH}
The induced Lie algebra representation $\ad\lie{h}:\lie{h}\To\End(\lie{p})$ of the isotropy representation
$\Ad:H\To\Aut(\lie{p})$ is given by
\begin{equation}
h\mapsto [h,\cdot]\quad\textrm{ for all } h\in\lie{h}.
\end{equation}
\end{cor}

\begin{rem}
Recall that if $G$ is a Banach-Lie group with Lie algebra $\lie{g}$ and 
$$
\phi:G\To\Aut(V)
$$ 
is a continuous representation of $G$ on a Banach space $V$, then
the induced Lie algebra representation is given by
\begin{equation}\label{indlierep}
\lie{g}\times V^\infty\To V^\infty,\quad (X,v)\mapsto  d\phi(X)\cdot v,
\end{equation}
where $V^\infty\subset V$ is the subspace of \emph{smooth vectors}, i.e. the vector space $V^\infty$ consists of
those vectors $v\in V$ such that the orbit mapping  $G\To V,\, g\mapsto\phi(g)\cdot v$ is smooth.

If one chooses for the representation $\phi:G\To\Aut(V)$ the adjoint representation
$\Ad:G\To\Aut(\lie{g})$ then it is known that that $\Ad$ is a representation of $G$ on $\lie{g}$ with a \emph{smooth} action
map \cite{Neeb2} p. 167. It follows from this that the subspace of smooth vectors $\lie{g}^\infty=\lie{g}$.
\end{rem}

\subsection{$\Ad H$-invariant metric on $\lie{p}$}
Recall that by mapping  $\gamma\To\Mat{0}{-\gamma^\ast}{\gamma}{0}$ with $\gamma\in\Lp^2(\Hilb_+,\Hilb_-)$,
the vector space $\lie{p}$ was canonically identified with the Hilbert space $\Lp^2(\Hilb_+,\Hilb_-)$
of complex linear Hilbert-Schmidt operators from $\Hilb_+$ to $\Hilb_-$. The isotropy
representation of $H$ on $\lie{p}$ is given under this isomorphism by
\begin{equation}\label{isotropy}
\Ad\Mat{a}{0}{0}{b}\cdot\gamma=b\circ\gamma\circ a^{-1}\quad\textrm{for all }
\gamma\in\Lp^2(\Hilb_+,\Hilb_-)
\end{equation}
from which we see that
\begin{equation}
g^{\lie{p}}(\gamma,\sigma):=2\Re\Tr_{\Hilb_+}(\gamma^\ast\sigma)\quad\textrm{for all }
\gamma,\sigma\in\Lp^2(\Hilb_+,\Hilb_-)
\end{equation}
determines an $\Ad(H)$-invariant metric on $\lie{p}$. Here $\gamma^\ast$ denotes the
adjoint of $\gamma$ as a complex linear operator from $\Hilb_+$ to $\Hilb_-$. It follows then
that the product $\gamma^\ast\sigma$, as a product of two Hilbert-Schmidt operators,
is in $\Lp^1(\Hilb_+)$ so that the trace 
$\Tr(\gamma^\ast\sigma)$ is indeed well-defined.

\subsection{Kähler structure on $\RGr$}\label{SecKähStrRGr}

Recall that at the level of Lie algebras we had $\Lie(\EGLres)\isom\glres\oplus\C$ with the Lie bracket given by
\begin{equation}
\Big[(A,\alpha),(B,\beta)\Big]=\Big([A,B],s(A,B)\Big),
\end{equation}
where the Lie algebra two-cocycle $s$ was given by the Schwinger term.

Next define the following ``new Schwinger term'' $\tilde{s}$ by
\begin{equation}
\tilde{s}(A,B)=-s(A,B)=\Tr(A_{+-}B_{-+}-B_{+-}A_{-+}),
\end{equation}
for $A,B\in\ures(\Hilb,\Hilb_+)$.
\begin{defn}
Let $\hat{\Omega}_{\Hilb_+}$ be the real-valued antisymmetric bilinear form on $\ures$ defined by
\begin{equation}
\hat{\Omega}_{\Hilb_+}(A,B):=(-\im)\tilde{s}(A,B).
\end{equation}
\end{defn}
\begin{lem}
The bilinear form $\hat{\Omega}_{\Hilb_+}$ on $\ures(\Hilb,\Hilb_+)$ vanishes on the isotropy subalgebra $\mathfrak{u}(\Hilb_+)\times\mathfrak{u}(\Hilb_-)$
and is invariant under the linear isotropy representation of $U(\Hilb_+)\times U(\Hilb_-)$.
\end{lem}
\begin{cor}
The bilinear form $\hat{\Omega}_{\Hilb_+}$ on $\ures(\Hilb,\Hilb_+)$ descends to a form $\Omega_{\Hilb_+}$ on
$$
\ures(\Hilb,\Hilb_+)/(\mathfrak{u}(\Hilb_+)\times\mathfrak{u}(\Hilb_-))\isom\Lp^2(\Hilb_+,\Hilb_-)\isom T_{\Hilb_+}\RGr
$$
that is invariant under the action of $U(\Hilb_+)\times U(\Hilb_-)$:
\begin{equation}
\Omega_{\Hilb_+}(\gamma,\delta)=2\Im\Tr(\gamma^\ast\delta)\quad\textrm{ for all }\gamma,\delta\in\Lp^2(\Hilb_+,\Hilb_-).
\end{equation}
\end{cor}
\begin{defn}
Let $J_{\Hilb_+}$ be the $U(\Hilb_+)\times U(\Hilb_-)$-invariant complex structure on $T_{\Hilb_+}\RGr\isom\Lp^2(\Hilb_+,\Hilb_-)$ defined by
\begin{equation}\label{cs}
J_{\Hilb_+}\gamma=\im\gamma\quad\textrm{ for all }\gamma\in\Lp^2(\Hilb_+,\Hilb_-).
\end{equation}
\end{defn}
We supply $\Lp^2(\Hilb_+,\Hilb_-)$ with a strongly non-degenerate bilinear form
\begin{equation}
g_{\Hilb_+}(\gamma,\delta):=\Omega_{\Hilb_+}(\gamma,J_{\Hilb_+}\delta)=2\Re\Tr(\gamma^\ast\delta)
\end{equation}
and with a strongly non-degenerate sesquilinear form
\begin{equation}
h_{\Hilb_+}(\gamma,\delta):=g_{\Hilb_+}(\gamma,\delta)+\im\Omega_{\Hilb_+}(\gamma,\delta)=2\Tr(\gamma^\ast\delta)
\end{equation}
inducing a Riemannian and a Hermitean metric on $\RGr(\Hilb,\Hilb_+)$, respectively.

\begin{prop}
The real-valued antisymmetric bilinear form $\Omega_{\Hilb_+}$ is a strongly non-de\-ge\-ne\-rate symplectic form on $\RGr$ and moreover the
quadruple $(\Omega,J,h,g)$ gives $\RGr$ the structure of a Kähler manifold which is homogeneous under its Kähler isometries.
\end{prop}

\subsection{The determinant line bundle on $\RGr$}
Let $\Hilb=\Hilb_+\oplus\Hilb_-$ be a complex separable polarized Hilbert space and
$\RGr$ the associated restricted Grassmannian manifold.
 It is proved in \cite{PrSe} that there exists a so-called
 \emph{determinant line bundle} $\DET\To\RGr$, which is a certain
 holomorphic complex line bundle with a holomorphic fibrewise linear action
 of $\EUres$ covering the transitive action of $\Ures$ on the base manifold
 $\RGr$.

\begin{prop}
The line bundle $\DET\To\RGr$ has a natural $\EUres$-invariant hermitean structure and its first Chern class is represented by $\big(\frac{-1}{2\pi}\big)$ times the
Kähler form on $\RGr$:
$$
[c_1(\DET)]=-\Big(\frac{1}{2\pi}\Big)[\Omega]\in H^2(\RGr,\Z).
$$
\end{prop}
\subsection{Holomorphic sections of $\DET^\ast$}\label{secholsecofdet}
First recall that (continuous, smooth, holomorphic,\ldots) sections of the dual $E^\ast$ of a finite dimensional complex vector bundle
$E\stackrel{\pi}{\To} M$ can be identified with functions on the total space $E$ that restrict to linear functionals on each fibre $E_x=\pi^{-1}(x)$.

We are going to consider the infinite dimensional setting in which $M$ is a Fr\'echet manifold and the typical fibre of $E$ is a Banach space. The fibres of $E^\ast$ are
then always given by the strong duals of the fibres of $E$. We denote by $\hG(M,E^\ast)$ the vector space of holomorphic sections of the dual vector bundle
$E^\ast$ and by $\mathcal{O}(E)$ the vector space of holomorphic functions $E\To\C$ and set 
$$
\Olin(E):=\{f\in\mathcal{O}(E)\mid f\big\vert_{E_x}:E_x\To\C\textrm{ is } \C\textrm{-linear for all }x\in M\}.
$$
The spaces $\hG(M,E^\ast)$ and $\Olin(E)$ have natural structures of locally convex topological vector spaces given by the
uniform convergence of all derivatives on compact subsets of $M$ and $E$, respectively.

Define a map
\begin{equation}
\hG(M,E^\ast)\To\Olin(E),\,\sigma\mapsto f_\sigma
\end{equation}
by $f_\sigma(\ell)=\sigma(\pi(\ell))(\ell)$ for all $\ell\in E$. Then we have the following analogue with the finite dimensional case.
\begin{lem}
Let $E\stackrel{\pi}{\To} M$ be a holomorphic Banach space bundle over a Fr\'echet manifold $M$. Then the map $\sigma\mapsto f_\sigma$ defined above is a continuous
linear isomorphism.
\end{lem}
\begin{prop}
Let $M$ be the restricted Grassmannian $\RGr$ and $E=\DET\stackrel{\pi}{\To}\RGr$. Then there exists a continuous linear $\EGLres$-equivariant injection
$\F^\ast\stackrel{r}{\To}\hG(\RGr,\DET^\ast)$, where $\F$ is the fermionic Fock space.
\end{prop}
\begin{rem}\label{detsec}
Pressley and Segal show in \cite{PrSe} that the space $\F^\ast$ is actually dense inside the space 
$\hG(\RGr,\DET^\ast)$ and gives an irreducible representation of $\EUres$, the so called
\emph{basic representation}. 
Moreover, Pickrell \cite{Pi1,Pi2} constructed
a natural Gaussian measure on $\RGr$ and showed that an element $s\in\hG(\RGr,\DET^\ast)$ is in the image $r(\F^\ast)$ if and only if $s$ is square-integrable with respect to this measure.
\end{rem}
\subsection{The Pfaffian line bundle}
Let $\Hilb_\Real$ be a \emph{real} Hilbert space with inner product $B(\cdot,\cdot)$. 
We denote by $\J(\Hilb_\Real)$ the set of \emph{complex structures} of
$\Hilb_\Real$, i.e.
$$
\J(\Hilb_\Real):=\set{J\in\mathcal{L}(\Hilb_\Real,\Hilb_\Real)\mid B(Ju,Jv)=B(u,v)\textrm{ for all }u,v\in\Hilb_\Real,\,J^2=-1}.
$$
Any complex structure $J\in\J(\Hilb_\Real)$ turns 
the complexification $\Hilb_\C=\Hilb_\Real\tensor_\Real\C$ of $\Hilb_\Real$ into a \emph{complex} polarized Hilbert space, which we denote by $\Hilb_J$. The polarization
$\Hilb_J=W\oplus\overline{W}$ is given by the $\pm\im$ eigenspaces $W$ and $\overline{W}$ of $J^\C:\Hilb_\C\To\Hilb_\C$, respectively.
\begin{defn}
Fix a complex structure $J\in\J(\Hilb_\Real)$. The associated \emph{restricted isotropic Grassmannian} is the set
$$
\Jres(\Hilb_\Real,J):=\set{J'\in\J(\Hilb_\Real)\mid J-J'\in\Lp^2(\Hilb_\Real)}.
$$
\end{defn}
The group $O(\Hilb_\Real)$ of \emph{real orthogonal operators} on $\Hilb_\Real$ consists of all invertible $\Real$-linear maps $\Hilb_\Real\To\Hilb_\Real$ that preserve the bilinear form $B$. It has the following subgroup.
\begin{defn}
Let $J\in\J(\Hilb_\Real)$ be a (fixed) complex structure and set
$$
\Ores(\Hilb_\Real,J):=\set{O\in O(\Hilb_\Real)\mid [O,J]\in\Lp^2(\Hilb_\Real)},
$$
which is topologized by the norm topology combined with the Hilbert-Schmidt norm (as we did with the group $\GLres$), so that it becomes a Banach Lie group.
\end{defn}

There is a transitive action of $\Ores(\Hilb_\Real,J)$ on $\Jres(\Hilb_\Real,J)$ given by $J'\mapsto OJ'O^{-1}$. Its stabilizer at the base point $J\in\Jres(\Hilb_\Real,J)$ is $\Ures(\Hilb_J)$,
realizing $\Jres(\Hilb_\Real,J)$ as a Banach manifold with
$$
\Jres(\Hilb_\Real,J)\isom\Ores(\Hilb_\Real,J)/\Ures(\Hilb_J,\Eig(J^\C,+\im)).
$$

There is a second description of the isotropic Grassmannian that can be obtained by complexifying $\Hilb_\Real$ to a complex Hilbert space $\Hilb_\C$ with \emph{hermitean} inner
product $\ip{\cdot}{\cdot}$, which we now describe. Extend the inner product $B$ on $\Hilb_\Real$ to a complex \emph{bilinear form} $B_\C$ on the complexification $\Hilb_\C$; then
$B_\C(v_1,v_2)=2\ip{\bar{v_1}}{v_2}$ for all $v_1,v_2\in\Hilb_\C$. Here the bar denotes
the canonical conjucation on $\Hilb_\C=\Hilb_\Real\tensor_\Real\C$ given by
$$
\overline{v\tensor\lambda}=v\tensor\bar{\lambda}\quad\textrm{ for all }
v\in\Hilb_\Real,\,\lambda\in\C.
$$

Any complex structure on $\Hilb_\Real$ induces a decomposition
$\Hilb_\C=W\oplus\overline{W}$, where $W\subset\Hilb_\C$ is a maximal isotropic subspace of $\Hilb_\C$ with respect to $B_\C$. Here isotropic means that $B_\C(w_1,w_2)=0$ for
all $w_1,w_2\in W$. It follows from this that
$$
\Jres(\Hilb_\Real,J)=\set{W\in\RGr(\Hilb_\C,\Eig(J^\C,+\im))\mid W\textrm{ isotropic and } W\oplus\overline{W}=\Hilb_\C}
$$
and that we have an $\Ores(\Hilb_\Real,J)$-equivariant embedding
\begin{equation}
i:\Jres(\Hilb_\Real,J)\hookrightarrow\RGr(\Hilb_\C,\Eig(J^\C,+\im)).
\end{equation}

There is a holomorphic hermitean line bundle $\Pf\To\Jres(\Hilb_\Real,J)$ called the \emph{Pfaffian} which is a holomorphic
square root of the determinant line bundle when restricted to $\Jres(\Hilb_\Real)\subset\RGr(\Hilb_\C)$, i.e. 
\begin{equation}
i^\ast\DET\isom\Pf^{\tensor 2}.
\end{equation}

\subsection{Infinite dimensional spinor representation}
The group $O(\Hilb_\C)$ 
of \emph{complex orthogonal operators} on $\Hilb_\C$  with respect to 
$B_\C$ consists of all 
bounded invertible $\C$-linear maps $\Hilb_\C\To\Hilb_\C$ preserving
the bilinear form $B_\C$, i.e.
\begin{equation}
O(\Hilb_\C):=\set{O\in GL(\Hilb_\C)\mid B_\C(Ou,Ov)=B_\C(u,v)\textrm{ for all }u,v\in\Hilb_\C}.
\end{equation}
\begin{defn}
Let $J\in\J(\Hilb_\Real)$ be a (fixed) complex structure and set
$$
\Ores(\Hilb_\C,J):=\set{O\in O(\Hilb_\C)\mid [O,J]\in\Lp^2(\Hilb_\C)}.
$$
\end{defn}
In other words $\Ores(\Hilb_\C,J)=O(\Hilb_\C)\cap\GLres(\Hilb_\C)$. We give $\Ores(\Hilb_\C,J)$ the topology induced from $\GLres(\Hilb_\C)$. This makes
$\Ores(\Hilb_\C,J)$ into a complex Banach Lie group. 
Moreover $\Ores(\Hilb_\C,J)\cap O(\Hilb_\Real)=\Ores(\Hilb_\Real,J)$, where
$$
\Ores(\Hilb_\Real,J):=\set{O\in O(\Hilb_\Real)\mid [O,J]\in\Lp^2(\Hilb_\Real)}.
$$

\begin{thm}[\cite{B,PrSe}] 
Let $J\in\J(\Hilb_\Real)$, $\Hilb_J=W\oplus\overline{W}$ be the corresponding polarization of the complexified Hilbert space, 
and let $\,\spib$ be the GNS representation space 
$$
\spib:=\Hilb^{\CAR(W)}_J\isom\F_+(\Hilb_J,W)=\bigoplus_{n\geq 0}\bigwedge^n W
$$ 
of the CAR algebra $\CAR(W)$. Then
in the category of complex Banach Lie groups, there
exists a central extension
$$
1\To\C^\ast\To\EOres(\Hilb_\C,J)\To \Ores(\Hilb_\C,J)\To 1
$$
such that $\EOres(\Hilb_\C,J)$ 
acts linearly on the Hilbert space $\spib$
implementing the Bogoliubov transformations $\alpha_O(a(f))\in\Aut(\CAR(W))$,
$$
\alpha_O(a(f)):=a(O\cdot f)\qquad\qquad (f\in W)
$$
for all $O\in \Ores(\Hilb_\C,J)$.
\end{thm}
\begin{rem}
You should compare the above exact sequence to the finite dimensional analogue
$$
1\To U(1)\To\Spinc(n)\To SO(n,\Real)\To 1.
$$
\end{rem}
\begin{defn}
The above linear action of $\EOres(\Hilb_\C,J)$ on $\spib$ is called the $\Spinc$-representation of 
$\Ores(\Hilb_\C,J)$.
\end{defn}
The Hilbert space $\spib$ gives an (irreducible) \emph{projective} 
unitary representation of $\Ores(\Hilb_\C,J)$.
\subsubsection{Realizing the infinite-dimensioanal spin representation via the Pfaffian line bundle}
In the same vein that we realized in section \S\ref{secholsecofdet} the fermionic Fock space representation of $\EUres$ inside
the holomorphic sections $\hG(\RGr,\DET^\ast)$ of the (dual) \emph{determinant}
line bundle, it is possible to realize the infinite-dimensional spin representation
of $\EOres(\Hilb_\C,J)$ via the \emph{Pfaffian} line bundle:
\begin{thm}\label{spinrep}
\begin{enumerate}
\item The group $\EOres(\Hilb_\C,J)$ acts equivariantly on the Pfaffian line bundle $\Pf\To\Jres(\Hilb_\Real,J)$,
\item The dual Hilbert space $\spib^\ast$ is canonically realized inside the holomorphic section module of $\Pf^\ast$, i.e. 
$\spib^\ast\subset\hG(\Jres(\Hilb_\Real),\Pf^\ast)$, where the latter space is equipped with the natural structure of a locally convex topological vector space the same
way as in the case of the (dual) determinant line bundle. Moreover, this inclusion
is $\EOres(\Hilb_\C,J)$-equivariant.
\end{enumerate}
\end{thm}
Thus, the above theorem shows that we have a Borel-Weil type description of the infinite dimensional spinor module.
\subsection{Spinor representation of the Lie algebra $\ores$}
We follow here closely \cite{Mic}.

Let $J\in\J(\Hilb_\Real)$ be a chosen complex structure on a real Hilbert space,
$\Hilb_J=W\oplus\overline{W}$ be the corresponding 
polarization of the complexified Hilbert space, 
and let $\,\spib\isom\F_+(\Hilb_J,W)=\F(W)$ be the GNS representation space of the CAR-algebra
$\CAR(W)$.

Let $\{e_1,e_2,\ldots\}$ be a fixed orthonormal basis of $W$ and set
\begin{equation}
a_i^\ast:=a^\ast(e_i),\quad a_i:=a(e_i).
\end{equation}
The element $1\in\F(W)$ is the vacuum vector $\vac_\spib$ characterized by
$a_i\vac_\spib=0$ for all $i$. Moreover, the only nonzero anticommutation relations are
\begin{equation}
a_i a_j^\ast+a_j^\ast a_i=\delta_{ij}.
\end{equation}

\subsubsection{The Lie algebra $\ores$}
The Lie algebra 
$$
\ores(\Hilb_\C,J):=\Lie(\Ores(\Hilb_\C,J))
$$ 
consists of all $\C$-linear operators 
that with respect to the polarization $\Hilb_J=W\oplus\overline{W}$ can be written
in the block-matrix form
\begin{equation}
x=\Mat{a}{b}{c}{d}
\end{equation}
such that $b,c\in\Lp^2$ and the transposes of the blocks satisfy
\begin{equation}
c^t=-c,\quad b^t=-b,\quad d=-a^t.
\end{equation}

Again, using methods of second quantization we obtain a representation $T$ of
the \emph{central extension} $\Eores(\Hilb_\C,J)$ of $\ores(\Hilb_\C,J)$ on 
the Fock space $\F(W)$ by defining
\begin{equation}
T(x)=
\sum_{\substack{i,j\in\Z\\i,j>0}}d_{ij}a_i^\ast a_j+\frac{1}{2} b_{ij} a_i a_j+\frac{1}{2} c_{ij} a_i^\ast a_j^\ast
\quad\textrm{ for all } x\in\ores(\Hilb_\C,J).
\end{equation}

The Lie algebra $2$-cocycle of $\ores(\Hilb_\C,J)$ 
defining
$\Eores(\Hilb_\C,J)$ can be read from the commutators
\begin{equation}
[T(x),T(x')]=T([x,x'])+\frac{1}{2}\Tr(cb'-c'b).
\end{equation}
Hence, except for the factor $1/2$ we obtain the same cocycle as the
Lie algebra cocycle of $\glres$ determining the central extension $\Eglres$.

Moreover, the representation $T$ can be exponentiated to give a projective representation
of the connected component of the identity $S\Ores(\Hilb_\C,J)\subset\Ores(\Hilb_\C,J)$.

\subsection{Infinite dimensional real Clifford algebra}
The infinite dimensional spinor module $\spib$ is related to the representation theory of the
infinite dimensional \emph{real Clifford algebra} $\Cliff(\Hilb_\Real,B)$, 
which is the universal real unital Banach algebra
generated by 
all $\Real$-linear maps
$f\mapsto \gamma(f),\,f\in\Hilb_\Real$, subject to
\begin{equation}
\gamma(f)\gamma(g)+\gamma(g)\gamma(f)=2B(f,g)\cdot\mathbf{1}
\end{equation}
Any complex structure $J$ of $\Hilb_\Real$ 
induces a polarization $\Hilb_J=W\oplus\overline{W}$ on the complexification 
$\Hilb_\C=\Hilb_\Real\tensor\C$ and
determines a representation of $\Cliff(\Hilb_\Real)$ on
the GNS representation space $\spib=\Hilb^{\CAR(W)}_J$ by first mapping
$\gamma(f)\mapsto a(f)+a(f)^\ast\in\CAR(W)$, and then using 
the natural action of the CAR algebra $\CAR(W)$ on $\Hilb^{\CAR(W)}_J$.

Let 
$$
\chi:\EOres(\Hilb_\Real,J)\To\Aut(\spib)
$$
denote for the spin representation 
and let
$$
\gamma:\Hilb_\Real\To\End(\spib)
$$
be the Clifford multiplication introduced above. Then these two maps are compatible in 
the following sense:
for every $O\in\Ores(\Hilb_\Real,J)$ there exists
$\tilde{O}\in\EOres(\Hilb_\Real,J)$ such that
\begin{equation}
\gamma(O\cdot f)=\chi(\tilde{O})\gamma(f)\chi(\tilde{O})^{-1} 
\end{equation}
for all $f\in\Hilb_\Real$.

\subsection{Infinite dimensional complexified Clifford algebra}
Let $(\Hilb_\Real,B)$ be a real Hilbert space and $\Hilb_\C=\Hilb_\Real\tensor_\Real\C$ its complexification, endowed with
the $\C$-linear extension $B_\C$ of the real inner product $B$ and the canonical
hermitean structure $\ip{\cdot}{\cdot}$. The complex Clifford algebra $\Cliff(\Hilb_\C,B_\C)$ is then by
definition
$$
\Cliff(\Hilb_\C,B_\C):=\Cliff(\Hilb_\Real,B)\tensor_\Real\C.
$$
It turns out to be a unital $C^\ast$-algebra isomorphic to a CAR-algebra $\CAR(W)$, where $W$ is any $B_\C$-isotropic
subspace of $\Hilb_\C$ such that $\bar{W}=W^\perp$.
\subsection{Spin representation of $\Cliff(\lie{p})$}
Let $\RGr(\Hilb,\Hilb_+)$ be the restricted Grassmannian manifold associated to a given
 polarization $\Hilb=\Hilb_+\oplus\Hilb_-$ of a separable complex Hilbert space $\Hilb$. We write
 $\RGr$ as a Banach homogeneous space $\RGr\isom G/H$, where $G=\Ures(\Hilb,\Hilb_+)$ and
 $H=U(\Hilb_+)\times U(\Hilb_-)$. We also set $\lie{g}:=\Lie(G)$ and
 $\lie{h}:=\Lie(H)$ so that $\lie{g}/\lie{h}\isom\lie{p}$, where $\lie{p}$ is 
 the $\Ad$-invariant complement of 
$$
\lie{h}=\set{\Mat{\alpha}{0}{0}{\beta}\in\B(\Hilb)\,\Big\vert\, \alpha^\ast=-\alpha,\,\beta^\ast=-\beta}
$$ 
given by
$$
\lie{p}=\set{\Mat{0}{-\gamma^\ast}{\gamma}{0}\,\Big\vert\, \gamma
\textrm{ is Hilbert-Schmidt}}.
$$
 
We start by considering the \emph{real} Hilbert space (i.e. our coefficients are in $\Real$)
$$
\Hilb_{\Real,\lie{p}}:=\Lp^2(\Hilb_+,\Hilb_-)\isom T_{\Hilb_+}\RGr\isom\lie{g}/\lie{h}\isom\lie{p}
$$
equipped with the real inner product 
$$
B(\gamma,\sigma):=g_{\Hilb_+}(\gamma,\sigma)=2\Re\Tr(\gamma^\ast\sigma).
$$ 

The real Hilbert space
$\Hilb_{\Real,\lie{p}}$ has a complex structure $J=J_{\Hilb_+}$ defined by equation (\ref{cs}). 
This induces a polarization $\Hilb_{\C,\lie{p}}=W\oplus\overline{W}$ on the 
\emph{complexified} Hilbert
space $\Hilb_{\C,\lie{p}}:=\Hilb_{\Real,\lie{p}}\tensor_\Real\C$ equipped
with the complexified bilinear (non-Hermitean) form
\begin{equation}
B_\C(\gamma,\sigma)=\ip{\gamma}{\overline{\sigma}}_{\lie{p},\C},
\end{equation}
where 
\begin{equation}\label{CorHermForm}
\ip{\gamma}{\sigma}_{\lie{p},\C}=h_{\Hilb_+}(\gamma,\sigma)=2\Tr(\gamma^\ast\sigma)
\end{equation}
for all $\gamma,\sigma\in\Hilb_{\C,\lie{p}}$ (Hermitean form corresponding to the real inner
product
$\ip{\cdot}{\cdot}_{\lie{p}}$).
Here the bar notation refers to conjugation. 

Hence 
by the general theory of the last section we
obtain a representation of the infinite dimensional real Clifford algebra $\Cliff(\lie{p})$ on
the corresponding spinor module $\spib_{\lie{p}}:=\Hilb_{\C,\lie{p}}^{\CAR(W)}$. Notice that $\spib_{\lie{p}}$ is also the representation
space for the infinite dimensional (projective) $\Spinc$-representation of $\Ores(\Hilb_{\C,\lie{p}},J_{\Hilb_+})$.

\subsection{Lifting the infinite-dimensional isotropy representation}
\subsubsection{Lifting $\Ad H$}
Recall that in the finite dimensional case $M=G/H$, with 
$G$ a compact Lie group and $M$ 
a spin manifold, the Hilbert space of $L^2$-spinors could
be interpreted as $L^2$-maps $s:G\To\spib_{\lie{p}}$ satisfying the $H$-equivariance condition
$s(gh)=(\Adt h^{-1})s(g)$ for all $g\in G$ and $h\in H$, where $\Adt:H\To\Spin(\lie{p})$
is the spin lift of the isotropy representation $\Ad:H\To\SO(\lie{p})$.

Another
way to see spinors on $G/H$ was to use the Peter-Weyl theorem to decompose
\begin{equation}\label{eltwo}
L^2(G/H,S)\isom L^2(G\times_H\spib_{\lie{p}})\isom
\widehat{\bigoplus}_\lambda V_\lambda\tensor(V_\lambda^\ast\tensor\spib_{\lie{p}})^H,
\end{equation}
where $S\To G/H$ is the spinor bundle with fibre $\spib_{\lie{p}}$, $H$ 
acts on $\spib_{\lie{p}}$ via the spin lift $\Adt:H\To\Spin(\lie{p})$ and where we
sum over all irreducible representations $V_\lambda$ of $G$.

In order to generalize this
to the infinite dimensional case of 
$$
\RGr(\Hilb,\Hilb_+)\isom\Ures(\Hilb,\Hilb_+)/(U(\Hilb_+)\times U(\Hilb_-))
$$
we first need to define a relevant lift 
\begin{equation}
\Adt:H\To\EOres(\Hilb_{\C,\lie{p}},J_{\Hilb_+}) 
\end{equation}
for the isotropy representation $\Ad$ in (\ref{isotropy}).
\begin{lem}
Let $M=\RGr(\Hilb,\Hilb_+)\isom G/H$, where $G=\Ures(\Hilb,\Hilb_+)$ and 
$H=U(\Hilb_+)\times U(\Hilb_-)$. Denote by $\lie{p}\isom\lie{g}/\lie{h}$ the $\Ad(H)$-invariant
complement of $\lie{h}$ in $\lie{g}$ given by equation (\ref{p}). Then the
isotropy representation $\Ad$ of $H$ in $\lie{p}$ given by equation (\ref{isotropy})
takes values in 
$\Ores(\Hilb_{\Real,\lie{p}},J_{\Hilb_+})\subset\Ores(\Hilb_{\C,\lie{p}},J_{\Hilb_+})$.
\end{lem}
\begin{proof}
First recall that the inner product of the real Hilbert space 
$$
\Hilb_{\Real,\lie{p}}:=\Lp^2(\Hilb_+,\Hilb_-)\isom\lie{p}
$$ 
was given by
$B(\gamma,\sigma)=g_{\Hilb_+}(\gamma,\sigma)=2\Re\Tr(\gamma^\ast\sigma)$.
Moreover $\Hilb_{\Real,\lie{p}}$ had a natural complex structure 
$J_{\Hilb_+}$, defined so that 
$J_{\Hilb_+}\gamma=\im\gamma$ for all $\gamma\in\Lp^2(\Hilb_+,\Hilb_-)$,
that definines a polarization $\Hilb_{\Real,\lie{p}}\tensor_\Real\C=W\oplus\overline{W}$
on the complexified Hilbert space.

We will now show that for each $h\in H=U(\Hilb_+)\times U(\Hilb_-)$ the linear mapping
$\Ad h$ on $\lie{p}\isom\Lp^2(\Hilb_+,\Hilb_-)$ given by
$$
\Ad\Mat{a}{0}{0}{b}\cdot\gamma=b\circ\gamma\circ a^{-1},
$$
where $h=\Mat{a}{0}{0}{b}\in H=U(\Hilb_+)\times U(\Hilb_-)\subset\Ures(\Hilb,\Hilb_+)$,
$\gamma\in\Lp^2(\Hilb_+,\Hilb_-)$
and where we have used the standard block-matrix description of the linear operator 
$h\in\Ures(\Hilb,\Hilb_+)$, is in $\Ores(\Hilb_{\C,\lie{p}})$.

\emph{Step 1}. $\Ad h\in GL(\Hilb_{\Real,\lie{p}})$: 
The inner product $\ip{\cdot}{\cdot}_{\lie{p}}:=B(\cdot\, , \,\cdot)$ induces a norm $\norm{\cdot}_{\Real,\lie{p}}$
on $\Hilb_{\Real,\lie{p}}$,
\begin{equation}
\norm{\gamma}_{\Real,\lie{p}}:=\sqrt{\ip{\gamma}{\gamma}_{\lie{p}}}=\sqrt{B(\gamma,\gamma)}=\sqrt{2\Re\Tr(\gamma^\ast\gamma)}
\end{equation}
for all $\gamma\in\Lp^2(\Hilb_+,\Hilb_-)$. Hence  by definition an operator $T:\Hilb_{\Real,\lie{p}}\To\Hilb_{\Real,\lie{p}}$ is bounded if 
\begin{equation}
\sup\set{\norm{T\gamma}_{\Real,\lie{p}}\mid\gamma\in\Lp^2(\Hilb_+,\Hilb_-),\norm{\gamma}_{\Real,\lie{p}}\leq 1}<\infty
\end{equation}
that is if and only if
\begin{equation}\label{bounded}
\sup\set{B(T\gamma,T\gamma)\mid\gamma\in\Lp^2(\Hilb_+,\Hilb_-),B(\gamma,\gamma)\leq 1}.
\end{equation}
On the other hand we know that $B(\gamma,\sigma)=g^{\lie{p}}(\gamma,\sigma)$ defines an $\Ad(H)$-invariant
metric on $\lie{p}$, i.e. 
\begin{equation}\label{trivbounded}
B\Big(\big(\Ad h\big)(\gamma),\big(\Ad h\big)(\sigma)\Big)=B(\gamma,\sigma)
\end{equation}
for all $\gamma,\sigma\in\Lp^2(\Hilb_+,\Hilb_-)\isom\lie{p}$ and for all linear mappings
$\Ad h$ of $\lie{p}$ where $h\in H$. 
Thus choosing $T=\Ad h$ in (\ref{bounded}), we see from equation  (\ref{trivbounded}) that $\Ad h$ is trivially in
$\B(\Hilb_{\Real,\lie{p}})$ for every $h\in H$.
Obviously $\Ad h$ has a bounded inverse $T_h$ given by
$T_h:\gamma\mapsto b^{-1}\circ\gamma\circ a$, where
$h=\Mat{a}{0}{0}{b}\in H$, showing that $\Ad h\in GL(\Hilb_{\Real,\lie{p}})$ for all $h\in H$.

\emph{Step 2.} $\Ad h\in O(\Hilb_{\Real,\lie{p}})$: Knowing that 
$\Ad h\in GL(\Hilb_{\Real,\lie{p}})$
we may once again use equation (\ref{trivbounded}) to conclude that
each
$\Ad h$ is in the real orthogonal group $O(\Hilb_{\Real,\lie{p}}):=O(\Hilb_{\Real,\lie{p}},B(\cdot\, , \,\cdot))$.

\emph{Step 3.}
$\Ad h\in\Ores(\Hilb_{\Real,\lie{p}},J_{\Hilb_+})\subset O(\Hilb_{\Real,\lie{p}})\subset GL(\Hilb_{\Real,\lie{p}})$:
We need to show that for all $h\in H$ the commutator 
$$
[\Ad h,J_{\Hilb_+}]=\Ad h\circ J_{\Hilb_+}-J_{\Hilb_+}\circ\Ad h\in \Lp^2(\Hilb_{\Real,\lie{p}}).
$$
Now since for every $h\in H$ the image $\Ad h$ is a $\C$-linear map 
$$
\Ad h:\Hilb_{\Real,\lie{p}}=\Lp^2(\Hilb_+,\Hilb_-)\To\Lp^2(\Hilb_+,\Hilb_-)=\Hilb_{\Real,\lie{p}}
$$ and on the other
hand $J_{\Hilb_+}\gamma=\im\gamma$, it is clear that the commutator is in 
$\Lp^2(\Hilb_{\Real,\lie{p}})$ (in fact, it is equal to zero).
Hence $\Ad h\in\Ores(\Hilb_{\Real,\lie{p}},J_{\Hilb_+})$ for every $h\in H$ and
since $\Ores(\Hilb_\C,J)\cap O(\Hilb_\Real)=\Ores(\Hilb_{\Real},J)$, it follows
that the map $\Ad h\in\Ores(\Hilb_{\C,\lie{p}}, J_{\Hilb_+}$) for all $h\in H$.
\end{proof}

Next, notice the following. If $\set{e_k}_{k\in\N}$ is an orthonormal basis for a complex Hilbert space $\Hilb$ with inner product
$\ip{\cdot}{\cdot}$ and let $u_k$ and $v_k$ denote $e_k$ and $\im e_k$, respectively, then
$\set{u_k,v_k}_{k\in\N}$ form an orthonormal system with respect to $\tau(\cdot\, , \, \cdot)=\Re\ip{\cdot}{\cdot}$. We call
the real span of $\set{u_k,v_k}_{k\in\N}$ for the \emph{real Hilbert space}, $\Hilb_\Real$, of $\Hilb$, with inner product
$\tau(\cdot\, , \, \cdot)$. Notice that $\Hilb_\Real$ and $\Hilb$ represent the same set.

The following result appears in \cite{PrSe}, \S 12.4.
\begin{lem}\label{conjugatelemma}
Let $\Hilb$ be a complex Hilbert space regarded as a real Hilbert space $\Hilb_\Real$. Then the complexification 
$(\Hilb_\Real)_\C=\Hilb_\Real\tensor_\Real \C$ is canonically
isomorphic to $\Hilb\oplus\overline{\Hilb}$. Here $\overline{\Hilb}$ denotes the abstract complex conjugate to
$\Hilb$. It is a copy of $\Hilb$ with the scalars acting in a conjugate way: $\lambda\cdot\bar{\xi}:=(\bar{\lambda}\cdot\xi)^{-}$
for all $\lambda\in\C$ and $\xi\in\Hilb$.
\end{lem}
\begin{proof}
The isomorphism is given by the map
$$
\lambda\tensor\xi\mapsto \lambda\xi\oplus\bar{\lambda}\xi
$$
for all $\lambda\in\C$ and $\xi\in\Hilb_\Real$.
\end{proof}

\begin{prop}\label{isotropylift}
The isotropy representation $\Ad: H=U(\Hilb_+)\times U(\Hilb_-)\To
\Ores(\Hilb_{\C,\lie{p}},J_{\Hilb_+})$ admits a spin lift
$\Adt: H\To\EOres(\Hilb_{\C,\lie{p}},J_{\Hilb_+})$, i.e. there is a commuting
diagram of Lie group homomorphisms
$$
\xymatrix{
            & \EOres\ar[d]^p\\
H\ar[r]^{\Ad}\ar@{-->}[ur]^{\Adt} & \Ores
}
$$
where $p$ is the projection map in the exact sequence
$$
1\To U(1)\To\EOres\stackrel{p}{\To}\Ores\To 1.
$$
\end{prop}
\begin{proof}
The proof is a modification of the idea of lifting an embedding 
of the Lie group $H=U(\Hilb_+)\times U(\Hilb_-)$
in $\Ores$ to $\EOres$
presented in §6.4., \cite{Mic}. The essential idea there was that the $U(1)$ central extension
$\EOres\To\Ores$ becomes trivial when restricted to the subgroup $H$ so that there exists a local
section $s:H\To\EOres$ which we can then use to construct the lift.

Recall that the real Hilbert space 
$\Hilb_{\Real,\lie{p}}\isom\lie{p}\isom\lie{g}/\lie{h}$ 
consisted of all ($\Real$-linear combinations of all) $\C$-linear operators
$\Hilb_+\To\Hilb_-$ that are Hilbert-Schmidt and where the inner product was given by
$B(\gamma,\sigma)=g_{\Hilb_+}(\gamma,\sigma)=2\Re\Tr(\gamma^\ast\sigma)$
for all $\gamma,\sigma\in\Hilb_{\Real}$. 
Using Lemma \ref{conjugatelemma} we may identify the complexification
$\Hilb_{\C,\lie{p}}=\Hilb_{\Real,{\lie{p}}}\tensor_\Real\C$ with the direct
sum of \emph{complex} Hilbert spaces
$(\Hilb_{\Real,\lie{p}},\ip{\cdot}{\cdot}_{\lie{p},\C})$ and its conjugate
$(\overline{\Hilb}_{\Real,\lie{p}},\ip{\cdot}{\cdot}_{\lie{p},\C})$. 
Here
the Hermitean inner product is given by
$\ip{\gamma}{\sigma}_{\lie{p},\C}=h_{\Hilb_+}(\gamma,\sigma)=2\Tr(\gamma^\ast\sigma)$.
On the other hand,
recalling that for an operator $T\in\mathcal{L}^2(\Hilb_+,\Hilb_-)$ its adjoint
$T^\ast$ satisfies 
$$
(\alpha T)^\ast=\bar{\alpha}T^\ast\quad\textrm{ for all } \alpha\in\C,
$$
one sees that the conjugate space $\overline{\mathcal{L}^2(\Hilb_+,\Hilb_-)}$ satisfies
$$
\overline{\mathcal{L}^2(\Hilb_+,\Hilb_-)}\isom\mathcal{L}^2(\Hilb_-,\Hilb_+)
$$
so that under these isomorphism the conjugation in
$$
\Hilb_{\C,\lie{p}}\isom \mathcal{L}^2(\Hilb_+,\Hilb_-)\oplus \mathcal{L}^2(\Hilb_-,\Hilb_+)
$$
is given explicitly by taking the adjoint.

The corresponding complex bilinear (non-Hermitean) form in $\Hilb_{\C,\lie{p}}$ is then given
by
$$
B_\C(\gamma,\sigma)=\ip{\gamma}{\overline{\sigma}}_{\lie{p},\C}
\quad\textrm{ for all } \gamma,\sigma\in
\Hilb_{\C,\lie{p}}.
$$

The Hilbert space $\Hilb_{\Real,\lie{p}}$ had
a natural complex structure $J_{\Hilb_+}:\Hilb_\C\To\Hilb_\C$ defined by
$J_{\Hilb_+}\gamma=\im\cdot\gamma$. This complex structure gives us a decomposition
$\Hilb_{\C,\lie{p}}=W\oplus\overline{W}$, where $W=\Hilb_{\Real,\lie{p}}$ and
$\overline{W}=\im\cdot\Hilb_{\Real,\lie{p}}$. Using this decomposition we can
write operators $A\in\Ores(\Hilb_{\C,\lie{p}},J_{\Hilb_+})$ in the block
matrix form
$$
A=\Mat{a}{b}{c}{d},
$$
where $a\in\Fred(W),\,d\in\Fred({\overline{W}}),b\in\Lp^2(\overline{W},W)$ and
$c\in\Lp^2(W,\overline{W})$. 
Since in the block matrix description $H$ is
a diagonal subgroup of $\Ures$
it follows that the image $\Ad(H)$ consists of all block matrices
of the form
\begin{equation}\label{imageH}
C=\Mat{\Ad_W h}{0}{0}{\Ad_{\overline{W}} h},
\end{equation}
where $h=\Mat{u}{0}{0}{v}\in H=U(\Hilb_+)\times U(\Hilb_-)$.

Now let $\Sk(W)$ be the Hilbert space of all skew Hilbert-Schmidt operators
$S:W\To\overline{W}$, i.e. those $\Lp^2$ operators $S:W\To\overline{W}$ satisfying
$$
B_\C(S\gamma,\sigma)=-B_\C(S\sigma,\gamma).
$$
The fibre $p^{-1}(g)\subset\EOres(\Hilb_{\C,\lie{p}})$ over a point 
$g=\Mat{a}{b}{c}{d}\in\Ores(\Hilb_{\C,\lie{p}})$ consists of all holomorphic functions
$f:\Sk(W)\To\C$ that are proportional to 
$$
S\mapsto\Pf((a+bS_0)^{-1}(a+bS)),
$$
where $S_0\in\Sk(W)$ is any element such that $a+bS_0$ is invertible \cite{Mic}. The product
in $\EOres(\Hilb_{\C,\lie{p}})$ is given by
\begin{equation}\label{prodores}
(g,f)(g',f'):=(g'',f''),
\end{equation}
where $g''=gg'$ and $f''$ is the function defined by
\begin{equation}\label{prodorf}
f''(S):=f(g'\cdot S)f'(S),\qquad g'\cdot S:=(c'+d'S)(a'+b'S)^{-1}.
\end{equation}
Choose
$g=\Mat{\Ad_W h}{0}{0}{\Ad_{\overline{W}} h}\in\Ores(\Hilb_{\C,\lie{p}})$, 
$g'=\Mat{\Ad_W h'}{0}{0}{\Ad_{\overline{W}} h'}\in\Ores(\Hilb_{\C,\lie{p}})$ and
$f=f'\equiv 1$, the constant function with value $1\in\C$, in equation
(\ref{prodores}). Now since $g'\cdot S\in\Sk(W)$ is just a shift of the 
argument $S\in\Sk(W)$ and on the
other hand $f$ was chosen to be a constant function with value $1$ it follows from
(\ref{prodorf}) that the product $(g,f)(g',f)$ defined by equation (\ref{prodores}) is again
of the same form as its factors, i.e. $(g,1)(g',1)=(gg',1)$. Hence the
map $\Ores(\Hilb_{\C,\lie{p}})\To\EOres(\Hilb_{\C,\lie{p}})$ defined
by $g\mapsto (g,1)$ is a group homomorphism, which is clearly
smooth and therefore a Lie group homomorphism.

Since the projection map of
the principal $U(1)$-bundle 
$$
p:\EOres(\Hilb_{\C,\lie{p}})\To\Ores(\Hilb_{\C,\lie{p}})
$$ 
sends the pair $(g,f)$ to $g\in\Ores(\Hilb_{\C,\lie{p}})$ 
it follows from the above that the 
composition of $\Ad$ (a Lie group homomorphism) with the Lie group
homomorphism $g\mapsto(g,1)$ defines the required lift
$\Adt$ of $\Ad$.
\end{proof}

\subsection{The induced lift of $\ad$}\label{inducedad}
Set $\EOres(\lie{p}^\C,W):=\EOres(\Hilb_{\C,\lie{p}},J_{\Hilb_+})$. Then
according to the proof of the Proposition \ref{isotropylift} the lift 
$\Adt:H\To\EOres(\lie{p}^\C,W)$ of the isotropy representation
$\Ad:H\To\Ores(\lie{p}^\C,W)$ is given by
the composition
\begin{equation}
\Adt=s\circ\Ad,
\end{equation}
where $s$ is 
the Lie group homomorphism
mapping $g\To(g,1)$. It follows from 
 Lemma \ref{AdH} that 
the induced Lie algebra map
$\adt:\lie{h}\To\Eores(\lie{p}^\C,W)$, 
$$
\adt=d(\Adt)=ds\circ d(\Ad),
$$ 
is given by
\begin{equation}\label{adlift}
X\mapsto([X,\cdot],0)\quad\textrm{ for all } X\in\lie{h}.
\end{equation}
Considering this result, then by abuse of notation, we shall often write $\ad$ instead of $\adt$ when we
talk about the spin lift.
\section{The algebraic tensor product $U(\Eglres)\tensor\Cliff(\lie{p^\C})$ of algebras}
\subsection{Natural set of generators for $\lie{p}^\C$}
Let $\set{e_k\mid k\in\Z\setminus\set{0}}$ be a Hilbert basis of $\Hilb$ such that $\Hilb_+$ respectively
$\Hilb_-$ is generated by $\set{e_k\mid k>0}$ respectively $\set{e_k\mid k<0}$
(notice that we neglect the index $k=0$; this is done in order to be able to
write our Dirac operator in a convenient form).
Define the complex linear rank one operators $E_{p,q}$ on $\Hilb$ for
$p,q\in\Z$ as in \cite{SpeWu} by setting
\begin{equation}
E_{p,q}(v):=\ip{e_q}{v}_\Hilb\cdot e_p.
\end{equation}
As a matrix each $E_{p,q}$ corresponds to the basic matrix, whose matrix element
at position $(p,q)$ is equal to one and all the other matrix coefficients are set to
zero.
It follows that the adjoints satisfy $(E_{p,q})^\ast=E_{q,p}$ and
\begin{eqnarray}
B(E_{-k,\,l},E_{-r,\,s}) &=& B(iE_{-k,\,l},\im E_{-r,\,s})=2\cdot\delta_{k,r}\cdot\delta_{l,s},\\
B(E_{-k,\,l},\im E_{-r,\,s})&=& 0. 
\end{eqnarray}

Recall that the complexification of $\ures$ was given by the Lie algebra $\glres$  so that the complexified Hilbert space of 
$$
\lie{p}=\set{\Mat{0}{-\gamma^\ast}{\gamma}{0}\in\lie{u}(\Hilb)\,\Big\vert\,\gamma\textrm{ is Hilbert-Schmidt}}
$$ 
can be given explicitly by
$$
\lie{p}^\C=\set{\Mat{0}{\sigma}{\gamma}{0}\in\B(\Hilb)\,\Big\vert\,\gamma,\sigma\textrm{ are Hilbert-Schmidt}}
$$
so that if we set $X_{p,q}:=\frac{1}{\sqrt{2}}E_{p,q}$, the union of the sets
\begin{equation}
\mathscr{B}_1:=\set{X_{-p,\,q},\,\Big\vert\, p>0,q>0}
\end{equation}
and
\begin{equation}
\mathscr{B}_2:=\set{X_{p,\,-q}\,\Big\vert\,p>0, q>0}
\end{equation}
form an orthonormal $\C$ -Hilbert bases of $\lie{p}^\C$ satisfying
$X_{p,q}^\ast=X_{q,p}$ with respect to the Hermitean form
(\ref{CorHermForm}). 
Note that introducing such a basis for $\lie{p}^\C$ is made possible by the fact that
finite rank operators are dense in the Hilbert-Schmidt operators with
respect to the norm induced by the Hilbert-Schmidt inner product $\ip{\cdot}{\cdot}_{HS}$.

The above two sets generate the
Hilbert subspaces
$$
\langle\mathscr{B}_1\rangle=W,\quad\langle\mathscr{B}_2\rangle=\overline{W}
$$
with 
$$
W\isom\lie{p},\quad\overline{W}\isom\overline{\lie{p}}\quad W\oplus\overline{W}=\lie{p}^\C,
$$
i.e. we obtain the polarization of $\lie{p}^\C$ introduced in the previous section.

Similarly the complexification of the Banach space
$$
\lie{h}=\set{\Mat{\alpha}{0}{0}{\beta}\in\lie{u}(\Hilb)\,\Big\vert\,\alpha=-\alpha^\ast,\beta=-\beta^\ast}
$$
can be given by
$$
\lie{h}^\C=\set{\Mat{\alpha}{0}{0}{\beta}\in\B(\Hilb)}.
$$
But the situation is now different from the case of $\lie{p}^\C$: Here $\alpha$ and $\beta$ can be any \emph{bounded} operators.
Since the closure of the set of finite rank operators $\F\mathcal{R}(\lie{h}^\C)$ in the space of bounded linear operators (with respect to the norm topology) 
$\B(\lie{h}^\C)$ is
exactly the space of compact operators $\mathcal{C}(\lie{h}^\C)\subsetneqq\B(\lie{h}^\C)$, the
space generated by the vectors $X_{p,q}\in\lie{h}^\C\,(pq>0)$ can \emph{not} be any bigger than the space
of compact operators, even if one considers converging infinite linear combinations of those $X_{p,q}$! This observation suggests that if we want to 
use arguments relying only on a generating set, we should replace $\lie{h}^\C$ with a smaller Lie algebra.
\begin{defn}
Let $\hinf^\C\subset\lie{h}^\C$ be the Lie subalgebra consisting of all \emph{finite} $\C$-linear combinations of
the vectors the vectors $X_{p,q}\in\lie{h}^\C\,(pq>0)$, i.e.
$$
\hinf^\C:=\spa\set{X_{p,q}\mid p,q\in\Z,\, pq>0}.
$$
\end{defn}

\subsection{The algebraic tensor product $U(\Eglres)\tensor\Cliff(\lie{p^\C})$}
\subsubsection{Explicit description of the operators $\rh(E_{ij})$
for the basic representation of $\Eglres$}

Let $\Hilb=\Hilb_+\oplus\Hilb_-$ be a complex separated polarized Hilbert space, $\set{u_k}_{k\in\N}$
an orthonormal basis for $\Hilb_+$ and $\set{v_k}_{k\in-\N}$ an orthonormal  basis for $\Hilb_-$, where we always
assume that $0\notin\N$. 

We consider the CAR-algebra $\CAR(\Hilb,\Hilb_+)$ and introduce the
operators acting on the fermionic Fock space $\F=\F(\Hilb,\Hilb_+)$,
\begin{equation}
A_k:=a(u_k)=\Psi(u_k),\quad A_k^\ast:=a^\ast(u_k)=\Psi^\ast(u_k)
\end{equation}
and
\begin{equation}
B_k:=b(v_k)=\Psi^\ast(v_k),\quad B^\ast_k:=b^\ast(v_k)=\Psi(v_k).
\end{equation}
These satisfy the anticommutation relations
\begin{eqnarray}
\{A_k^\ast,A_{k'}\}&=&\sigma_{kk'}=\{B_k^\ast,B_{k'}\},\label{CAR1}\\
\{A_k,A_{k'}\}&=&\{B_k,B_{k'}\}=\{A_k,B_{k'}\} \textrm{ etc. }=0.\label{CAR2}
\end{eqnarray}
Moreover
\begin{equation}\label{basicvac}
A_k\vac_\F=\Psi(u_k)\vac_\F=0,\quad B_k\vac_\F=\Psi^\ast(v_k)\vac_\F=0.
\end{equation}
and it is known that the vectors
\begin{equation}
A_{k_1}^\ast\cdots A_{k_\alpha}^\ast B_{l_1}^\ast\cdots B^\ast_{l_\beta}\vac_\F
\end{equation}
give an orthonormal Hilbert basis of $\F$.

To simplify notation one may define
\begin{equation}
\psi_k=
\left\{
\begin{array}{ll}
\Psi(u_k), & \textrm{$k>0$}\\
\Psi(v_k), & \textrm{$k<0$}
\end{array}
\right.
\quad\psi_k^\ast=
\left\{
\begin{array}{ll}
\Psi^\ast(u_k), & \textrm{$k>0$}\\
\Psi^\ast(v_k), & \textrm{$k<0$}.
\end{array}
\right.
\end{equation}
These satisfy
\begin{equation}\label{posnegCAR}
\psi_i\psi_j^\ast+\psi_j^\ast\psi_i=\delta_{ij},\quad \psi_i\psi_j+\psi_j\psi_i=0,\quad \psi_i^\ast\psi_j^\ast+\psi_j^\ast\psi_i^\ast=0,
\end{equation}
$$
\psi_k\vac=0 \textrm{ for positive $k$},\quad \psi_k^\ast\vac=0\textrm{ for negative $k$}
$$
and the vectors
$$
\psi_{k_1}^\ast\cdots\psi_{k_\alpha}^\ast\psi_{l_1}\cdots\psi_{l_b}\vac_\F\quad (k_i>0\, \forall i=1,\ldots,\alpha,\, l_j<0\,\forall j=1,\ldots,\beta)
$$
give an orthonormal basis of $\F$.

Recall that for $A\in\glres(\Hilb,\Hilb_+)$ we defined the normally ordered operator
$\normord A\Psi\Psi^\ast\normord$ acting on the fermionic Fock space $\F(\Hilb,\Hilb_+)$ so that
$$
\normord A\Psi\Psi^\ast\normord=A_+a^\ast a+ A_{+-}a^\ast b^\ast+A_{-+}ba-A_{-}b^\ast b.
$$
Hence for the basic matrices $E_{p,q}\in\glres\,(p,q\in\Z\setminus\set{0})$
\begin{equation}\label{basicmat}
\rh(E_{p,q})=
\left\{ \begin{array}{ll}
\psi_p^\ast\psi_q=A_p^\ast A_q, & \textrm{if $p>0,q>0$}\\
\psi_p^\ast\psi_q=B_p A_q, & \textrm{if $p<0,q>0$}\\
\psi_p^\ast\psi_q=A_p^\ast B_q^\ast, & \textrm{if $p>0, q<0$}\\
-\psi_q\psi_p^\ast=-B_{q}^\ast B_p& \textrm{if $p<0,q<0$}.
\end{array} \right.
\end{equation}
Taking into account (\ref{posnegCAR}) this becomes
\begin{equation}\label{explicit rh}
\rh(E_{p,q})=
\left\{ \begin{array}{ll}
\psi^\ast_p\psi_q-1 &\textrm{when $p=q<0$}\\
\psi^\ast_p\psi_q & \textrm{otherwise}.
\end{array}\right.
\end{equation}

From this explicit form of the basic representation one deduces the following.
\begin{prop}\label{basicrepannihilate}
For the basic representation $\rh:\Eglres\To\End(\F)$ the 
images $\rh(E_{p,q})$ of the basic matrices $E_{p,q}\in\glres\,(p,q\neq 0)$ annihilate the vacuum $\vac_\F\in\F$ except
for the case when $p>0$ and $q<0$.
\end{prop}
\subsubsection{Complexified infinite dimensional Clifford algebra}\label{complexified Clifford}
We shall consider the complex spi\-nor module
$\spib_\lie{p}^c:=\F(\lie{p}^\C,B_\C,W)$ as a module for the infinite-dimensional 
\emph{complexified} Clifford algebra 
$$
\Cliff(\lie{p}^\C)=\Cliff(W\oplus\overline{W})
$$
defined by the complexified symmetric bilinear form 
$B_\C(\gamma,\sigma)=\ip{\gamma}{\overline{\sigma}}_{\lie{p},\C}$ on $\lie{p}^\C$. 
Here the conjugation in the
Clifford algebra is given on the orthonormal basis elements of 
the Hilbert space $\lie{p}^{\C}$ by
\begin{equation}\label{conjugation}
\overline{X}_{p,q}:=X_{p,q}^\ast=X_{q,p},
\end{equation} 
and as mentioned
\begin{equation}
\ip{X_{ij}}{X_{nm}}=\Tr(E_{ij}^\ast E_{nm})=\Tr(E_{ji}E_{nm})=\Tr(\delta_{in}E_{jm})=\delta_{in}\delta_{jm}.
\end{equation}

The complexified Clifford algebra $\Cliff(\lie{p}^\C)$ is generated by elements of $W$ and $\overline{W}$ subject to
the anticommutation relation
\begin{equation}
\gamma(\overline{w}_1)\gamma(w_2)+\gamma(w_2)\gamma(\overline{w}_1)=B_\C(\overline{w}_1,w_2)
\end{equation}
for $\overline{w}_1\in\overline{W}$ and $w_2\in W$.
In concrete terms the Clifford algebra $\Cliff(\lie{p}^\C)$ is
described by the anti-commutators
\begin{equation}\label{Cliffp}
\{\gamma(X_{ij}),\gamma(X_{mn})\}=B_\C(X_{ij},X_{mn})=\ip{X_{ij}}{\overline{X}_{mn}}=\ip{X_{ij}}{X_{nm}}
=\delta_{in}\delta_{jm},
\end{equation}
where $ij<0$ and $mn<0$.

Since  $\spib_\lie{p}^c=\F(\lie{p}^\C,W)$ is a fermionic Fock space 
associated to a polarized Hilbert space $\lie{p^\C}=W\oplus\overline{W}$ we would like to
represent $\Cliff(\lie{p}^\C,B_\C)$ on it via the representation of the CAR-algebra
$\CAR(\lie{p}^\C,B_\C,W)$.
\begin{lem}
For $ij<0$ the map
\begin{equation}
\gamma(X_{pq})\mapsto
\left\{ \begin{array}{ll}
\Psi(X_{pq}) & \textrm{if $p<0,q>0\iff X_{pq}\in W$}\\
\Psi^\ast (X_{pq}) & \textrm{if $p>0,q<0\iff X_{pq}\in\overline{W}$}
\end{array}\right.
\end{equation}
gives a representation of $\Cliff(\lie{p}^\C)$ on $\F(\lie{p}^\C,W)$.
\end{lem}
\begin{proof}
For $X_{ij}\in W$ and $X_{mn}\in\overline{W}$ the anticommutator
$$
\{\gamma(X_{ij}),\gamma(X_{mn})\}\mapsto\set{\Psi(X_{ij}),\Psi^\ast(X_{mn})}=B_\C(X_{ij},X_{mn})=\ip{X_{ij}}{X_{nm}}
=\delta_{in}\delta_{jm}.
$$
For the other possible choices of $X_{ij}$ and $X_{mn}$ (e.g. $X_{ij},X_{mn}\in W$) the anticommutator
$$
\{\gamma(X_{ij}),\gamma(X_{mn})\}\mapsto 0
$$
because of the CAR-algebra relations.
\end{proof}

\begin{cor} \label{Cliffordvacannihilated}
The vacuum vector $\vac_{\spib}\in\spib_{\lie{p}}^c\isom\F(\lie{p}^\C,W)$ associated to the CAR algebra representation of
$\CAR(\lie{p}^\C,W)$ is annihilated by $\gamma(X_{pq})\in\Cliff(\lie{p}^\C)$ when $X_{pq}\in W$, i.e. when $p<0,q>0$.
\end{cor}

Next we want to have a look at the tensor product $\F\tensor\spib_{\lie{p}}^c$ considered as a module for the algebraic tensor product $U(\Eglres)\tensor\Cliff(\lie{p}^\C)$
where $\Eglres$ acts on $\F$ via the basic representation $\rh$ and 
the complexified Clifford algebra
$\Cliff(\lie{p}^\C)$ acts on the complexifed spinor module $\spib_{\lie{p}}^c$
via the CAR algebra representation of $\CAR(\lie{p}^\C,W)$ so that in particular
$$
\Big[\rh(E_{pq})\tensor\gamma(E_{rs})\Big](\psi\tensor s)=\rh(E_{pq})\psi\tensor\gamma(E_{rs})s\qquad (\psi\in\F,s\in\spib_{\lie{p}}^c).
$$

Combining Proposition \ref{basicrepannihilate} with Corollary \ref{Cliffordvacannihilated} yields the following.
\begin{prop}\label{vacuumstr}
Let $\vac:=\vac_\F\tensor\vac_\spib\in\F\tensor\spib_{\lie{p}}^c$ be the vacuum vector for the tensor product.
Then $\vac$ is annihilated by $E_{pq}\tensor 1$ unless $p>0$ and $q<0$ and it is annihilated by
$1\tensor\gamma(E_{pq})$ when $p<0$ and $q>0$.
\end{prop}

\section{Realizing the isotropy representation via CAR algebra}
\subsection{The finite-dimensional analogue}
Here we follow \cite{HuaPan} section \S 2.1.8.
\subsubsection{Chevalley map}
Let $V$ be an $n$-dimensional real vector space with an inner product $\ip{\cdot}{\cdot}$.
Let $T(V)$ be the tensor algebra over $V$ and consider the ideal $I$ in $T(V)$ generated
by all $v\tensor v+\ip{v}{v}$ for $v\in V$. Then the quotient
algebra
$$
\Cliff(V)=T(V)/I
$$
is the Clifford algebra of $V$. We can choose an orthonormal basis $Z_i$ of $V$ with respect to
$\ip{\cdot}{\cdot}$ as a set of generators for $\Cliff(V)$. The relations then become
$$
Z_i Z_j=-Z_j Z_i,\quad i\neq j;\quad Z_i^2=-1.
$$
It is then clear that the set
$$
\set{Z_{i_1}Z_{i_2}\cdots Z_{i_k}\mid 1\leq i_1<i_2<\cdots<i_k\leq n=\dim V}
$$
(together with the element $1$ regarded as an ``empty product'') spans $\Cliff(V)$.

The canonical projection $T(V)\To\bigwedge(V)$ is known to have a linear right inverse given by
linearly embedding $\bigwedge(V)$ into the tensor algebra $T(V)$ as the skew-symmetric
tensors:
$$
v_1\wedge\cdots\wedge v_k\mapsto\frac{1}{k!}\sum_{\sigma\in S_k} \sign(\sigma)
v_{\sigma(1)}\tensor\cdots\tensor v_{\sigma(k)},
$$
where $S_k$ denotes the permutation group of $k$ letters. The
Chevalley map $j$ is obtained by composing this skew symmetrization map 
$\bigwedge(V)\To T(V)$
with
the canonical projection $T(V)\To T(V)/I=\Cliff(V)$.
Using an orthonormal basis $Z_i$ of $V$, this map is determined
on the corresponding basis of $\bigwedge(V)$ simply by
$$
Z_{i_1}\wedge\cdots\wedge Z_{i_k}\mapsto Z_{i_1}\cdots Z_{i_k}
:=\gamma(Z_{i_1})\cdots \gamma(Z_{i_k})
\quad (\textrm{ and } 1\mapsto 1),
$$
where $1\leq i_1<i_2<\cdots<i_k\leq n$. The Chevalley map is known to be an isomorphism and
it is often referred to as the Chevalley identification. We will say
that the elements of $\Cliff(V)$ which are in the image of $\bigwedge^k(V)$ under $j$
are of pure degree $k$.
\subsubsection{Embedding $\lie{so}(V)$ into $\Cliff(V)$}\label{EmbedSoToCliff}
We will consider the image $j(\bigwedge^2(V))\subset\Cliff(V)$ under the Chevalley map.
This space is a Lie subalgebra of $\Cliff(V)$, if we consider $\Cliff(V)$ to be a Lie
algebra in the usual way, by setting $[a,b]=ab-ba$ for all $a,b\in\Cliff(V)$.
On the other hand, $\bigwedge^2(V)$ is linearly isomorphic to the Lie
algebra $\lie{so}(V)$, with $Z_i Z_j$ corresponding to the operator
with matrix $E_{ij}-E_{ji}$ in basis $Z_i$. Hence the correspondece
$$
E_{ij}-E_{ji}\longleftrightarrow -\frac{1}{2}Z_i Z_j=-\frac{1}{2}\gamma(Z_i)\gamma(Z_j)
$$
is an isomorphism of the Lie algebras $\lie{so}(V)$ and $j(\bigwedge^2(V))\subset\Cliff(V)$.
\subsubsection{Realizing the isotropy representation via Clifford algebra}
We start with a definition.
\begin{defn}
A \emph{quadratic Lie algebra} is a Lie algebra $\lie{g}$ with a nondegenerate invariant symmetric bilinear form $B$. A \emph{quadratic
subalgebra} of $\lie{g}$ is a Lie subalgebra $\lie{h}\subset\lie{g}$ such that the restriction of $B$ to
$\lie{h}\times\lie{h}$ is nondegenerate.
\end{defn}
In this section we will be interested in cases where both $\lie{g}$ and $\lie{h}$ are both reductive and complex.
\begin{ex}
Every complex semisimple Lie algebra is quadratic; one can choose for $B$ the Killing form.
\end{ex}
\begin{ex}
If $\lie{g}$ is (complex) reductive, then it is always quadratic. This is because $\lie{g}$ is a direct sum 
$\lie{g}=Z_{\lie{g}}\oplus[\lie{g},\lie{g}]$ where $Z_{\lie{g}}$ is the center of $\lie{g}$
and where the commutator ideal $[\lie{g},\lie{g}]$ is semisimple.
Now in the semisimple
part $[\lie{g},\lie{g}]$ one can choose the Killing form and then extend this over the center by any nongenerate symmetric bilinear form.
\end{ex}
\begin{ex}
Not all reductive subalgebras need to be quadratic subalgebras. For example $\C X$ with $X$ nilpotent is not a quadratic
subalgebra of $\lie{g}$, but it is abelian and hence reductive.
\end{ex}

If $\lie{h}$ is a quadratic subalgebra of $\lie{g}$ then
$$
\lie{g}=\lie{h}\oplus\lie{p}
$$
where $\lie{p}=\lie{h}^\perp$ is the orthogonal complement of $\lie{h}$ with respect to $B$. Moreover the restriction
of $B$ to $\lie{p}\times\lie{p}$ and the invariance of $B$ implies that
$$
[\lie{h},\lie{p}]\subset\lie{p}.
$$

By the invariance of $B$, the adjoint action of $\lie{h}$ on $\lie{p}$ defines a Lie algebra homomorphism
$\ad:\lie{h}\to\so{p}$. Now composing $\ad$ with the embedding $\so{p}\hookrightarrow\Cliff(\lie{p})$ constructed
in \S\ref{EmbedSoToCliff}, one obtains a Lie algebra map
\begin{equation}\label{KostantAlphaKostant}
\alpha:\lie{h}\To\Cliff(\lie{p}).
\end{equation}

\begin{rem}
Kostant denotes this map by $v_\ast$ in his papers.
\end{rem}

If we choose an orthonormal basis $X_i$ for $\lie{p}$, then the embedding $\so{p}\hookrightarrow\Cliff(\lie{p})$ is given
explicitly by
$$
E_{ij}-E_{ji}\mapsto-\frac{1}{2}\gamma(X_i)\gamma(X_j).
$$
Since the matrix entries of $\ad Z,\, Z\in\lie{h}$ in the basis $X_i$ are
$$
(\ad Z)_{ij}=B(\ad Z(X_j),X_i)=B([Z,X_j],X_i)=-B(Z,[X_i,X_j])
$$
Kostant \cite{Ko} obtains an explicit formula for $\alpha$ in this basis:
\begin{equation}
\alpha(Z)=\frac{1}{2}\sum_{i<j}B(Z,[X_i,X_j])\gamma(X_i)\gamma(X_j),\quad Z\in\lie{h}.
\end{equation}
\begin{rem}
Notice that in view of the map $\alpha:\lie{h}\To\Cliff(\lie{p})$ defined above, any $\Cliff(\lie{p})$-module can be viewed as an $\lie{h}$-module. 

In
particular, one could consider spin modules for $\Cliff(\lie{p})$. In this case all
the weights of the $\lie{h}$-representation can be given in a completely explicit form. For this
the reader should consult for instance \cite{HuaPan} \S 2.3.4.
\end{rem}
\subsection{The case with $\lie{h}$ symmetric} In this section, that follows very closely \cite{HuaPan},
we shall make the extra assumption
that $\lie{h}$ is a \emph{symmetric subalgebra} of $\lie{g}$, that is to say, there exists an involution $\sigma$ of $\lie{g}$
such that $\lie{h}$ is the fixed point set of $\sigma$. We further impose the condition that
$\sigma$ is orthogonal with respect to $B$. 

The picture to keep in mind is
the situation where $B$ is essentially the Killing form, which is known to be
invariant under all automorphisms of $\lie{g}$ and hence $\sigma$ is automatically orthogonal. It follows that
$\lie{p}:=\lie{h}^\perp$ (orthogonal complement with respect to $B$) is exactly the $(-1)$-eigenspace of $\sigma$ and
that the commutator $[\lie{p},\lie{p}]\subset\lie{h}$.
\begin{ex}
Let $G$ be a real reductive Lie group with Cartan involution $\Theta$ such that the fixed points of $\Theta$
form a maximal compact subgroup $H$ of $G$. If $\lie{g}$ is the complexified Lie algebra of $G$ and $\theta$ is the
complexified differential of $\Theta$, then the complexified Lie algebra $\lie{h}$ of $H$ is a symmetric subalgebra
of $\lie{g}$.
\end{ex}
In fact, one can show that every symmetric subalgebra is of this form for a suitably chosen $G$.

Now with the assumptions as above the following is known to hold.
\begin{prop}\label{KostCasConst}
Let $\lie{g}=\lie{h}\oplus\lie{p}$ be a Cartan decomposition and let $\alpha:U(\lie{h})\To\Cliff(\lie{p})$ be the
map defined in the previous section. Then the image of the Casimir element $\Delta_{\lie{h}}\in U(\lie{h})$ under $\alpha$
is the scalar $\norm{\rho_{\lie{g}}}^2-\norm{\rho_{\lie{h}}}^2$.
\end{prop}
\subsection{Our infinite-dimensional case}\label{OurInfDimCase}

Our plan in this section is to generalize the map $\alpha:\lie{h}\To\Cliff{\lie{p}}$ given in (\ref{KostantAlphaKostant}) to
our infinite-dimensional case at hand and realize the image as operators acting on the infinite-dimensional spinor-module
$\spib_{\lie{p}}^c$.
\begin{defn}
For each $N\in\N$ and $i,j\in\Z$ with $\abs{i},\abs{j}\leq N$ and $ij>0$ (i.e. $i$ and $j$ have the same sign) define a linear operator 
$K_{ij}^{(N)}$ acting on $\spib_\lie{p}^c$
via the Clifford algebra representation of $\Cliff(\lie{p}^\C)$,
\begin{equation}\label{Kij}
K_{ij}^{(N)}:=\frac{1}{2}\sum_{\substack{k\in\Z\\ \abs{k}\leq N\\ ik<0}}\gamma_{ik}\gamma_{kj}.
\end{equation}
\end{defn}
\begin{rem}
Notice that since $i$ and $j$ have the same sign, the requirement $ik<0$ is equivalent to $kj<0$, and hence the Clifford algebra elements
$\gamma_{ik}$ and $\gamma_{kj}$ are indeed well-defined in the sum (\ref{Kij}).
\end{rem}

\begin{lem}
For each $m,l\in\Z$ with $\abs{m},\abs{l}\leq N$ and $ml<0$
\begin{equation}\label{Kijisadj}
[K_{ij}^{(N)},\gamma_{ml}]=\sigma_{jm}\gamma_{il}-\delta_{il}\gamma_{mj}=\Big[\ad E_{ij}\Big]\Big(\gamma(E_{ml})\Big).
\end{equation}
\end{lem}
\begin{proof}
Using the identity $[AB,C]=A\{B,C\}-\{A,C\}B$ one obtains for $ik,kj,ml<0$ that
$$
[\gamma_{ik}\gamma_{kj},\gamma_{ml}]=2\gamma_{ik}\delta_{kl}\delta_{jm}-2\gamma_{kj}\delta_{il}\delta_{km}
$$
so that
\begin{equation}
[K_{ij}^{(N)},\gamma_{ml}]=\sum_{\substack{\abs{k}\leq N\\ ik<0}}[\gamma_{ik}\gamma_{kj},\gamma_{ml}]=\delta_{jm}\gamma_{il}-\delta_{il}\gamma_{mj}.
\end{equation}

The claim on the right hand side of (\ref{Kijisadj}) follows directly from the definition of the Lie algebra action of $\lie{h^\C}$ on $\Cliff(\lie{p}^\C)$, 
\begin{eqnarray}
\Big[\ad E_{ij}\Big]\Big(\gamma(E_{ml})\Big)&:=&\gamma\Big(\Big[\ad E_{ij}\Big](E_{ml})\Big)=
\gamma([E_{ij},E_{ml}])=\gamma(\delta_{jm}E_{il}-\delta_{il}E_{mj})\nonumber\\
&=&
\delta_{jm}\gamma(E_{il})-\delta_{il}\gamma(E_{mj}).\nonumber
\end{eqnarray}
\end{proof}

Next we notice that the operators $K_{ij}^{(N)}$ defined above satisfy the commutator relations of those
standard generators $E_{ij}\in\hinf^\C$, where $ij>0$ and $\abs{i},\abs{j}\leq N$.
\begin{lem}
For each $N\in\N$ and $i,j,m,n\in\Z$ with $ij>0$ and $mn>0$ 
\begin{equation}\label{KijCommutation}
[K_{ij}^{(N)},K_{mn}^{(N)}]=
\delta_{jm}K_{in}^{(N)}-\delta_{in}K_{mj}^{(N)}.
\end{equation}
\end{lem}

\begin{proof}
Using the well-known formula relating commutators with anti-commutators
$$
[AB,CD]=A\{B,C\}D-\{A,C\}BD+CA\{B,D\}-C\{A,D\}B
$$
one first computes that for $ij>0$ and $mn>0$ and $ik,kj,ml,ln<0$
\begin{eqnarray}
[\gamma_{ik}\gamma_{kj},\gamma_{ml}\gamma_{ln}]
&=&\gamma_{ik}\{\gamma_{kj},\gamma_{ml}\}\gamma_{ln}-\{\gamma_{ik},\gamma_{ml}\}\gamma_{kj}\gamma_{ln}\\
&+&\gamma_{ml}\gamma_{ik}\{\gamma_{kj},\gamma_{ln}\}-
\gamma_{ml}\{\gamma_{ik},\gamma_{ln}\}\gamma_{kj}\nonumber\\
&=&2\delta_{kl}\delta_{jm}\gamma_{ik}\gamma_{ln}-2\delta_{il}\delta_{km}\gamma_{kj}\gamma_{ln}\nonumber\\
&+&2\delta_{kn}\delta_{jl}\gamma_{ml}\gamma_{ik}-2\delta_{kl}\delta_{in}\gamma_{ml}\gamma_{kj}\nonumber.
\end{eqnarray}
Hence
\begin{eqnarray}
[K_{ij}^{(N)},K_{mn}^{(N)}]
&=&[\frac{1}{2}\sum_{\substack{\abs{k}\leq N\\ki<0}}\gamma_{ik}\gamma_{kj},
\frac{1}{2}\sum_{\substack{\abs{l}\leq N\\ lm<0}}\gamma_{ml}\gamma_{ln}]=\frac{1}{4}
\sum_{\substack{\abs{k},\abs{l}\leq N\\ ki,lm<0}}
[\gamma_{ik}\gamma_{kj},\gamma_{ml}\gamma_{ln}]\nonumber\\
&=&\delta_{jm}K_{in}^{(N)}-\delta_{in}K_{mj}^{(N)}.
\end{eqnarray}
\end{proof}

Next we shall introduce an auxiliary set of Lie algebra generators that will prove out to be useful in what follows.
\begin{defn}
For each $i,j\in\Z$ with $ij>0$ and $\abs{i},\abs{j}\leq N$ set
\begin{equation}
\Hn{ij}:=\frac{1}{4}\sum_{\substack{k\in\Z \\ ik<0\\ \abs{k}\leq N}}[\gamma_{ik},\gamma_{kj}].
\end{equation}
\end{defn}

Immediately one sees that since
$$
[\gamma_{ik},\gamma_{kj}]=2\gamma_{ik}\gamma_{kj}-\{\gamma_{ik},\gamma_{kj}\}=2\gamma_{ik}\gamma_{kj}-2\delta_{ij}\cdot I
$$
we have
\begin{equation}\label{HijGamma}
[\Hn{ij},\gamma_{ml}]=[K_{ij}^{(N)},\gamma_{ml}]=\delta_{jm}\gamma_{il}-\delta_{il}\gamma_{mj},
\end{equation}
where $m,j\in\Z$ with $ij<0$.

\begin{lem}
For each $i,j,m,n\in\Z$ with absolute value $\leq N$ and such that $ij,mn>0$, we have
\begin{equation}
[\Hn{ij},\Hn{mn}]=\delta_{jm}\Hn{in}-\delta_{in}\Hn{mj}
\end{equation}
\end{lem}
\begin{proof}
Follows from the general commutator formula
$$
[A,BC]=[A,B]C+B[A,C]
$$
and equation (\ref{HijGamma}).
\end{proof}

We would like to let $N\to\infty$ in $K_{ij}^{(N)}$ and obtain this way a representation of the Lie algebra
$\hinf^\C\subset\glres$ in $\spib_{\lie{p}}^c$ (notice that for $\lie{h}^\C$ the Schwinger term vanishes). Unfortunately, as usual, the operators $K_{ij}^{(N)}$ diverge in the limit.
To circumvent this difficulty one has to introduce suitable normal orderings for our operators
by using the fact that $\gamma_{ij}\vac_{\spib}=0$ when $i<0,j>0$ which in our
situation is seen as equivalent to the condition $i<j$. That is we want to replace the operators $K_{ij}^{(N)}$ with
operators $\KN{ij}$ satisfying the commutation relations (\ref{KijCommutation}) and annihilating the vacuum, $\KN{ij}\vac_{\spib}=0$.

Using the anticommutation relations determining $\Cliff(\lie{p}^\C)$ we can write for $i\in\Z$ with $\abs{i}\leq N$

\begin{eqnarray}\label{Kii}
K_{ii}^{(N)}&=&
\frac{1}{2}\Big(\sum_{\substack{k<i\ \\ik<0 \\ \abs{k}\leq N}}\gamma_{ik}\gamma_{ki}+
\sum_{\substack{k>i\\ ik<0 \\ \abs{k}\leq N}}(2-\gamma_{ki}\gamma_{ik})\Big)\nonumber\\
&=&
\frac{1}{2}\Big(\sum_{\substack{k<i\\ ik<0 \\ \abs{k}\leq N}}\gamma_{ik}\gamma_{ki}-
\sum_{\substack{k>i\\ ik<0 \\ \abs{k}\leq N}}\gamma_{ki}\gamma_{ik}\Big)+\sum_{\substack{k>i\\ik<0 \\ \abs{k}\leq N}} 1
\end{eqnarray}
and similarly for $i\neq j,ij>0$ with $\abs{i},\abs{j}\leq N$ we have
\begin{eqnarray}
K_{ij}^{(N)}&=&\frac{1}{2}\Big(\sum_{\substack{k<j\\ik<0,jk<0\\ \abs{k}\leq N}}\gamma_{ik}\gamma_{kj}+
\sum_{\substack{k>j\\ ik<0,jk<0\\ \abs{k}\leq N}}\gamma_{ik}\gamma_{kj}\Big)\nonumber\\
&=&
\frac{1}{2}\Big(\sum_{\substack{k<j\\ ik<0,jk<0\\ \abs{k}\leq N}}\gamma_{ik}\gamma_{kj}-
\sum_{\substack{k>j\\ ik<0,jk<0 \\ \abs{k}\leq N}}\gamma_{kj}\gamma_{ik}\Big).
\end{eqnarray}
Notice that
$$
\sum_{\substack{k>i\\ik<0 \\ \abs{k}\leq N}} 1=
\left\{
\begin{array}{ll}
N & \textrm{if $i<0$}\\
0 & \textrm{if $i>0$}.
\end{array}\right.
$$

Thus if we define the normal ordering
\begin{equation}\label{normalorderedgamma}
\normord\gamma_{ik}\gamma_{kj}\normord:=
\left\{ \begin{array}{ll}
\gamma_{ik}\gamma_{kj} & \textrm{if $i\neq j$ }\\
\gamma_{ik}\gamma_{ki}& \textrm{if $i=j$ and $k<i$ }\\
 -\gamma_{ki}\gamma_{ik}=\gamma_{ik}\gamma_{ki}-2 & \textrm{if $i=j$ and $k>i$},
\end{array} \right.
\end{equation}
where $ij>0$ and $ik<0$, and set
\begin{equation}\label{KijtildeVersusKij}
\KN{ij}:=\frac{1}{2}\sum_{\substack{k\in\Z\\ \abs{k}\leq N\\ik<0}}
\normord \gamma_{ik}\gamma_{kj}\normord=
\left\{
\begin{array}{ll}
K_{ij}^{(N)} & \textrm{if $i\neq j$ or $i=j>0$}\\
K_{ij}^{(N)}-N\cdot 1 & \textrm{if $i=j<0$}.
\end{array}\right.
\end{equation}
then since the constant term $-1$ in (\ref{KijtildeVersusKij}) commutes with everything it follows that
\begin{eqnarray}
[\KN{ij},\KN{mn}]&=&[K_{ij}^{(N)},K_{mn}^{(N)}]=\delta_{jm}K_{in}^{(N)}-\delta_{in}K_{jm}^{(N)}\nonumber\\
&=& \delta_{jm}\KN{in}-\delta_{in}\KN{jm}+c(K_{ij}^{(N)},K_{mn}^{(N)}).
\end{eqnarray}

\begin{lem} 
For all $ij,mn>0$, the term
$c(K_{ij}^{(N)},K_{mn}^{(N)})=0$,
or equivalently
\begin{equation}\label{comrelKN}
[\KN{ij},\KN{mn}]=\delta_{jm}\KN{in}-\delta_{in}\KN{jm}.
\end{equation}
\end{lem}
\begin{proof}
Let $\chi_{\leq 0}:\Real\To\Real$ be the characteristic function of the non-positive
real numbers,
$$
\chi_{\leq 0}:\Real\To\Real,\quad\chi_{\leq 0}(x)=
\left\{
\begin{array}{ll}
1 & \textrm{if $x\leq 0$}\\
0 & \textrm{if $x>0$}.
\end{array}\right.
$$
Notice that if $i,j\in\Z$ with $ij>0$ (i.e. $i$ and $j$ have the same sign) then
$$
\chi_{\leq 0}(j)-\chi_{\leq 0}(i)=0
$$

Now write $\KN{ij}=K_{ij}^{(N)}-\delta_{ij}\chi_{\leq 0}(i)N\cdot 1$ so that
\begin{eqnarray}
\delta_{jm}\KN{in}-\delta_{in}\KN{kj}
&=&\delta_{jm}(K_{in}^{(N)}-\delta_{in}\chi_{\leq 0}(i)N)-\delta_{in}(K^{(N)}_{jm}-\delta_{jm}\chi_{\leq 0}(j)N)\nonumber\\
&=&
\delta_{jm}K_{in}^{(N)}-\delta_{in}K_{jm}^{(N)}
+\delta_{in}\delta_{jm}N\Big(\chi_{\leq 0}(j)-\chi_{\leq 0}(i)\Big)\nonumber\\
&=&
\delta_{jm}K_{in}^{(N)}-\delta_{in}K_{jm}^{(N)}.
\end{eqnarray}
\end{proof}


\begin{lem}\label{Kijlimit}
For each $m,l\in\Z$ with $\abs{m},\abs{l}\leq N$ and $ml<0$
\begin{eqnarray}
[\KN{ij},\gamma_{ml}] &=& [K_{ij}^{(N)},\gamma_{ml}]=\sigma_{jm}\gamma_{il}-\delta_{il}\gamma_{mj}
=\Big[\ad E_{ij}\Big]\Big(\gamma(E_{ml})\Big),\nonumber\\
\,[\KN{ii},\gamma_{ml}] &=& \left\{ \begin{array}{ll}
\delta_{il}\cdot \sign(i)\sign(m)\gamma_{mi}=-\delta_{il}\gamma_{mi} & \textrm{if $i>0$ and $m<0$}\nonumber\\
\delta_{im}\cdot \sign(i)\sign(m)\gamma_{il}=\delta_{im}\gamma_{il} & \textrm{if $i>0$ and $m>0$}\nonumber\\
\delta_{im}\cdot\sign(i)\sign(m)\gamma_{il}=\delta_{im}\gamma_{il} & \textrm{if $i<0$ and $m<0$ }\nonumber\\
\delta_{il}\cdot\sign(i)\sign(m)\gamma_{mi} =-\delta_{il}\gamma_{ml}& \textrm{if $i<0$ and $m>0$},
\end{array} \right.\\
\KN{ij}\vac_{\spib} &=& 0.\nonumber
\end{eqnarray}
\end{lem}

\begin{proof}
\emph{The first claim.} Again, the constant terms in (\ref{Kii}) commute with everything so that the commutator rules
for the normally ordered operators remain the same as the original ones.

\emph{The second claim.}
First note that since $i$ and $k$ have different signs we have that $k<i\iff k<0,i>0$ and $k>i\iff k>0, i<0$. Write 
$$
\KN{ii}=\frac{1}{2}\sum_{\substack{\abs{k}\leq N\\ ik<0}}\normord\gamma_{ik}\gamma_{ki}\normord=\frac{1}{2}
\overbrace{\sum_{\substack{\abs{k}\leq N\\ ik<0,k<i}}\gamma_{ik}\gamma_{ki}}^{\neq 0\iff  i>0}-
\frac{1}{2}
\overbrace{\sum_{\substack{\abs{k}\leq N\\ ik<0,k>i}}\gamma_{ki}\gamma_{ik}}^{\neq 0\iff i<0}
$$
and recall that
$$
[\gamma_{ik}\gamma_{kj},\gamma_{ml}]=
2\gamma_{ik}\delta_{kl}\delta_{jm}-2\gamma_{kj}\delta_{il}\delta_{km}.
$$
This gives us
\begin{eqnarray}
[\KN{ii},\gamma_{ml}]&=&\frac{1}{2}
\sum_{\substack{\abs{k}\leq N\\ ik<0,\, k<i}}[\gamma_{ik}\gamma_{ki},\gamma_{ml}]-\frac{1}{2}
\sum_{\substack{\abs{k}\leq N\\ ik<0,\, k>i}}[\gamma_{ki}\gamma_{ik},\gamma_{ml}]\nonumber\\
&=&
\sum_{\substack{\abs{k}\leq N\\ ik<0,\, k<i}}
(\gamma_{ik}\delta_{kl}\delta_{im}-\gamma_{ki}\delta_{il}\delta_{km})
-\sum_{\substack{\abs{k}\leq N\\ ik<0,\, k>i}}
(\gamma_{ki}\delta_{il}\delta_{km}-\gamma_{ik}\delta_{kl}\delta_{im})\nonumber\\
\end{eqnarray}

After this, the claim follows after a careful case-by-case analysis.

\emph{The third claim.} This follows directly from the definitions.
\end{proof}

\begin{rem}\label{remarkforstrongconv}
Let $m,l\in\Z$. Notice that if $\abs{m},\abs{l}\leq N$ with $ml<0$, and if we have any integer $M\geq N$, the
statement of Lemma \ref{Kijlimit} still holds trivially if one replaces $\KN{ij}$ with
$\widetilde{K}_{ij}^{(M)}$, e.g.
$[\widetilde{K}_{ij}^{(M)},\gamma_{ml}]=[\ad E_{ij}](\gamma(E_{ml}))$ for all
$M\geq N$ etc. That is to say, with the chosen parameters, the result holds starting
\emph{from} $N$. To make things clearer, we give the following definition.
\end{rem}

\begin{defn}
For every $N\in\Z_{\geq 0}$ define the subspace $\spib^{\mathrm{pol},(N)}\subset\spib^{\mathrm{pol}}$ to consist of all
finite $\C$-linear combinations of the vacuum vector $\vac_\spib$ and the basis vectors 
\begin{equation}\label{basisforourfermions}
s=\gamma_{m_1 l_1}\gamma_{m_2 l_2}\cdots\gamma_{m_k l_k}\vac_{\spib}\in \spib^{\mathrm{pol}}\quad (m_i l_i<0 \textrm{ for all }i),
\end{equation}
such that $\abs{m_i},\abs{l_i}\leq N$
for all $i=1,\ldots k\geq 0$.
\end{defn}

It is then clear that $\spib^{\mathrm{pol},(M)}\subset\spib^{\mathrm{pol},(N)}$ if $M\leq N$ and that
$\bigcup_{N=1}^\infty \spib^{\mathrm{pol},(N)}=\spib^{\mathrm{pol}}$,
which is dense in $\spib_{\lie{p}}^c$. 
Thus, if one sets $F_p=F_p\spib^{\mathrm{pol}}:=\spib^{\mathrm{pol},(p)}$ for $p\in\Z_{\geq 0}$, then
$$
F_\bullet=\set{F_p}_{p\geq 0}
$$
yields an increasing filtration of the $\C$-vector space $\spib^{\mathrm{pol}}$ by $\C$-vector subspaces.
Note also that by Corollary
\ref{Cliffordvacannihilated} and anti-commutation rules of the Clifford algebra we may assume the above basis vectors $s$ in (\ref{basisforourfermions})
to be of the form
\begin{equation}
s=\gamma_{m_1 l_1}\gamma_{m_2 l_2}\cdots\gamma_{m_k l_k}\vac_{\spib}\in \spib^{\mathrm{pol}}\quad 
(m_i>0,\, l_i<0 \textrm{ for all }i),
\end{equation}

\begin{prop}\label{actioniscommutator}
Fix $M\in\Z_{\geq 0}$ and consider all $N\in\Z_{\geq 0}$ such that $N\geq M$.
Then for all pairs $i,j\in\Z$ with $ij>0$, the action of the operators $\KN{ij}$ on 
the subspace $\spib^{\mathrm{pol},(M)}\subset\spib_{\lie{p}}^c$ coincides with the commutator action, i.e. 
$\KN{ij}\vac_\spib=0$ and
if
$$
s=\gamma_{m_1 l_1}\gamma_{m_2 l_2}\cdots\gamma_{m_k l_k}\vac_{\spib}
=\gamma(E_{m_1 l_1})\gamma(E_{m_2 l_2})\cdots\gamma(E_{m_k l_k})\vac_{\spib}\in\spib^{\mathrm{pol},(M)}
$$
is a basis vector 
for $\spib^{\mathrm{pol},(M)}$ 
then
$$
\KN{ij}\cdot s=[\KN{ij},s]:=
\sum_{h=1}^k \gamma_{m_1 l_1}\cdots[\KN{ij},\gamma_{m_h l_h}]\cdots\gamma_{m_k l_k}\vac_\spib.
$$
\end{prop}

\begin{proof}
By induction on the 'length' $k$.
Let first $k=1$ so that 
$s=\gamma_{ml}\vac_\spib$, where $ml<0$. By Lemma \ref{Kijlimit} 
and Remark \ref{remarkforstrongconv}
we may write
$$
\KN{ij}\cdot s=\KN{ij}\cdot\gamma_{ml}\vac_{\spib}=[\KN{ij},\gamma_{ml}]\vac_{\spib}.
$$
We make the induction hypothesis that the claim holds for $k=n$, i.e.
that for
$$
s=\gamma_{m_1 l_1}\gamma_{m_2 l_2}\cdots\gamma_{m_n l_n}\vac_{\spib}
$$
$$
\KN{ij}\cdot s=
\sum_{h=1}^k \gamma_{m_1 l_1}\cdots[\Kt{ij},\gamma_{m_h l_h}]\cdots\gamma_{m_n l_n}\vac_\spib.
$$
Now let
$$
s'=\gamma_{m l}\gamma_{m_1 l_1}\gamma_{m_2 l_2}\cdots\gamma_{m_n l_n}\vac_{\spib}
=\gamma_{m l}\cdot s
$$
and write 
$$
\Kt{ij}s'=\Kt{ij}\gamma_{ml}s=[\Kt{ij},\gamma_{ml}]s+\gamma_{ml}\Kt{ij}s
$$
and use the induction hypothesis to the term $\gamma_{ml}\Kt{ij}s$.
\end{proof}

\begin{defn}
For each $N\in\Z_{>0}$ let $\lie{h}^\C_{(N)}\subset\hinf^\C$ be the Lie subalgebra/$\,\C$ spanned 
by all the elements $E_{ij}\in\lie{h}^\C\,\,(ij>0)$ with
$\abs{i},\abs{j}\leq N$, so that $\lie{h}^\C_{(N)}\isom\gl_n(\C)\times\gl_n(\C)$.
\end{defn}

\begin{cor}\label{hcN}
Fix $M\in\Z_{\geq 0}$ and consider all integers $N\in\Z_{>0}$ such that $N\geq M$.
For each such $N$
the map $E_{ij}\mapsto\KN{ij}$, when extended $\C$-linearly,
yields a Lie algebra representation 
$$
\pi_{(N)}:\lie{h}^\C_{(N)}\To\End(\spib^{\mathrm{pol}}),
$$
where
$\spib^{\mathrm{pol}}=\F^{\mathrm{pol}}(\lie{p}^\C,W)$. When acting on the elements of
the subspace $\spib^{\mathrm{pol},(M)}$, the action given by $\pi_{(N)}$ coincides with
the action given by the isotropy / representation $\ad$, when $\ad$ is extended to act
on the infinite-dimensional wedge product $\spib^{\mathrm{pol}}$.
\end{cor}

\begin{proof}
The commutation relation (\ref{comrelKN}) guarantees that $\pi_{(N)}$ is indeed a Lie algebra
representation of $\lie{h}^\C_{(N)}$. According to
Lemma \ref{Kijlimit}, for proper values of $m,l$, one has  
$[\KN{ij},\gamma_{ml}]=[\ad E_{ij}](\gamma(E_{ml}))$ so that by
Proposition \ref{actioniscommutator} it holds that for every element
$$
s=\gamma_{m_1 l_1}\gamma_{m_2 l_2}\cdots\gamma_{m_k l_k}\vac_{\spib}\in\spib^{\mathrm{pol},(M)}
$$
the action of $\KN{ij}$ is given by
\begin{equation}
\KN{ij}\cdot s=
\sum_{h=1}^k \gamma_{m_1 l_1}\cdots\Big([\ad E_{ij}](\gamma_{m_h l_h})\Big)\cdots\gamma_{m_k l_k}\vac_\spib
\qquad(\textrm {for all } N\geq M).
\end{equation}
\end{proof}

\begin{defn}
Let
$$
\Kt{ij}:=
\frac{1}{2}\sum_{\substack{k\in\Z\\ik<0}}
\normord \gamma_{ik}\gamma_{kj}\normord.
$$
\end{defn}

\begin{cor}\label{isotropycorollary}
For each $i,j\in\Z$ with $ij>0$ the normal ordered infinite sum $\Kt{ij}$ defines a well-defined unbounded operator 
$\Kt{ij}:\spib_{\lie{p}}^c\To\spib_{\lie{p}}^c$
with dense domain $\domain(\Kt{ij})=\spib^{\mathrm{pol}}\subset\spib_{\lie{p}}^c$. 
The image of the common domain $\spib^{\mathrm{pol}}$ for all $\Kt{ij}$ maps into itself, 
$\Kt{ij}(\spib^{\mathrm{pol}})\subset\spib^{\mathrm{pol}}$.
Moreover,
the map $E_{ij}\mapsto \Kt{ij}$, when extended $\C$-linearly, realizes the isotropy representation $\ad$ of the Lie algebra
$\hinf^\C\subset\Eglres(\Hilb,\Hilb_+)$ on $\spib^{\mathrm{pol}}=\F^{\mathrm{pol}}(\lie{p}^\C,W)$
as an infinite sum of quadratic terms of CAR algebra generators.
\end{cor}

\begin{proof}
It is easy to see, that the computations made in the proof of Lemma \ref{Kijlimit} are still
valid when $N=\infty$ in which case we have no restrictions on the absolute values
of the indices $m,l$ appearing in the Clifford algebra generators $\gamma_{ml}\,(ml<0)$. It follows
from this that also the result of Proposition \ref{actioniscommutator} holds
when we replace $\KN{ij}$ with $\Kt{ij}$ and $\spib^{\mathrm{pol},(M)}$ with
$\spib^{\mathrm{pol}}$, i.e. the action of $\Kt{ij}$ on $\spib^{\mathrm{pol}}$ coincides
with $\ad E_{ij}$, where $\ad$ denotes the isotropy representation of $\hinf^\C$ on $\spib^{\mathrm{pol}}$, which
is of course well-defined unbounded operator and maps the dense domain $\spib^{\mathrm{pol}}$
into itself.
\end{proof}

\begin{prop}\label{KijStrongLimit}
For all $i,j\in\Z$ with $ij>0$, the strong limit
$$
\stlim{N}\KN{ij}=\Kt{ij},
$$
where we consider each of the operators $\KN{ij},\Kt{ij}$ as an unbounded
operator $\spib_{\lie{p}}^c\To\spib_{\lie{p}}^c$ with dense domain $\spib^{\mathrm{pol}}$.
Moreover, the above limit is uniform in the pair $(i,j)\in\Z\times\Z$ in the sense that, if
$s\in\spib^{\mathrm{pol}}$ there exists an integer $M=M(s)\in\Z_{\geq 0}$ such that
$$
\KN{ij}(s)=\Kt{ij}(s)\quad\textrm{ for all } i,j\in\Z,\,ij>0,
$$
whenever $N\geq M$.
\end{prop}
\begin{proof}
Let $s\in\spib^{\mathrm{pol}}=\domain(\KN{ij})=\domain(\Kt{ij})$ 
be arbitrary. Then obviously there exists an integer $M\in\Z_{\geq 0}$ such
that $s\in\spib^{\mathrm{pol},(M)}\subset\spib^{\mathrm{pol}}$. Using Corollary \ref{hcN}
and Corollary \ref{isotropycorollary} we obtain that
$$
\KN{ij}\cdot s=[\ad E_{ij}](s)=\Kt{ij}\cdot s\quad\textrm{ for all } N\geq M,
$$
so that in particular
$$
\lim_{N\to\infty}\norm{\KN{ij}\cdot s-\Kt{ij}\cdot s}=0,
$$
from which the claim follows.
\end{proof}

\section{The Dirac operator}
\subsection{The space of ``$L^2$-spinors'' on $\RGr$}

Recall that in the case of a compact Lie group $G$ the Hilbert space of square integrable spinors on a homogeneous
space $G/H$ satisfies
$$
L^2(G/H,S)\isom L^2(G\times_H\spib_{\lie{p}})\isom\widehat{\bigoplus}_\lambda
V_\lambda\tensor(V_\lambda^\ast\tensor\spib_{\lie{p}})^H,
$$
where the Hilbert space direct sum is taken over all irreducible representations $V_\lambda$
of $G$. Moreover the Dirac $\KDirac$ operator on $G/H$ kept fixed all individual summands in
this decomposition and acted trivially on each $V_\lambda$ on the left hand side
of the tensor product $V_\lambda\tensor(V_\lambda^\ast\tensor\spib_{\lie{p}})$. 

The
Lie subgroup $H\subset G$  acts on each $V_\lambda^\ast$ by restricting the
action of $G$ on $V_\lambda^\ast$ and on the spinor module $\spib_{\lie{p}}$ via
the composition of the lift of the isotropy representation $\Adt:H\To\Spin(\lie{p})$ and
the spin representation $\Spin(\lie{p})\To\End(\spib_{\lie{p}})$. Hence $H$ acts
on each $V_\lambda^\ast\tensor\spib_{\lie{p}}$ via tensor product of the above representations.
Then the subspace $(V_\lambda^\ast\tensor\spib_{\lie{p}})^H\subset V_\lambda^\ast\tensor\spib_{\lie{p}}$ of
$H$-\emph{invariant} vectors in $V_\lambda^\ast\tensor\spib_{\lie{p}}$ consist of all vectors in $V_\lambda^\ast\tensor\spib_{\lie{p}}$
staying invariant under the above action of $H$. Concretely, if we denote by $\pi_{\lambda^{\ast}}:G\To\Aut(V_\lambda^\ast)$ and
$\Adt:H\To\Aut(\spib_{\lie{p}})$ the above representations, then
\begin{eqnarray}
(V_\lambda^\ast\tensor\spib_{\lie{p}})^H&=&
\big\{\sum v_i\tensor s_j\in V_\lambda^\ast\tensor\spib_{\lie{p}}\mid \sum \pi_{\lambda^\ast}(h)v_i\tensor \Adt(h)s_j=\sum v_i\tensor s_j\nonumber\\
& &\textrm{for all } h\in H\big\}.
\end{eqnarray}

Differentiating the above representations we obtain the corresponding Lie algebra representations
$r_{\lambda^\ast}:\lie{g}\To\End(V_\lambda^\ast)$ and $\adt:\lie{h}\To\End(\spib_{\lie{p}})$ and we
may consider the subspace $(V_\lambda^\ast\tensor\spib_{\lie{p}})^{\lie{h}}\subset V_\lambda^\ast\tensor\spib_{\lie{p}}$
of $\lie{h}$-\emph{invariant} (or $\lie{h}$-\emph{equivariant}) vectors in $V_\lambda^\ast\tensor\spib_{\lie{p}}$
consisting of all elements in $V_\lambda^\ast\tensor\spib_{\lie{p}}$ that are annihilated by
the (diagonal) tensor product representation $r_{\lambda^\ast}\tensor\adt$ of $\lie{h}$, i.e.
\begin{eqnarray}
(V_\lambda^\ast\tensor\spib_{\lie{p}})^{\lie{h}}&=&\big\{\sum v_i\tensor s_j\in V_\lambda^\ast\tensor\spib_{\lie{p}}
\mid \sum r_{\lambda^\ast}(h)v_i\tensor s_j+\sum v_i\tensor \adt(h)s_j=0\nonumber\\
& &\textrm{for all } h\in\lie{h}
\big\}
\nonumber\\
&=&
\big\{\sum v_i\tensor s_j\in V_\lambda^\ast\tensor\spib_{\lie{p}}
\mid \sum r_{\lambda^\ast}(h)v_i\tensor s_j=-\sum v_i\tensor \adt(h)s_j\nonumber\\
& &\textrm{for all } h\in\lie{h}
\big\}\nonumber\\
&=&\Big\{w\in V_\lambda^\ast\tensor\spib_{\lie{p}}\,\Big\vert\,\Big[r_{\lambda^\ast}(h)\tensor 1\Big](w)=\Big[-1\tensor \adt(h)\Big](w)
\quad\textrm{ for all }\nonumber\\
& & h\in\lie{h}\Big\}.
\end{eqnarray}
Thus, if $E_1,\ldots, E_m$ is a basis for $\lie{h}$, we may write the $\lie{h}$-invariant sector as
\begin{eqnarray}
(V_\lambda^\ast\tensor\spib_{\lie{p}})^{\lie{h}}&=&\Big\{w\in V_\lambda^\ast\tensor\spib_{\lie{p}}\,\Big\vert\,\Big[r_{\lambda^\ast}(E_i)\tensor 1\Big](w)=\Big[-1\tensor \adt(E_i)\Big](w)
\quad\textrm{ for all }\nonumber\\
& & i=1,\ldots,m\Big\}.
\end{eqnarray}

Now we want to consider a proper analog of the above definition for the 
infinite-dimensional restricted Grassmannian manifold
$M=\RGr(\Hilb,\Hilb_+)\isom G/H$, where we have $G=\Ures(\Hilb,\Hilb_+)$ and 
$H=U(\Hilb_+)\times U(\Hilb_-)$. Now instead of considering \emph{all} the possible
irreducible (projective) representations of $\Ures$ we decide to work with a kind of ``minimal'' Hilbert space
in which we only consider \emph{one} irreducible representation of $\EUres$, namely the basic
representation of $\EUres$ on the fermionic Fock space $\F:=\F(\Hilb,\Hilb_+)$. Differentiating we obtain
the basic representation $\rh$ of $\Eures$ on $\F$ which we complexify to a representation
of $\Eglres$ on $\F$.

We will also want to consider the complexifed Clifford algebra $\Cliff(\lie{p}^\C)=\Cliff(W\oplus\overline{W})$
presented in section \S\ref{complexified Clifford}. Recall that $\Cliff(\lie{p}^\C)$ had an action
on the space $\spib_{\lie{p}}^c=\F(\lie{p}^\C,W)\isom\F_+(\lie{p^\C})\tensor\F_-(\lie{p}^\C)$ which we
consider as the complexified spinor module for $\Cliff(\lie{p}^\C)$.

In section \S\ref{inducedad} we noticed that the spin lift $\adt$ of the isotropy representation was trivial with the central element acting
by zero. Hence the complexification $\lie{h}^\C$ acts nonprojectively on $\spib_{\lie{p}}^c$ via $\adt$ and for this reason
we will denote $\ad$ and its spin lift $\adt$ by the same symbol from now on. Now let's see what this means concretely. Recall
that if $V$ is a (finite-dimensional) representation of a Lie algebra $\lie{g}$ of a Lie subgroup of
$GL(n,\C)$ then $\lie{g}$ acts on $\bigwedge^k V$ by
$$
X\sum v_{i_1}\wedge\cdots\wedge v_{i_k}=\sum (Xv_{i_1})\wedge\cdots\wedge v_{i_k}+
\cdots+
\sum v_{i_1}\wedge\cdots\wedge(Xv_{i_k}).
$$
Now since the complexified spinor module $\spib_{\lie{p}}^c=\F(\lie{p}^\C,W)$ is by definition an infinite direct
sum of wedge products of Hilbert spaces and since $[\lie{h^\C},\lie{p}^\C]\subset\lie{p}^\C$ we
have the following situation: If
$$
s=\gamma_{m_1 l_1}\gamma_{m_2 l_2}\cdots\gamma_{m_k l_k}\vac_\spib=\gamma(E_{m_1 l_1})\gamma(E_{m_2 l_2})\cdots\gamma(E_{m_k l_k})\vac_\spib,
$$
where each $m_j\, l_j>0\,(j=1,\ldots k)$, is a basis vector for $\spib_{\lie{p}}^c=\F(\lie{p}^\C,W)$ then
for the basis elements $E_{pq}\in\hinf^\C\, (p,q\in\Z,\,pq>0)$
\begin{equation}
\ad(E_{pq})(s)=\sum_{i=1}^k \gamma(E_{m_1 l_1})\cdots\gamma([E_{pq},E_{m_i l_i}])\cdots\gamma(E_{m_k l_k})\vac_\spib.
\end{equation}

We are ready to consider the tensor product representation
\begin{equation}
\varrho:\hinf^\C\To\End(\F\tensor\spib_{\lie{p}}^c),\quad\varrho=\hat{r}\tensor 1+1\tensor\ad,
\end{equation}
where $\rh$ denotes the restriction of the basic representation of $\Eglres$ on $\F$ to the Lie subalgebra
$\hinf^\C\subset\Eglres$. Hence, by linearity, the space of $\hinf^\C$-invariants 
is given by
\begin{eqnarray}
(\F\tensor\spib_{\lie{p}}^c)^{\hinf^\C,\varrho} &=&
\Big\{w\in \F\tensor\spib_{\lie{p}}^c\,\Big\vert\,\Big[\rh(E_{pq})\tensor 1\Big](w)=\Big[-1\tensor \ad(E_{pq})\Big](w)
\quad\textrm{ for all }\nonumber\\
& & p,q\in\Z\textrm{ such that }pq>0\Big\}\nonumber\\
&=&
\Big\{w\in \F\tensor\spib_{\lie{p}}^c\,\Big\vert\,\Big[\rh(E_{pq})\tensor 1\Big](w)=\Big[-1\tensor \Kt{pq}\Big](w)
\quad\textrm{ for all }\nonumber\\
& & p,q\in\Z\textrm{ such that }pq>0\Big\}.
\end{eqnarray}
From now on, in order to save notation we set, by abuse of notation, 
\begin{equation}
(\F\tensor\spib_{\lie{p}}^c)^{\lie{h}}:=(\F\tensor\spib_{\lie{p}}^c)^{\hinf^\C,\varrho}.
\end{equation}

Motivated by these facts
we give the following definition.

\begin{defn}[Hilbert space of spinors]
Consider the restricted Grassmannian manifold
$M=\RGr(\Hilb,\Hilb_+)\isom G/H$, where we have $G=\Ures(\Hilb,\Hilb_+)$ and 
$H=U(\Hilb_+)\times U(\Hilb_-)$.
Set 
\begin{equation}\label{ourspinors}
L^2(\RGr(\Hilb,\Hilb_+),S_\C):=\Big(\F_0(\Hilb,\Hilb_+)\tensor\spib_\lie{p}^c\Big)^{\hinf^\C,\varrho}.
\end{equation}
\end{defn}

\begin{rem}\label{invariantsecnonempt}
Of course we need to know that the space $(\F_0\tensor\spib_{\lie{p}}^c)^{\lie{h}}$ is nonempty! This follows
from the fact that for $p,q\in\Z$ with $pq>0$ we have 
$\rh(E_{pq})\vac_\F=0=\Kt{pq}\vac_\spib$ 
so that at least the vacuum vector $\vac=\vac_\F\tensor\vac_\spib$ belongs to $(\F_0\tensor\spib_{\lie{p}}^c)^{\lie{h}}$.
\end{rem}
\begin{rem}
In order us to be able to prove the finite-di\-men\-sio\-na\-li\-ty of
the kernel $\ker\KDirac$ we shall need to consider the
smaller Hilbert space $(\F_0\tensor\spib_{\lie{p}}^c)^{\lie{h}}$ instead
of the larger space $(\F\tensor\spib_{\lie{p}}^c)^{\lie{h}}$.
However, if some properties, e.g. being well-defined, can be proved also over the bigger space
$\F\tensor\spib_{\lie{p}}^c$ or $(\F\tensor\spib_{\lie{p}}^c)^{\lie{h}}$ we will
do so in this generality.
\end{rem}

\subsection{The diagonal Casimir operator}\label{diagonalCasOp}
Recall that we had the diagonal representation $\varrho:\hinf^\C:\To\End(\F\tensor\spib_{\lie{p}}^c)$ with
\begin{equation}
E_{ij}\mapsto E_{ij,\varrho}:=\rh(E_{ij})\tensor 1+1\tensor\ad(E_{ij})\quad\textrm{ for all } ij>0,
\end{equation}
where $\rh$ denotes the restriction of the basic representation of $\Eglres$ on $\F$ to its Lie subalgebra $\hinf^\C$.
Let
$$
E_{ij,\varrho}^{(N)}:=\rh(E_{ij})\tensor 1+1\tensor\KN{ij}\quad\textrm{ for all } ij>0,
$$
so that in the `limit' $N\to\infty$ the operator $E_{ij,\varrho}^{(N)}\stackrel{N\to\infty}{\to}E_{ij,\varrho}$.

We define the $N^\mathrm{th}$ cut-off diagonal Casimir operator of $\hinf^\C$ to be
\begin{eqnarray}\label{diagonalCasimiroperator}
\Delta_{\varrho}^{(N)} &:=& \sum_{\substack{i,j\in\Z \\ ij>0\\ \abs{i},\abs{j}\leq N}}E_{ij,\varrho}^{(N)}E_{ji,\varrho}^{(N)}\nonumber\\
&=&
\sum_{\substack{i,j\in\Z \\ ij>0\\ \abs{i},\abs{j}\leq N}}\Big(\rh(E_{ij})\tensor 1+1\tensor\KN{ij}\Big)\Big(\rh(E_{ji})\tensor 1+1\tensor\KN{ji}\Big)\nonumber\\
&=&
\sum_{\substack{i,j\in\Z \\ ij>0\\ \abs{i},\abs{j}\leq N}}\rh(E_{ij})\rh(E_{ji})\tensor 1+
2\sum_{\substack{i,j\in\Z \\ ij>0\\ \abs{i},\abs{j}\leq N}}\rh(E_{ij})\tensor\KN{ji}\nonumber\\
&+&
\sum_{\substack{i,j\in\Z \\ ij>0\\ \abs{i},\abs{j}\leq N}} 1\tensor\KN{ij}\KN{ji}
\end{eqnarray}
so that $\Delta_{\varrho}^{(N)}\in\End(\F\tensor\spib_{\lie{p}}^c)$.

\subsection{Working the intuition}
Since the restricted Grassmannian manifold is a homogeneous space, in
giving a reasonable definition for a Dirac operator in our infinite-dimensional setting, at first step we would like to mimic the expression of the (true)
Dirac operator (\ref{DiracGH}) on homogeneous spaces
$G/H$ where $G$ is a compact Lie group,
$$
\KDirac=\sum_{i=1}^{\dim\lie{p}}r(X_i)\tensor X_i^\ast+1\tensor
\frac{1}{2} X_i^\ast\cdot\adt_{\lie{p}}X_i\in U(\lie{g})\tensor\Cliff(\lie{p}),
$$
where $\set{X_i}$ is a basis of $\lie{p}\isom\lie{g}/\lie{h}$ and $\set{X_i^\ast}$ its
dual basis. Here $r$ restricts from $\lie{g}$ to a representation of $\lie{p}\subset\lie{g}$.
Unfortunately, as we noted the square of the above expression turned out to be quite a mess and one would
expect that the analogous situation would be even harder to control in infinite dimensions.

However, the restricted Grassmannian manifold $\RGr\isom G/H$,
where we have $G=\Ures(\Hilb,\Hilb_+)$ and $H=U(\Hilb_+)\times U(\Hilb_-)$, is not just a (Banach) homogeneous space, but has
more structure being an infinite-dimensional symmetric space:

\begin{thm}[Spera, Wurzbacher, \cite{SpeWu}] The restricted Grassmannian manifold is
a Hermitean symmetric space. In particular
$$
[\lie{h},\lie{p}]\subset\lie{p},\quad [\lie{p},\lie{p}]\subset\lie{h}.
$$
$\RGr$ is geodesically complete and its Riemann curvature tensor is completely fixed by its value in the point
$\Hilb_+$. Furthermore the trace corresponding to Ricci curvature of $\RGr$ is ``linearly divergent''.
\end{thm}

Thus according to \S\ref{DOonSymSpa} a better analogue would be the operator
\begin{equation}\label{OurCandidate}
\KDirac=\sum_{i=1}^{\dim\lie{p}}r(X_i)\tensor X_i^\ast,
\end{equation}
for which the Dirac operator on a homogeneous space reduces when the homogeneous space happens to be symmetric. Notice
that in this kind of special case this coincides with the Kostant's cubic Dirac operator \cite{Ko} for a pair $(\lie{g},\lie{h})$ ($\lie{g}$ compact
semisimple, $\lie{h}$ reductive maximal rank subalgebra of $\lie{g}$), namely if 
we are in a situation $\lie{g}=\lie{h}\oplus\lie{p}$ with $\lie{h}$ and $\lie{p}$ complementary and
$[\lie{p},\lie{p}]\subset\lie{h}$, the cubic term in Kostant's Dirac operator vanishes indetically.

In general, both operators, the Dirac operator on a compact Riemannian symmetric space and Kostant's (algebraic) cubic Dirac operator
have the special property that their squares can be expressed in terms of various Casimir operators for $\lie{g}$ and $\lie{h}$.
This has the impact that the squares of these Dirac operators can be analyzed completely in terms of representation theory.
Keeping this in mind equation (\ref{OurCandidate}) would then seem like a good candidate expression to be extended to our infinite
dimensional case (essentially by letting $\dim\lie{p}\to\infty$ in this expression and showing that this makes sense as an infinite sum of operators).
One would then expect that defined this way, the square of the Dirac operator could be written as a sum
of various \emph{normal ordered} Casimir operators for $\hat{\lie{g}}$ and $\lie{h}$ plus possibly some extra (normal ordered) terms
originating from the central extension. The Casimir operators would then be analyzed in terms of (the highest weight) representation theory
of $\hat{\lie{g}}$ and $\lie{h}$ and hopefully the extra terms originating from the central extension would be simple enough
to be analyzed straightforwardly by hand. 

The final ingredient to add is to do everyting in a complexified setting. In particular, we
replace our Lie algebras $\hat{\lie{g}}=\Eures$ and $\lie{h}$ with their complexifications $\hat{\lie{g}}_\C=\Eglres$ and $\lie{h}^\C$.
Since $\Gr(\Hilb,\Hilb_+)$ is an infinite-dimensional Kähler manifold, we next recall 
from \cite{LaMich} that for a $2n$-dimensional Kähler manifold $(X,J,\ip{\cdot}{\cdot})$ 
with canonical Riemannian connection $\nabla$,
the Dirac operator
$$
D:\Gamma(\CCl(X))\To\Gamma(\CCl(X))
$$
associated to the complexified Clifford bundle $\CCl(X):=\Cl(TX)\tensor\C$ has a decomposition
\begin{equation}
D=\CDirac+\overline{\CDirac}
\end{equation}
such that $\CDirac$ and $\overline{\CDirac}$ are first order differential operators which are formal
adjoints of one another $(\CDirac^\ast=\overline{\CDirac})$. These operators are defined by
\begin{equation}
\CDirac\phi=\sum_j c(\varepsilon_j)\nabla_{\overline{\varepsilon}_j}\phi,\quad
\overline{\CDirac}\phi=\sum_j c(\overline{\varepsilon}_j)\nabla_{\varepsilon_j}\phi,
\end{equation}
where
\begin{equation}
\varepsilon_j=\frac{1}{2}(e_j-\im Je_j)\quad\textrm{ and }\quad\overline{\varepsilon}_j=\frac{1}{2}(e_j+\im Je_j)
\end{equation}
for any local orthonormal frame field of the form $e_1,Je_1,\ldots,e_n,Je_n$.


\subsection{The Definition}

We first introduce our Dirac operator as a formal sum of operators and then prove that it actually
defines an unbounded operator with a dense domain.

\begin{defn}[Dirac operator] As a formal element of $\End(\F\tensor\spib_{\lie{p}}^c)$, the 
Dirac operator $\KDirac$ on the restricted Grassmannian manifold $\RGr$ is defined as
\begin{eqnarray}
\KDirac &:=&
\sum_{\substack{k>0,\\l> 0}}\hat{r}(X_{l,\,-k})\tensor \gamma(X_{-k,\,l})+ 
\sum_{\substack{p>0,\\q>0}}
\hat{r}(X_{-p,\, q})\tensor \gamma(X_{q,-p})\\
&=& \frac{1}{2}\sum_{\substack{i,j\in\Z\\ij<0}} E_{ij}\tensor \gamma_{ji}\nonumber,
\end{eqnarray}
where we have used the shorthand  notation $E_{ij}=\rh(E_{ij}),\,\gamma_{ji}=\gamma(E_{ji})$. We shall also denote by
$\Psi_{ji}=\Psi(E_{ji})$ the images of $\gamma_{ji}$ inside the CAR algebra $\CAR(\lie{p}^\C,W)$.
\end{defn}

Here the the first sum in the $\rh(\cdot)$ component is over the basis of $W\isom\lie{p}$ and the second one is 
a sum over the
basis of the complex conjugate $\overline{W}\isom\lie{p}$.

\begin{prop}\label{Diracdefined}
For the basic representation $\hat{r}$ of $\Eglres(\Hilb,\Hilb_+)$ on the
fermionic Fock space $\F=\F(\Hilb,\Hilb_+)$,
the formal Dirac operator $\KDirac$ defines a well-defined unbounded symmetric linear 
operator 
$\KDirac:\F\tensor\spib_{\lie{p}}^c\To\F\tensor\spib_{\lie{p}}^c$
between two Hilbert spaces with dense domain
$\mathcal{D}(\KDirac)=\F^{\mathrm{pol}}\tensor\spib^{\mathrm{pol}}$ whose image
under $\KDirac$ is again in $\mathcal{D}(\KDirac)$.
\end{prop}
\begin{proof}
Written in terms of the field operators the Dirac operator $\KDirac$ becomes
\begin{eqnarray}\label{concreteD}
\KDirac &=&
\frac{1}{2}\sum_{\substack{k>0\\l>0}}A^\ast_{l}B^\ast_{-k}\tensor\Psi_{-k,l}+
\frac{1}{2}\sum_{\substack{p>0\\ q>0}}B_{-p}A_q\tensor \Psi^\ast_{q,-p}\nonumber\\
&=& 
\frac{1}{2}\sum_{\substack{k>0\\l>0}}A^\ast_{l}B^\ast_{-k}\tensor a(E_{-k,l})+
\frac{1}{2}\sum_{\substack{p>0\\ q>0}}B_{-p}A_q\tensor b^\ast(E_{q,-p}).
\end{eqnarray}

Notice first that each of the operators  $B_{-p}, A_q$ and $a(E_{-k,l})$ are annihilation operators associated to
an orthogonal basis vector of $\Hilb_-,\Hilb_+$ and $\lie{p}^\C$, respectively. 

Fix $\psi\in\FPol$. Then according to Corollary \ref{finitepol} there exists only \emph{finitely} many indices $q\in\Z$
such that $\psi_{(q)}:=A_q\cdot\psi\neq 0$. Notice that in any case $\psi_{(q)}\in\FPol$ for all $l\in\Z$. 
For simplicity we may assume that $\psi_{(q)}$ is a monomial. Now for
each $\psi_{(q)}\neq 0$ there exists again only finitely many indices $p$ for which
$B_{-p}\cdot\psi_{(q)}\neq 0$. 
Hence there are only finitely many pairs $(p,q)\in\Z\times\Z$ for which $(B_{-p}A_q)\cdot\psi\neq 0$ and
we conclude that the sum $\sum_{p,q}B_{-p}A_q\tensor b^\ast(E_{q,-p})$ on the right hand side of
(\ref{concreteD}) is a well-defined operator on $\FPol\tensor\spib^{\mathrm{pol}}$.

On the other hand, again according to the Corollary \ref{finitepol}, for a fixed $s\in\spib^{\mathrm{pol}}$ there
exists only finitely many pairs of indices $(k,l)\in\Z\times\Z$ for which $a(E_{-k,l})\cdot s\neq 0$ showing
that the sum $\sum_{k,l}A^\ast_{l}B_{-k}^\ast\tensor a(E_{-k,l})$ on the left hand side of
equation (\ref{concreteD}) is a well-defined operator on
$\FPol\tensor\spib^{\mathrm{pol}}$.

The statement concerning the denseness property is evident since by its very definition a Fock space is a completion of
its purely algebraic counterpart.

The symmetricity part follows from the explicit description (\ref{concreteD}) and the adjoint properties of
the various field operators, e.g. the operator $\Psi^\ast_{-k,l}$ is an adjoint operator of  $\Psi_{-k,l}$ etc.
\end{proof}

\begin{cor}
The square $\KDirac^2=\KDirac\circ\KDirac:\F\tensor\spib_{\lie{p}}^c\To\F\tensor\spib_{\lie{p}}^c$ is well-defined as an unbounded linear
operator with dense domain $\D(\KDirac^2)=\D(\KDirac)=\F^{\mathrm{pol}}\tensor\spib^{\mathrm{pol}}$.
\end{cor}

\begin{cor}\label{TenVacInKer}
The kernel of the Dirac operator satisfies $\C\vac\subset\ker(\KDirac)$, where $\vac=\vac_\F\tensor\vac_\spib$. 
\end{cor}
\begin{proof}
It follows directly from Proposition \ref{vacuumstr} that $\vac=\vac_\F\tensor\vac_\spib\in\ker(\KDirac)$.
Since
for a well-defined operator, the kernel is always a subspace the claim follows.
\end{proof}

\begin{lem}\label{Dcommutes}
The Dirac operator $\KDirac$ on $\RGr\isom\Ures/(U(\Hilb_+)\times U(\Hilb_-))=G/H$
commutes with the right action
\begin{equation}
\varrho(Z):=\hat{r}(Z)\tensor 1+1\tensor\adt Z
=\hat{r}(Z)\tensor 1+1\tensor\ad Z
\in
\End(\F\tensor\spib_{\lie{p}}^c)
\end{equation}
for all $Z\in\hinf^\C$.
\end{lem}
\begin{proof}
By linearity, we may assume that $Z=E_{kl}$ with $kl>0$, a basis vector of $\hinf^\C$.
We want to show that
$
[\varrho(Z),\KDirac]\varphi=[\varrho(E_{kl}),\sum_{\substack{i,j\in\Z\\ij<0}}E_{ij}\tensor \gamma_{ji}]\varphi=0
$
for all $\varphi\in\domain(\KDirac)$. Using Corollary \ref{isotropycorollary} and the fact
that the representation of $\Eglres$ on the fermionic Fock space $\F$ was given in
terms of CAR algebra generators, one sees that
$
\varrho(Z)\in\End(\F^{\mathrm{pol}}\tensor\spib^{\mathrm{pol}})=\End(\domain(\KDirac))
$
for all $Z\in\hinf^\C$
so that it makes sense to take the commutator and the image satisfies
$$
[\varrho(Z),\KDirac]\varphi=\varrho(Z)(\KDirac\varphi)-\KDirac(\varrho(Z)\varphi)\in\domain(\KDirac).
$$
Moreover, according to the proof of Proposition \ref{Diracdefined}, for each element
$\varphi\in\domain(\KDirac)$ the
sum $(\sum_{ij<0}E_{ij}\tensor \gamma_{ji})\varphi$ is actually finite:
$$
(\sum_{\substack{i,j\in\Z \\ ij<0}}E_{ij}\tensor \gamma_{ji})\varphi=
(\sum_{\substack{i,j\in\Z \\ ij<0, \abs{i},\abs{j}\leq N}}E_{ij}\tensor \gamma_{ji})\varphi
$$
for some $N\in\N$ depending on $\varphi$. This implies that pointwise we are free to change the
summation order any way we want. This allows us to prove the claim by the following
formal computation:

\begin{eqnarray}
[\varrho(Z),\sum_{\substack{i,j\in\Z\\ij<0}}E_{ij}\tensor \gamma_{ji}]&=&
\sum_{\substack{i,j\in\Z\\ij<0}}[Z,E_{ij}]\tensor\gamma_{ji}+E_{ij}\tensor[\adt Z,\gamma_{ji}]\nonumber\\
&\stackrel{\textrm{Corol. }\ref{isotropycorollary}}{=}&
\sum_{\substack{i,j\in\Z\\ij<0}}[E_{kl},E_{ij}]\tensor\gamma_{ji}+E_{ij}\tensor[\Kt{kl},\gamma_{ji}]\nonumber\\
&\stackrel{\textrm{Lemma } \ref{Kijlimit}}{=}&
\sum_{\substack{i,j\in\Z\\ij<0}}[E_{kl},E_{ij}]\tensor\gamma(E_{ji})+E_{ij}\tensor\gamma\Big([\ad E_{kl}](E_{ji})\Big)\nonumber\\
&=&
\sum_{\substack{i,j\in\Z\\ij<0}}[E_{kl},E_{ij}]\tensor\gamma(E_{ji})+ E_{ij}\tensor\gamma\Big([E_{kl},E_{ji}]\Big)\nonumber
\end{eqnarray}
This simplifies to
\begin{eqnarray}
& &\sum_{\substack{i,j\in\Z\\ij<0}}[E_{kl},E_{ij}]\tensor\gamma(E_{ji})
+ 
\sum_{\substack{i,j\in\Z\\ij<0}}E_{ij}\tensor\gamma\Big([E_{kl},E_{ji}]\Big) \nonumber\\
&=&
\sum_{\substack{i,j\in\Z\\ij<0}}\Big(\delta_{il}E_{kj}-\delta_{jk}E_{il}\Big)\tensor\gamma(E_{ji})
+
\sum_{\substack{i,j\in\Z\\ij<0}} E_{ij}\tensor\gamma\Big(\delta_{jl}E_{ki}-\delta_{ik}E_{jl}\Big)\nonumber\\
&=&
\sum_{\substack{i,j\in\Z\\ij<0}}
\delta_{il}E_{kj}\tensor\gamma(E_{ji})
-
\sum_{\substack{i,j\in\Z\\ij<0}}
\delta_{jk}E_{il}\tensor\gamma(E_{ji})
\nonumber\\
&+&
\sum_{\substack{i,j\in\Z\\ij<0}}
E_{ij}\tensor\delta_{jl}\gamma(E_{ki})
-
\sum_{\substack{i,j\in\Z\\ij<0}}
E_{ij}\tensor\delta_{ik}\gamma(E_{jl}).
\nonumber
\end{eqnarray}
Now write the first infinite sum appearing above as
\begin{equation}\label{com1}
S_1:=\sum_{\substack{i,j\in\Z\\ij<0}}
\delta_{il}E_{kj}\tensor\gamma(E_{ji})=
\sum_{\substack{i>0 \\j<0}}\delta_{il}E_{kj}\tensor\gamma(E_{ji})+
\sum_{\substack{i<0 \\ j>0}}\delta_{il}E_{kj}\tensor\gamma(E_{ji})
\end{equation}
and the fourth sum term the same way:
\begin{equation}\label{com2}
S_4:=\sum_{\substack{i,j\in\Z\\ij<0}}
E_{ij}\tensor\delta_{ik}\gamma(E_{jl})=\sum_{\substack{i>0 \\j<0}}\delta_{ik}E_{ij}\tensor\gamma(E_{jl})
+\sum_{\substack{i<0 \\ j>0}}\delta_{ik}E_{ij}\tensor\gamma(E_{jl}).
\end{equation}

Notice that trivially $\delta_{ab}=0$ if $a\in\Z$ and $b\in\Z$ have different signs and
that a necessary condition to allow $\delta_{ab}$ the possibility to be equal to $1$ is that $a$ and $b$ must 
have the same sign. Recalling that $k$ and $l$ have the same sign
it follows from (\ref{com1}) and (\ref{com2}) that $S_1-S_4=0$: For example if
$k,l>0$
$$
S_1-S_4=\sum_{j<0}E_{kj}\tensor\gamma(E_{jl})-\sum_{j<0}E_{kj}\tensor\gamma(E_{jl})=0.
$$
Similarly one sees that
$$
S_3-S_2=\sum_{\substack{i,j\in\Z\\ij<0}}
E_{ij}\tensor\delta_{jl}\gamma(E_{ki})-
\sum_{\substack{i,j\in\Z\\ij<0}}
\delta_{jk}E_{il}\tensor\gamma(E_{ji})=0
$$
so that
$$
[\varrho(Z),\sum_{\substack{i,j\in\Z\\ij<0}}E_{ij}\tensor \gamma_{ji}]=
S_1-S_4+S_3-S_2=0
$$
from which the claim follows.
\end{proof}

\begin{cor}
For the basic representation $\hat{r}:\Eglres\To\End(\F)$ the Dirac operator descends to
\begin{equation}
\KDirac:\Big(\F\tensor\spib_{\lie{p}}^c\Big)^{\lie{h}}\To
\Big(\F\tensor\spib_{\lie{p}}^c\Big)^{\lie{h}}
\end{equation}
where it is a well-defined unbounded symmetric linear operator with dense domain
$\mathcal{D}(\KDirac)=(\F^{\mathrm{pol}}\tensor\spib^{\mathrm{pol}})^{\lie{h}}$.
\end{cor}
\begin{proof}
This follows from Lemma \ref{Dcommutes} and Proposition \ref{Diracdefined}.
\end{proof}

Just as it holds in the finite-dimensional case, we would 
like to show that $\KDirac$ is besides a symmetric operator, it also has
a finite-dimensional kernel. 
To deal with this question, it turns out to be a good idea to study the square $\KDirac^2$ of the
Dirac operator instead of trying to look directly at $\KDirac$ which is much harder. 
\subsection{The square $\KDirac^2$}\label{sectionNorD2}
\subsubsection{Introducing cut-offs}
Since clearly $\ker\KDirac\subset\ker\KDirac^2$ , in order to show that
$\dim\ker\KDirac<\infty$ it is sufficient to show that $\dim\ker\KDirac^2<\infty$.
Our strategy to approach the kernel of the square $\KDirac^2$ will be to give a concrete
diagonalization for $\KDirac^2$ from which we can easily deduce explicitly what the
kernel $\ker\KDirac^2$ is. 

The diagonalization is made possible by introducing and analyzing a cut-off operator 
$\KDirac_{(N)}^2$
for each $N\in\N$, and then showing that
the limit of these cut-offs as $N\to\infty$, in the strong sense, is the original square $\KDirac^2$. The cut-off is a \emph{finite} sum of operators, allowing
us a better control in handling the various diverging polynomial terms in the variable $N$ appearing in
the algebraic expression for $\KDirac_{(N)}^2$, that ultimately are ought to annihilate each other, since we already
know from the previous section that $\KDirac^2$ is a well-defined operator.

We use the shorthand notation 
$$
\KDirac=\sum_{\substack{i,j\in\Z\\ij<0}}\rh(X_{ij})\tensor\gamma(X_{ji})=\frac{1}{2}\sum_{ij<0}\rh(E_{ij})\tensor\gamma(E_{ji})=
\frac{1}{2}\sum_{ij<0}E_{ij}\tensor\gamma_{ji}
$$
in the calculations that will follow.

\begin{defn}
For each $N\in\N$, introduce the cut-off 
$\KDirac_{(N)}:\F\tensor\spib_{\lie{p}}^c\To\F\tensor\spib_{\lie{p}}^c$,
$$
\KDirac_{(N)}:=\frac{1}{2}\sum_{\substack{ij<0\\ \abs{i},\abs{j}\leq N}} E_{ij}\tensor\gamma_{ji},
\qquad\D(\KDirac_{(N)})=(\F^{\mathrm{pol}}\tensor\spib^{\mathrm{pol}})^{\lie{h}}
$$
which is a \emph{finite} sum of operators.
\end{defn}

The following Proposition tells us in which sense $\KDirac_{(N)}$ is an
approximation of $\KDirac$.

\begin{prop}\label{KDiracNvsKDirac}
The \emph{strong limit}
$\slim_{N\to\infty}\KDirac_{(N)}=\KDirac$, i.e.
$$
\lim_{N\to\infty}\norm{\KDirac_{(N)}\varphi-\KDirac\varphi}_{\F\tensor\spib_{\lie{p}}^c}=0
$$ 
for all $\varphi\in\D(\KDirac)=\bigcap_{N\in\N}\D(\KDirac_{(N)})$. Here the norm $\norm{\cdot}$
is induced by the Hilbert space inner product on the tensor product
Hilbert space $\F\tensor\spib_{\lie{p}}^c$.
\end{prop}
\begin{proof}
Let $\varphi\in\bigcap_N\D(\KDirac_{(N)})=\FPol\tensor\spib^{\mathrm{pol}}$ be arbitrary. 
Then according to
the proof of Proposition \ref{Diracdefined}
there exists $N_\varphi\in\N$ such that
$$
\KDirac_{(N)}\varphi=\KDirac\varphi
$$
for all $N\geq N_\varphi$, from which the claim follows.
\end{proof}

\subsubsection{Invariant sectors}
Looking at the proof of Lemma \ref{Dcommutes} one sees that $\KDirac_{(N)}$ descends to an operator
$$
\KDirac_{(N)}:(\F\tensor\spib_{\lie{p}}^c)^{\lie{h},N}\To(\F\tensor\spib_{\lie{p}}^c)^{\lie{h},N},
$$
where
\begin{eqnarray}
(\F\tensor\spib_{\lie{p}}^c)^{\lie{h},N} 
&:=&
\Big\{w\in \F\tensor\spib_{\lie{p}}^c\,\Big\vert\,\Big[\rh(E_{pq})\tensor 1\Big](w)=\Big[-1\tensor \KN{pq}\Big](w)
\quad\textrm{ for all }\nonumber\\
& & p,q\in\Z\textrm{ such that }pq>0\textrm{ and } \abs{p},\abs{q}\leq N\Big\}.
\end{eqnarray}
For each $N\in\N$ the operator $\KDirac_{(N)}$ has a dense domain
$(\F^{\mathrm{pol}}\tensor\spib^{\mathrm{pol}})^{\lie{h},N}$ which is defined in the
same manner as its previous analogues.

It is then clear that the square $\KDirac_{(N)}^2$
is an operator
$$
\KDirac_{(N)}^2:(\F\tensor\spib_{\lie{p}}^c)^{\lie{h},N}\To(\F\tensor\spib_{\lie{p}}^c)^{\lie{h},N}.
$$

\begin{prop}\label{StroLimInEqSector}
Consider the Dirac operator $\KDirac$ as an unbounded operator, $\KDirac:(\F\tensor\spib_{\lie{p}}^c)^{\lie{h}}\To
(\F\tensor\spib_{\lie{p}}^c)^{\lie{h}}$.
Then for each $\varphi\in\domain(\KDirac)$
there exists
$N_\varphi\in\N$ such that for all $N\geq N_\varphi$
\begin{enumerate}
\item $\varphi\in\domain(\KDirac_{(N)})$;
\item $\KDirac_{(N)}\varphi=\KDirac\varphi$;
\item The image $\KDirac_{(N)}\varphi\in\domain(\KDirac)\cap\domain(\KDirac_{(N)})$, so that
the compositions $\KDirac^2\varphi$ and $\KDirac_{(N)}^2\varphi$ are both defined and their
values coincide.
\end{enumerate}
\end{prop}
\begin{proof}
\begin{enumerate}
\item
Let $\varphi\in\domain(\KDirac)=(\F\tensor\spib_{\lie{p}}^c)^{\lie{h}}$,
so that by definition $\varphi$ satisfies
\begin{equation}\label{FromNtoNonN}
\Big[\rh(E_{pq})\tensor 1\Big](w)=\Big[-1\tensor \Kt{pq}\Big](w)
\end{equation}
for all
$p,q\in\Z$ with $pq>0$. 
Next recall that according to Proposition \ref{KijStrongLimit} for all
$p,q\in\Z$ with $pq>0$
$$
\stlim{N}\KN{pq}=\Kt{pq},
$$
uniformly in the pair $(p,q)\in\Z\times\Z$. Thus there exists
an integer $N_\varphi\in\Z_{\geq 0}$ such that
$$
[1\tensor\KN{pq}]w=[1\tensor\Kt{pq}]w\quad\textrm{ for all $p,q\in\Z,\,pq>0$}
$$
whenever $N\geq N_\varphi$.
It follows then from equation (\ref{FromNtoNonN}) that
$$
\Big[\rh(E_{pq})\tensor 1\Big](w)=\Big[-1\tensor \KN{pq}\Big](w)\quad
\textrm{ \emph{for all} $p,q\in\Z,\,pq>0$}
$$
whenever $N\geq N_\varphi$. In particular
$\varphi\in\domain(\KDirac_{(N)})$ for all $N\geq N_\varphi$.
\item This follows directly from Proposition \ref{KDiracNvsKDirac}.
\item Clear from the part $(2)$ above since the images
$\mathrm{im}(\KDirac)\subset\domain(\KDirac)$ and
$\mathrm{im}(\KDiracc)\subset\domain(\KDiracc)$.
\end{enumerate}
\end{proof}

\subsubsection{A rough computation}
We start by giving a very brutal looking expression for the square of the cut-off Dirac operator 
$\KDirac_{(N)}$ and
then proceed to analyze and normal order the various terms in it step by step in order to
end up having as simple expression as possible for the square, i.e. an operator which
we can actually understand. 

\begin{prop}\label{squareN}
The square of the $N^{\textrm{th}}$ cut-off Dirac operator 
$\KDiracc: \F\tensor\spib_{\lie{p}}^c\To
\F\tensor\spib_{\lie{p}}^c$ is given explicitly by
\begin{eqnarray}
\KDiracc^{\,2}&=&
\frac{1}{16}\sum_{\substack{ij,mn<0\\ \abs{i},\abs{j},\abs{m},\abs{n}\leq N}}
\Big(\delta_{jm}E_{in}-\delta_{ni}E_{mj}\Big)\tensor[\gamma_{ji},\gamma_{nm}]
\nonumber\\
&+&
\frac{1}{8}\sum_{\substack{ij<0\\ \abs{i},\abs{j}\leq N}} 1\tensor 
\sign(i)\gamma_{ij}\gamma_{ji}
+\frac{1}{4}\sum_{\substack{ij<0\\ \abs{i},\abs{j}\leq N}}
E_{ij}E_{ji}\tensor 1\nonumber
\end{eqnarray}
\end{prop}

\begin{proof}
First note that in general one may
write the product $ST$ of two operators in terms of a commutator and an anti-commutator
as
\begin{equation}
ST=\frac{1}{2}\left([S,T]+\{S,T\}\right).
\end{equation}
Applied to the cut-off square $\KDiracc^2$ this gives
\begin{eqnarray}
\KDiracc^2 &=& \frac{1}{4}\sum_{\substack{ij,mn<0\\ \abs{i},\abs{j},\abs{m},\abs{n}\leq N}} E_{ij}E_{mn}\tensor \gamma_{ji}\gamma_{nm}\nonumber\\
&=&
\frac{1}{16}\sum_{\substack{ij,mn<0\\\abs{i},\abs{j},\abs{m},\abs{n}\leq N}}
[E_{ij},E_{mn}]\tensor[\gamma_{ji},\gamma_{nm}]\nonumber\\
&+&
\frac{1}{16}\sum_{\substack{ij,mn<0\\ \abs{i},\abs{j},\abs{m},\abs{n}\leq N}}
\{E_{ij},E_{mn}\}\tensor\{\gamma_{ji},\gamma_{nm}\}\nonumber\\
&=&
\frac{1}{16}\sum_{\substack{ij,mn<0\\ \abs{i},\abs{j},\abs{m},\abs{n}\leq N}}
\Big(\delta_{jm}E_{in}-\delta_{ni}E_{mj}+s(E_{ij},E_{mn})\Big)\tensor[\gamma_{ji},\gamma_{nm}]
\nonumber\\
&+&
\frac{1}{16}\sum_{\substack{ij,mn<0\\ \abs{i},\abs{j},\abs{m},\abs{n}\leq N}}
\{E_{ij},E_{mn}\}\tensor 2\delta_{jm}\delta_{in}\nonumber
\end{eqnarray}
Next recall that
\begin{eqnarray}
s(E_{ij},E_{ji})&=&-s(E_{ji},E_{ij})=1,\quad\textrm{ if } i<0, j>0,\nonumber\\
s(E_{ij},E_{mn})&=&0\quad\textrm {in all other cases}.
\end{eqnarray}
so that
\begin{eqnarray}
\KDiracc^2 &=&
\frac{1}{16}\sum_{\substack{ij,mn<0\\\abs{i},\abs{j},\abs{m},\abs{n}\leq N}}
\Big(\delta_{jm}E_{in}-\delta_{ni}E_{mj}-\sign(i)\delta_{in}\delta_{jm}\cdot 1\Big)\tensor[\gamma_{ji},\gamma_{nm}]
\nonumber\\
&+&
\frac{1}{8}\sum_{\substack{ij<0\\ \abs{i},\abs{j}\leq N}}
\{E_{ij},E_{ji}\}\tensor 1\nonumber\\
&=&
\frac{1}{16}\sum_{\substack{ij,mn<0\\ \abs{i},\abs{j},\abs{m},\abs{n}\leq N}}
\Big(\delta_{jm}E_{in}-\delta_{ni}E_{mj}\Big)\tensor[\gamma_{ji},\gamma_{nm}]
\nonumber\\
&-&
\frac{1}{16}\sum_{\substack{ij,mn<0\\ \abs{i},\abs{j},\abs{m},\abs{n}\leq N}}1\tensor \sign(i)\delta_{in}\delta_{jm}\cdot[\gamma_{ji},\gamma_{nm}]
+\frac{1}{4}\sum_{\substack{ij<0\\ \abs{i},\abs{j}\leq N}}
E_{ij}E_{ji}\tensor 1
\nonumber\\
&=&
\frac{1}{16}\sum_{\substack{ij,mn<0\\ \abs{i},\abs{j},\abs{m},\abs{n}\leq N}}
\Big(\delta_{jm}E_{in}-\delta_{ni}E_{mj}\Big)\tensor[\gamma_{ji},\gamma_{nm}]
\nonumber\\
&-&
\frac{1}{16}\sum_{\substack{ij<0\\ \abs{i},\abs{j}\leq N}} 1\tensor \sign(i)\cdot[\gamma_{ji},\gamma_{ij}]
+\frac{1}{4}\sum_{\substack{ij<0\\ \abs{i},\abs{j}\leq N}}
E_{ij}E_{ji}\tensor 1\nonumber
\end{eqnarray}
\end{proof}

\begin{lem}\label{commutatorANDcommutator}
As an operator $\F\tensor\spib_{\lie{p}}^c\To\F\tensor\spib_{\lie{p}}^c$
$$
\frac{1}{16}\sum_{\substack{ij,mn<0\\ \abs{i},\abs{j},\abs{m},\abs{n}\leq N}}
\Big(\delta_{jm}E_{in}-\delta_{ni}E_{mj}\Big)\tensor[\gamma_{ji},\gamma_{nm}]=
-\frac{1}{8}\sum_{\substack{ij,\,in<0\\(jn>0)}} E_{nj}\tensor[\gamma_{ji},\gamma_{in}].
$$
\end{lem}
\begin{proof}
During the proof we shall always assume that all the indices satisfy 
$\abs{i},\abs{j},\abs{m}$, $\abs{n}\leq N$ 
and in order to simplify notation we shall drop this condition off from the sum terms. Now let
\begin{eqnarray}\label{Sterm}
S &:=& \frac{1}{16}\sum_{\substack{ij,mn<0\\ \abs{i},\abs{j},\abs{m},\abs{n}\leq N}}
(\delta_{jm}E_{in}-\delta_{in}E_{mj})\tensor[\gamma_{ji},\gamma_{nm}]\nonumber\\
&=&
\frac{1}{16}\sum_{\substack{ij<0\\jn<0}}E_{in}\tensor[\gamma_{ji},\gamma_{nj}]-
\frac{1}{16}\sum_{\substack{ij<0 \\ im<0}} E_{mj}\tensor[\gamma_{ji},\gamma_{im}]\nonumber\\
&=&
-\frac{1}{16}\sum_{\substack{ij<0\\jn<0}}E_{in}\tensor[\gamma_{nj},\gamma_{ji}]-
\frac{1}{16}\sum_{\substack{ij<0\\in<0}}E_{nj}\tensor[\gamma_{ji},\gamma_{in}].
\end{eqnarray}
Notice that the conditions $ji<0,jn<0$ imply that $in>0$ and similarly one sees that
necessarily $mj>0$ so that $E_{in},E_{mj}\in\lie{h}^\C$.

Now the first and second term in (\ref{Sterm}) look very similar and they are
indeed equal. This can be seen by permuting indices: In the sum above
$-\frac{1}{16}\sum_{\substack{ij<0\\jn<0}}E_{in}\tensor[\gamma_{nj},\gamma_{ji}]$,
first switch the indices $i\leftrightarrow j$ and after that
switch the indices $n\leftrightarrow j$, from which the claim follows.
\end{proof}

\subsubsection{Link between $\KDirac^2_{(N)}$ and various Casimir operators}
Next introduce the naive $N^{\mathrm{th}}$ cut-off Casimir operator of $\Eglres$,
$$
\Delta_{\lie{g}}^{(N)}:=\sum_{\substack{i,j\in\Z\\ i,j\neq 0\\ \abs{i},\abs{j}\leq N}}E_{ij}E_{ji}
$$
and the naive diagonal Casimir operator of $\hinf^{\C}$ by
\begin{equation}\label{NaiveDiagCas}
\Delta_{\lie{h},\mathrm{diag}}^{(N)}=\sum_{\substack{ij>0 \\ \abs{i},\abs{j}\leq N}}E_{ij}E_{ji}\tensor 1
+2\sum_{\substack{ij>0 \\ \abs{i},\abs{j}\leq N}}E_{ij}\tensor\Hn{ji}
+\sum_{\substack{ij>0 \\ \abs{i},\abs{j}\leq N}}1\tensor\Hn{ij}\Hn{ji}.
\end{equation}
Then the following holds

\begin{prop}\label{PropAnaWitHuanPand}
The square $\KDirac_{(N)}^2:\F\tensor\spib_{\lie{p}}^c\To\F\tensor\spib_{\lie{p}}^c$ of the $N^{\mathrm{th}}$ cut-off Dirac operator is given by
\begin{equation}
4\KDirac_{(N)}^2=\Big(\Delta_{\lie{g}}^{(N)}\tensor 1-\Delta_{\lie{h},\mathrm{diag}}^{(N)}+\sum_{\substack{ij>0 \\ \abs{i},\abs{j}\leq N}}1\tensor\Hn{ij}\Hn{ji}\Big)+
1\tensor\sum_{\substack{\abs{i}\leq N\\ i\neq 0}}\sign(i)\Hn{ii}.\nonumber
\end{equation}
\end{prop}
\begin{proof}
To simplify notations we shall assume during the proof that the indices $i,j,m,n\in\Z$ appearing in various sums always satisfy $\abs{i},\abs{j},\abs{m},\abs{n}\leq N$. 

By Lemma \ref{commutatorANDcommutator} we may write
\begin{eqnarray}
4\KDirac_{(N)}^2=-\frac{1}{2}\sum_{\substack{ij,in<0\\ (jn>0)}}E_{nj}\tensor[\gamma_{ji},\gamma_{in}]+
\sum_{ij<0} E_{ij}E_{ji}-
\frac{1}{4}\sum_{ij,mn<0}1\tensor\sign(i)[\gamma_{ji},\gamma_{ij}]\nonumber.
\end{eqnarray}
Now write
\begin{eqnarray}
-\frac{1}{2}\sum_{\substack{ij,in<0\\ (jn>0)}}E_{nj}\tensor[\gamma_{ji},\gamma_{in}]&=&
-2\sum_{\substack{j,n \\ jn>0}}\Big(E_{nj}\tensor\frac{1}{4}\sum_{\substack{i \\ ij,in<0}}[\gamma_{ji},\gamma_{in}]\Big)\nonumber\\
&=&
-2\sum_{\substack{jn>0}}E_{nj}\tensor\Hn{jn}\nonumber
\end{eqnarray}
and similarly, noticing that for a pair $ij<0$, we have $\sign(i)=-\sign(j)$, hence
\begin{eqnarray}
-\frac{1}{4}\sum_{ij,mn<0}1\tensor\sign(i)[\gamma_{ji},\gamma_{ij}] &=&
1\tensor\sum_{\substack{j\neq 0}}\Big(\sign(j)\frac{1}{4}\sum_{\substack{i \\ ij<0}}[\gamma_{ji},\gamma_{ij}]\Big)\nonumber\\
&=& 1\tensor\sum_{\substack{j\neq 0}}\sign(j)\Hn{jj}\nonumber
\end{eqnarray}
so that
$$
4\KDirac_{(N)}^2=-2\sum_{\substack{jn>0}}E_{nj}\tensor\Hn{jn}+\sum_{ij<0} E_{ij}E_{ji}+
1\tensor\sum_{\substack{j\neq 0}}\sign(j)\Hn{jj}
$$
from which the result follows since
$$
-2\sum_{\substack{jn>0}}E_{nj}\tensor\Hn{jn}+\sum_{ij<0} E_{ij}E_{ji}=
\Delta_{\lie{g}}^{(N)}\tensor 1-\Delta_{\varrho}^{(N)}+\sum_{\substack{ij>0 \\ \abs{i},\abs{j}\leq N}}1\tensor\Hn{ij}\Hn{ji}.
$$
\end{proof}
\begin{rem}
The reader should compare the expression inside the parenthesis in Proposition \ref{PropAnaWitHuanPand} with
Proposition 3.1.6 in \cite{HuaPan}, where the square of a `finite-dimensional' algebraic Dirac operator introduced by Parthasarathy
is computed.
\end{rem}
\subsubsection{Introducing normal ordering for $\Delta_{\lie{g}}^{(N)}$}
Next set
\begin{equation}
\Delta_{\lie{g},\mathrm{ren}}^{(N)}:=\Delta_{\lie{g}}^{(N)}-N^2.
\end{equation}
Then we have the following result.
\begin{prop}
One may write
\begin{equation}\label{GrenCasForAinf}
\Delta_{\lie{g},\mathrm{ren}}^{(N)}=2\sum_{\substack{j<i\\ i,j\neq 0 \\ \abs{i},\abs{j}\leq N}}E_{ij}E_{ji}+
\sum_{\substack{\abs{i}\leq N \\ i\neq 0}}E_{ii}\Big(E_{ii}-2i+\sign(i)\cdot 1\Big).
\end{equation}
\end{prop}
\begin{proof}
The claim follows easily from the results proved in the beginning of Appendix C.
\end{proof}

By the same reasoning as with the operator $\Delta^{(N)}$ in Appendix C, the operator $\Delta_{\lie{g},\mathrm{ren}}^{(N)}$ is well-defined
in the `limit' $N\to\infty$ and this limit
behaves similarly to $\Delta$ when realized on the highest weight representation spaces $\F_{k}\subset\F$ of the fundamental representations of $\ainf\subset\Eglres$. 
More precisely, we have the following Proposition, whose proof is identical to the proofs of corresponding statements in Appendix C.

\begin{prop}\label{ConclusiveDeltaG}
Denote by $\Delta_{\lie{g}}$ the `limit' of $\Delta_{\lie{g},\mathrm{ren}}^{(N)}$ as $N\to\infty$, which
we consider as an unbounded operator $\F_0\To\F_0$ with dense domain $\D(\Delta_{\lie{g}})=\F_0^{\mathrm{pol}}$, i.e. we set
$$
\Delta_{\lie{g}}:=2\sum_{\substack{i,j\in\Z \\j<i\\ i,j\neq 0}}E_{ij}E_{ji}+
\sum_{\substack{i\in\Z \\ i\neq 0}}E_{ii}\Big(E_{ii}-2i+\sign(i)\cdot 1\Big).
$$ 
Then the
following holds.
\begin{enumerate}
\item The unbounded operators $\Delta_{\lie{g}},\Delta_{\lie{g},\mathrm{ren}}^{(N)}:\F_0\To\F_0$ with common dense
domain $\F_0^{\mathrm{pol}}$ satisfy the strong limit condition
$$
\stlim{N}\Delta_{\lie{g},\mathrm{ren}}^{(N)}=\Delta_{\lie{g}}.
$$
\item The restriction
$$
\Delta_{\lie{g}}\Big\vert_{\F_0^{\mathrm{pol},M}}=2M\id_{\F_0^{\mathrm{pol},M}},
$$
where we have used the decomposition
$$
\F_0^{\mathrm{pol}}=\bigoplus_{M=0}^\infty \F_0^{\mathrm{pol},M}\quad
\textrm{(algebraic direct sum, orthogonal decomposition)}
$$
occuring in Appendix C right below Corollary \ref{DiagonalizingCasOfG}.
\item It follows from the above that $\Delta_{\lie{g}}$  is diagonalizable with spectrum equal to
$$
\spec(\Delta_{\lie{g}})=2\Z_{\geq 0}.
$$
\item The kernel of $\Delta_{\lie{g}}$ satisfies
$$
\ker(\Delta_{\lie{g}})=\F_0^{\mathrm{pol},0}=\C\vac_\F
$$
so that in particular $\dim\ker(\Delta_{\lie{g}})=1<\infty$.
\end{enumerate}
\end{prop}

Hence we may now write

\begin{eqnarray}
4\KDirac_{(N)}^2 &=&\Big(\Delta_{\lie{g},\mathrm{ren}}^{(N)}\tensor 1-\Delta_{\lie{h},\mathrm{diag}}^{(N)}+\sum_{\substack{ij>0 \\ \abs{i},\abs{j}\leq N}}1\tensor\Hn{ij}\Hn{ji}\Big)+
1\tensor\sum_{\substack{\abs{i}\leq N\\ i\neq 0}}\sign(i)\Hn{ii}\nonumber\\
&+& N^2.\nonumber
\end{eqnarray}
where $\Delta_{\lie{g},\mathrm{ren}}^{(N)}$ is now well-defined as $N\to\infty$, but the other terms in this expression still need to be properly
normal ordered.

\subsubsection{Normal ordering for the diagonal $\hinf^\C$-Casimir operator}
To introduce a valid normal ordering for all the relevant summands of $\KDirac^2_{(N)}$ one has to know how the different operators
$K_{ij}^{(N)},\KN{ij}$ and $\Hn{ij}$ are linked to each other. Writing $[\gamma_{ik},\gamma_{kj}]=2\gamma_{ik}\gamma_{kj}-2\delta_{ij}\cdot I$ one sees
immediately that

\begin{eqnarray}
\Hn{ij} &=& K_{ij}^{(N)}-\frac{1}{2}\delta_{ij}N=\KN{ij}-\frac{1}{2}\sign(i)\delta_{ij}N
\label{KijHij1}\\
K_{ij}^{(N)} &=& \left\{
\begin{array}{ll}
\KN{ij} & \text{if } i\neq j \text{ or } i=j> 0 \\
\KN{ij}+N\cdot 1 & \text{if } i = j < 0.\label{KijHij2}
\end{array} \right.
\end{eqnarray}

\begin{lem} On the $N^{\mathrm{th}}$ level $\hinf^\C$-invariant sector, i.e. as an operator
$$
(\F\tensor\spib_{\lie{p}}^c)^{\lie{h},N}\To\F\tensor\spib_{\lie{p}}^c
$$
the naive diagonal Casimir operator of $\hinf^\C$ becomes the constant operator
$$
\Delta_{\lie{h},\mathrm{diag}}^{(N)}=\frac{1}{2}N^3.
$$
\end{lem}
\begin{proof}
We want to use the fact that the diagonal Casimir operator $\Delta_{\varrho}^{(N)}$
introduced in \S\ref{diagonalCasOp} is identically zero in the
$N^{\mathrm{th}}$ level $\hinf^\C$-invariant sector. For this reason we want to switch
the operators $\Hn{ij}$ appearing in the definition of $\Delta_{\lie{h},\mathrm{diag}}^{(N)}$
to the operators $\KN{ij}$ one sees in the
expression for $\Delta_{\varrho}^{(N)}$. 

To begin with, equation (\ref{KijHij1}) yields
\begin{eqnarray}
2\sum_{\substack{ij>0\\ \abs{i},\abs{j}\leq N}}E_{ij}\tensor \Hn{ji} &=&
2\sum_{\substack{ij>0 \\ \abs{i},\abs{j}\leq N}}E_{ij}\tensor
(\KN{ij}-\frac{1}{2}\sign(i)\delta_{ij}N)\nonumber\\
&=&
2\sum_{\substack{ij>0\\ \abs{i},\abs{j}\leq N}}E_{ij}\tensor\KN{ij}-
N\sum_{\substack{\abs{i}\leq N \\ i\neq 0}}\sign(i)E_{ii}\tensor 1. \nonumber
\end{eqnarray}
and similarly noting that here $\sign(i)=\sign(j)$ one computes using again equation (\ref{KijHij1})
$$
\sum_{\substack{ij>0\\ \abs{i},\abs{j}\leq N}}\Hn{ij}\Hn{ji}=
\sum_{\substack{ij>0\\ \abs{i},\abs{j}\leq N}}\KN{ij}\KN{ji}-
N\sum_{\substack{\abs{i}\leq N\\ i\neq 0}}\sign(i)\KN{ii}+\frac{1}{2}N^3.
$$

Thus, we may now relate the naive diagonal Casimir operator
$\Delta_{\lie{h},\mathrm{diag}}^{(N)}$
given by (\ref{NaiveDiagCas})
with the diagonal Casimir operator
$\Delta_{\varrho}^{(N)}$ appearing in equation (\ref{diagonalCasimiroperator}):
$$
\Delta_{\lie{h},\mathrm{diag}}^{(N)}=\Delta_{\varrho}^{(N)}-
N\sum_{\substack{\abs{i}\leq N\\ i\neq 0}}\sign(i)E_{ii}\tensor 1-
N\sum_{\substack{\abs{i}\leq N\\ i\neq 0}}1\tensor\sign(i)\KN{ii}+
\frac{1}{2}N^3.
$$
Since in the $N^{\mathrm{th}}$ level $\hinf^\C$-invariant sector $\Delta_{\varrho}^{(N)}\equiv 0$
and $E_{ij}\tensor 1=-1\tensor\KN{ij}$ the claim follows.
\end{proof}

It follows from the above that at this point we may write in
the $N^{\mathrm{th}}$ level $\hinf^\C$-invariant sector the
square of the Dirac operator as
\begin{eqnarray}\label{squareInTheEquivSec}
4\KDirac_{(N)}^2 &=&\Big(\Delta_{\lie{g},\mathrm{ren}}^{(N)}\tensor 1+
\sum_{\substack{ij>0 \\ \abs{i},\abs{j}\leq N}}1\tensor\Hn{ij}\Hn{ji}\Big)+
1\tensor\sum_{\substack{\abs{i}\leq N\\ i\neq 0}}\sign(i)\Hn{ii}\nonumber\\
&+& N^2-\frac{1}{2}N^3,
\end{eqnarray}
where inside the parenthesis we think we have a normal ordered Casimir operator
for $\lie{g}$ acting on the left hand side of the tensor product, and a non-normal ordered
Casimir operator for the Lie subalgebra $\hinf^\C$ acting on the
right hand side of the tensor product.

\subsubsection{Getting rid of the diverging polynomial terms $P(N)$}
We begin by trying to write the naive spinorial Casimir operator
$\sum_{ij>0}1\tensor\Hn{ij}\Hn{ji}$ appearing in (\ref{squareInTheEquivSec}) in
a form that could be normal ordered.

Using equation (\ref{KijHij1}) one may as a first step write
\begin{equation}\label{SumHijHij}
\sum_{\substack{ij>0\\ \abs{i},\abs{j}\leq N}}\Hn{ij}\Hn{ji}=
\sum_{\substack{ij>0\\ \abs{i},\abs{j}\leq N}}\Kn{ij}\Kn{ji}
-N\sum_{\substack{\abs{i}\leq N\\ i\neq 0}}\Kn{ii}+
\frac{1}{2}N^3,
\end{equation}
which now also includes a term $\frac{1}{2}N^3$ annihilating the diverging term
$-\frac{1}{2}N^3$ appearing in (\ref{squareInTheEquivSec}).

Moreover, in order to annihilate
the diverging polynomial term $N^2$ in 
(\ref{squareInTheEquivSec}),
we want to write the term in equation (\ref{squareInTheEquivSec})
containing sign operators in terms of the operators $\KN{ij}$. This yields
\begin{eqnarray}\label{signHii}
1\tensor\sum_{\substack{\abs{i}\leq N\\ i\neq 0}}\sign(i)\Hn{ii} &=&
1\tensor\sum_{\substack{\abs{i}\leq N\\ i\neq 0}}\sign(i)
\Big(\KN{ii}-\frac{1}{2}\sign(i)\delta_{ij}N\Big)\nonumber\\
&=&
1\tensor\sum_{\substack{\abs{i}\leq N\\ i\neq 0}}\sign(i)
\KN{ii}-N^2.
\end{eqnarray}

Inserting equations (\ref{SumHijHij}) and (\ref{signHii}) into
equation (\ref{squareInTheEquivSec}) gives us now
\begin{eqnarray}
4\KDirac^2_{(N)} &=&
\Delta_{\lie{g},\mathrm{ren}}^{(N)}+1\tensor
\sum_{\substack{ij>0\\ \abs{i},\abs{j}\leq N}}\Kn{ij}\Kn{ji}-1\tensor
N\sum_{\substack{\abs{i}\leq N\\ i\neq 0}}\Kn{ii}\nonumber\\
&+&
1\tensor\sum_{\substack{\abs{i}\leq N\\ i\neq 0}}\sign(i)\KN{ii}.\nonumber
\end{eqnarray}
Writing the sum
$$
N\sum_{\substack{\abs{i}\leq N\\ i\neq 0}}\Kn{ii}=
N\sum_{\substack{\abs{i}\leq N\\ i\neq 0}}\KN{ii}+
N\sum_{\substack{i<0\\ \abs{i}\leq N}}N\cdot 1=
N\sum_{\substack{\abs{i}\leq N\\ i\neq 0}}\KN{ii}+
N^3\cdot 1
$$
this becomes
\begin{eqnarray}\label{KDirac2WithHiddenZeroOp}
4\KDirac^2_{(N)} &=&
\overbrace{\Delta_{\lie{g},\mathrm{ren}}^{(N)}}^{=: I}\tensor 1+1\tensor
\Big(\overbrace{\sum_{\substack{ij>0\\ \abs{i},\abs{j}\leq N}}\Kn{ij}\Kn{ji}-N^3}^{=: II}\Big)\nonumber\\
&+&
1\tensor\underbrace{\sum_{\substack{\abs{i}\leq N\\ i\neq 0}}\sign(i)\KN{ii}}_{=: III}
-1\tensor\underbrace{N\sum_{\substack{\abs{i}\leq N\\ i\neq 0}}\KN{ii}}_{=: IV}.
\end{eqnarray}

We proceed from here by giving a detailed analysis of the operators $1\tensor\sum_{i}\sign(i)\KN{ii}$
and $N\sum_{i}\KN{ii}$ appearing in (\ref{KDirac2WithHiddenZeroOp}). It turns out that when interpreted properly
the former is essentially a fermion number operator and the latter is actually a zero operator in disguise.
Also, the operator $\sum_{ij>0}\Kn{ij}\Kn{ji}-N^3$ 
appearing in (\ref{KDirac2WithHiddenZeroOp}) turns
out to be identically zero.
\subsection{Operator III -- The fermion number operator}\label{FNOP}
Consider a basis vector
\begin{equation}\label{basisforS}
s=\gamma_{m_1 l_1}\gamma_{m_2 l_2}\cdots\gamma_{m_k l_k}\vac_{\spib}\in \spib^{\mathrm{pol}}\quad (m_i>0 \textrm{ for all }i),
\end{equation}
where $k\in\Z_{\geq 0}$. Here we understand that the value $k=0$ corresponds to the case $s\in\C\vac_\spib$. We set for 
a basis vector like above
\begin{equation}
\Length(s):=k.
\end{equation}
\begin{defn}
We define the subspace $\spib^{\mathrm{pol},k}\subset\spib^{\mathrm{pol}}$ consisting of all
finite $\C$-linear combinations of the basis vectors (\ref{basisforS}) such that $\Length(s)=k$ and equip it with
the inner product inherited from $\spib^{\mathrm{pol}}$.
\end{defn}

Obviously then the \emph{algebraic} direct sum satisfies $\bigoplus_{k=1}^\infty \spib^{\mathrm{pol},k}=\spib^{\mathrm{pol}}$.

Recall that for every $N\in\Z_{\geq 0}$ we defined the subspace $\spib^{\mathrm{pol},(N)}\subset\spib^{\mathrm{pol}}$ to consist of all
finite $\C$-linear combinations of the vacuum $\vac_\spib$ and the basis vectors (\ref{basisforS}) such that $\abs{m_i},\abs{l_i}\leq N$
for all $i=1,\ldots k$.
We also noticed that $\spib^{\mathrm{pol},(M)}\subset\spib^{\mathrm{pol},(N)}$ if $M\leq N$ and that
$\bigcup_{N=1}^\infty \spib^{\mathrm{pol},(N)}=\spib^{\mathrm{pol}}$.

\begin{lem}\label{fernum}
Fix $M\in\Z_{\geq 0}$ and let $N\in\Z_{>0}$ be any integer satisfying $N\geq M$. 
Based on these assumptions define
$$
F_{(N)}:=\sum_{\substack{i\in\Z\setminus\set{0}\\ \abs{i}\leq N}}\sign(i)\KN{ii}.
$$ Then for $k,l\in\Z$ with $kl<0$ and $\abs{k},\abs{l}\leq M$
\begin{equation}
[F_{(N)},\gamma_{kl}]=
\left\{ \begin{array}{ll}
+2\gamma_{kl} & \textrm{if $k>0$}\\
-2\gamma_{kl} &\textrm{if $k<0$}.
\end{array} \right.
\end{equation}
\end{lem}

\begin{proof}
The claim follows directly from Proposition \ref{actioniscommutator} and Lemma \ref{Kijlimit}.
\end{proof}

\begin{cor}\label{FNonBasisVec}
Fix $M\in\Z_{\geq 0}$ and let $N\in\Z_{>0}$ be any integer satisfying $N\geq M$.
The operator $F_{(N)}$ acts on the basis vectors 
\begin{equation}\label{Fons}
s=\gamma_{m_1 l_1}\gamma_{m_2 l_2}\cdots\gamma_{m_k l_k}\vac_{\spib}\in \spib^{\mathrm{pol},(M)}\quad (m_i>0 \textrm{ for all }i)
\end{equation}
so that
$$
F_{(N)}(s)=2k\cdot s\in \spib^{\mathrm{pol},(M)},
$$
where also the value $k=0$ is accepted by which we mean that in the equation (\ref{Fons}) $s\in\C\vac_\spib$.
\end{cor}
\begin{proof}
Since for each integer $N\geq M$, $\KN{ii}$ operates on 
$\spib^{\mathrm{pol},(M)}\subset\spib_{\lie{p}}^c$ via the commutator action by 
Proposition \ref{actioniscommutator}, we have that
$$
F_{(N)}(s)=\sum_{h=1}^k\gamma_{m_1 l_1 }\cdots [F_{(N)},\gamma_{m_h l_h}]\cdots\vac_{\spib}
$$
and the result follows from Lemma \ref{fernum}.
\end{proof}

\begin{prop}\label{FIsWelldefined}
The infinite sum of operators,
$$
F:=\sum_{i\in\Z\setminus\set{0}}\sign(i)\Kt{ii}.
$$
is a well-defined unbounded operator $F:\spib_{\lie{p}}^c\To\spib_{\lie{p}}^c$ with (dense) domain
$\D(F)=\spib^{\mathrm{pol}}\subset\spib_{\lie{p}}^c$. Moreover for a basis vector
$$
s=\gamma_{m_1 l_1}\gamma_{m_2 l_2}\cdots\gamma_{m_k l_k}\vac_{\spib}\in \spib^{\mathrm{pol}}\quad (m_i>0 \textrm{ for all }i)
$$
we have
$$
F(s)=2k\cdot s\in \spib^{\mathrm{pol}}.
$$
Thus the restriction
\begin{equation}\label{Frestricted}
F\Big\vert_{\spib^{\mathrm{pol},k}}:\spib^{\mathrm{pol},k}\To\spib^{\mathrm{pol},k},\quad F\Big\vert_{\spib^{\mathrm{pol},k}}=2k\cdot I,
\end{equation}
for all $k\in\Z_{\geq 0}$.
\end{prop}
\begin{proof}
First note that according to Corollary \ref{isotropycorollary} 
$F(s)\in\spib^{\mathrm{pol}}$
for each
$$
s=\gamma_{m_1 l_1}\gamma_{m_2 l_2}\cdots\gamma_{m_k l_k}\vac_{\spib}\in \spib^{\mathrm{pol}}
\quad (m_i>0 \textrm{ for all }i),
$$
since the action of each
$\Kt{ii}$ is the isotropy action, and that on the other hand actually only a finite
number of actions of various $\Kt{ii}$ in this sum are nonzero, because
$[\Kt{ii},\gamma_{m_j l_j}]\neq 0\iff i=m_j$ or $i=l_j$. For these values of 
$i$ the commutator $[\Kt{ii},\gamma_{m_j l_j}]=\sign(i)\gamma_{m_j l_j}$, since
we assumed that each $m_j>0$. We conclude from this
the claimed eigenvalues of $F$ and that
$F$ is indeed a well-defined operator on the dense subspace $\spib^{\mathrm{pol}}\subset\spib_{\lie{p}}^c$, thus
an unbounded operator on $\spib_{\lie{p}}^c$.
\end{proof}

\begin{cor}\label{conclusiveF}
The operator $F:\spib_{\lie{p}}^c\To\spib_{\lie{p}}^c$ is diagonalizable and non-negative with
spectrum $\spec(F)=2\Z_{\geq 0}$ and kernel
\begin{equation}
\ker(F)=\C\vac_\spib.
\end{equation}
In particular
$\dim\ker(F)=1<\infty$.
\end{cor}
\begin{proof}
This follows from equation (\ref{Frestricted}) after we have proved that the
eigenspace decomposition of $F$
$$
\spib^{\mathrm{pol}}=\bigoplus_{k=0}^\infty\spib^{\mathrm{pol},k}
$$
is an orthogonal decomposition, i.e. that
$$
\spib^{\mathrm{pol},k}\perp \spib^{\mathrm{pol},l}
$$
when $k\neq l$.

For this
look at the generators
$$
s=\gamma_{m_1 l_1}\gamma_{m_2 l_2}\cdots\gamma_{m_k l_k}\vac_\spib\in\spib^{\mathrm{pol},k}\qquad
(m_h>0,\, l_h<0\textrm{ for all } h=1,\ldots k),
$$
with $\abs{m_h},\abs{l_h}\leq N$ for every $h$,
and
notice
that because of the sign conventions the anti-commutator
$$
\{\gamma_{m_h l_h},\gamma_{m_k l_k}\}=2\delta_{m_h l_k}\delta_{l_h m_k}=0\quad
(m_h, m_k>0,\, l_h, l_k<0),
$$
i.e. the various Clifford algebra elements in the above expression for
$s$ always anti-commute. In particular $2\gamma_{ml}^2=\{\gamma_{ml},\gamma_{ml}\}=0$. 
Hence we may permute the Clifford algebra generators in the above expression for
$s$ anyway we want as long as we keep track how the sign changes under these
permutations. It follows that for a nonzero $s$ as above,
all pairs of indices $(m_h, l_h)$ are pairwise distinct. Since we are interested in generating
$\spib^{\mathrm{pol},k}$ we may assume that our elements $s$ are
always of this form. But then by the definition of our Clifford action on $\spib_{\lie{p}}^c$ the generator
$$
s=\Big(\frac{1}{\sqrt{2}}\Big)^k\Psi^\ast(E_{m_1 l_1})\cdots\Psi^\ast(E_{m_k l_l})\vac_\spib,
$$
i.e. scalar multiple of an element of the standard orthogonal basis for $\spib_{\lie{p}}^c=\F(\lie{p}^\C,W)$.
It follows that
$$
\spib^{\mathrm{pol},k}\perp \spib^{\mathrm{pol},l}
$$
when $k\neq l$.
\end{proof}

\begin{prop}
Consider the unbounded operators $F_{(N)}:\spib_{\lie{p}}^c\To\spib_{\lie{p}}^c\,\,(N\in\N)$ and
$F:\spib_{\lie{p}}^c\To\spib_{\lie{p}}^c$ with common dense domain
$\spib^{\mathrm{pol}}\subset\spib_{\lie{p}}^c$. Then it holds that
\begin{equation}
\stlim{N} F_{(N)}=F.
\end{equation}
\end{prop}
\begin{proof}
Follows directly from Corollary \ref{FNonBasisVec} and Proposition \ref{FIsWelldefined}.
\end{proof}

\subsection{Operator IV -- The zero operator in disguise}
Once again, fix $M\in\Z_{\geq 0}$.
For each $N\in\Z_{\geq 0}$ with $N\geq M$ we are interested in the unbounded operator
$T(N):\spib_{\lie{p}}^c\To\spib_{\lie{p}}^c$,
$$
T(N):=N\sum_{\substack{\abs{i}\leq N\\ i\neq 0}}\KN{ii},\qquad\domain(T(N))=\spib^{\mathrm{pol}},
$$
and its behaviour as $N\to\infty$. 

\begin{lem}
Fix $M\in\Z_{\geq 0}$.
Then as an unbounded operator $S(N):\spib_{\lie{p}}^c\To\spib_{\lie{p}}^c$,
$$
S(N):=\sum_{\substack{\abs{i}\leq N\\ i\neq 0}}\KN{ii},\qquad\domain(T(N))=\spib^{\mathrm{pol}}
$$
acts as zero on the subspace $\spib^{\mathrm{pol},(M)}\subset\spib_{\lie{p}}^c$
for each $N\in\Z_{>0}$ satisfying $N\geq M$.
\end{lem}
\begin{proof}
Let
$$
s=\gamma_{m_1 l_1}\gamma_{m_2 l_2}\cdots\gamma_{m_k l_k}\vac_{\spib}\in \spib^{\mathrm{pol},(M)}\quad (m_i>0 \textrm{ for all }i)
$$
be a basis vector for $\spib^{\mathrm{pol},(N)}$ (hence $\abs{m_i},\abs{l_i}\leq M\leq N$ for all $i$). Again,
we want to apply the results of Proposition \ref{actioniscommutator} and Lemma \ref{Kijlimit}. Proposition
\ref{actioniscommutator} tells us that the action of each $\KN{ii}$ is the 
isotropy / adjoint action and according to
Lemma \ref{Kijlimit} the commutator
$[\Kt{ii},\gamma_{m_j l_j}]\neq 0\iff i=m_j$ or $i=l_j$, and for these values of 
$i$ the commutator $[\Kt{ii},\gamma_{m_j l_j}]=\pm\gamma_{m_j l_j}$, i.e. the \emph{opposite} values.
Using this result the claim follows when one changes the summation order and writes
$$
S(N)\cdot s=\sum_{h=1}^k\gamma_{m_1 l_1 }\cdots [S(N),\gamma_{m_h l_h}]\cdots\vac_{\spib}.
$$
\end{proof}

\begin{cor}\label{ZeroInDisguise}
Fix $M\in\Z_{\geq 0}$ and let $N\in\Z_{>0}$ be any integer satisfying
$N\geq M$.
Then the unbounded operator $T(N):\spib_{\lie{p}}^c\To\spib_{\lie{p}}^c$
acts as zero when restricted to the subspace $\spib^{\mathrm{pol},(M)}\subset\spib^{\mathrm{pol}}$,
\begin{equation}
T(N)\Big\vert_{\spib^{\mathrm{pol},(M)}}=
NS(N)\Big\vert_{\spib^{\mathrm{pol},(M)}}=N\cdot 0\equiv 0.
\end{equation}
\end{cor}
We immediately conclude from this the following Proposition.
\begin{prop}
The unbounded operators $T(N):\spib_{\lie{p}}^c\To\spib_{\lie{p}}^c$ with common
dense domain $\domain(T(N))=\spib^{\mathrm{pol}}$ satisfy
the strong limit condition
\begin{equation}
\stlim{N} T(N)=0.
\end{equation}
\end{prop}

\subsubsection{Filtrations for the invariant sectors}
Recall that in \S \ref{OurInfDimCase} we defined for the $\C$-vector space $\spib^{\mathrm{pol}}$ a natural
increasing filtration $F_\bullet=\set{F_p}_{p\geq 0}$ by $\C$-vector subspaces, where 
$F_p=F_p\spib^{\mathrm{pol}}=\spib^{\mathrm{pol},(p)}$.
This induces an increasing filtration $G_\bullet=\set{G_p}_{p\geq 0}$ on the
tensor product $\F^{\mathrm{pol}}\tensor\spib^{\mathrm{pol}}$ by
$$
G_p:=\F^{\mathrm{pol}}\tensor F_p\spib^{\mathrm{pol}}=\F^{\mathrm{pol}}\tensor\spib^{\mathrm{pol},(p)}.
$$
Taking intersections defines then an increasing filtration $H_\bullet^N=\set{H_p^N}_{p\geq 0}$ 
by $\C$-vector subspaces of
each invariant sector $(\F^{\mathrm{pol}}\tensor\spib^{\mathrm{pol}})^{\lie{h},N}$,
$$
H_p^N:=(\F^{\mathrm{pol}}\tensor\spib^{\mathrm{pol}})^{\lie{h},N}\cap G_p,
$$
which in concrete terms is given by
\begin{eqnarray}
H_p^N(\F^{\mathrm{pol}}\tensor\spib^{\mathrm{pol}})^{\lie{h},N} 
&=&
\Big\{w\in \F^{\mathrm{pol}}\tensor\spib^{\mathrm{pol},(p)}\,\Big\vert\,\Big[\rh(E_{pq})\tensor 1\Big](w)=\Big[-1\tensor \KN{pq}\Big](w)
\nonumber \\ 
& & \textrm{ for all } p,q\in\Z\textrm{ such that }pq>0\textrm{ and } \abs{p},\abs{q}\leq N\Big\}\nonumber.
\end{eqnarray}

We collect the results obtained so far in the form of the following Lemma.

\begin{lem}\label{OurSpinCasIsZero}
Consider the Dirac operator $\KDirac$ as an unbounded operator
$\KDirac:(\F\tensor\spib_{\lie{p}}^c)^{\lie{h}} \To(\F\tensor\spib_{\lie{p}}^c)^{\lie{h}}$, and
similarly we consider the cut-offs $\KDiracc$ as unbounded operators
$\KDiracc:(\F\tensor\spib_{\lie{p}}^c)^{\lie{h},N} \To(\F\tensor\spib_{\lie{p}}^c)^{\lie{h},N}$.
Then for each $\varphi\in\domain(\KDirac)$ there exists $N_\varphi\in\N$ such that for all
$N\geq N_\varphi$, $\varphi\in\domain(\KDiracc)$ and

\begin{eqnarray}\label{SecondLastCut-Off}
4\KDirac^2\varphi=4\KDirac^2_{(N)}\varphi &=&
\Big(\Delta_{\lie{g},\mathrm{ren}}^{(N)}\tensor 1+1\tensor
\sum_{\substack{ij>0\\ \abs{i},\abs{j}\leq N}}\Kn{ij}\Kn{ji}-N^3
+
1\tensor F_{(N)}\Big)\varphi.
\end{eqnarray}
\end{lem}
\begin{proof}
First notice that since
$\varphi\in(\F^{\mathrm{pol}}\tensor\spib^{\mathrm{pol}})^{\lie{h}}\subset
\F^{\mathrm{pol}}\tensor\spib^{\mathrm{pol}}$
and since $G_\bullet=\set{G_p}_{p\geq 0}$ is an increasing filtration of the
tensor product $\F^{\mathrm{pol}}\tensor\spib^{\mathrm{pol}}$ there
exists an integer $q\in\Z_{\geq 0}$ such that $\varphi\in G_q$.
On the other hand, by Proposition \ref{StroLimInEqSector} there exists
$N_\varphi'\in\Z_{\geq 0}$ such that 
$\varphi\in(\F^{\mathrm{pol}}\tensor\spib^{\mathrm{pol}})^{\lie{h},N}$
for all $N\geq N_\varphi'$. Hence
$$
\varphi\in(\F^{\mathrm{pol}}\tensor\spib^{\mathrm{pol}})^{\lie{h},N}\cap G_q=
\domain(\KDiracc)\cap G_q=H_q^N
$$
for all $N\geq N_\varphi'$. Moreover, Proposition \ref{StroLimInEqSector} tells us that
for these values of $N$,
$$
\KDirac\varphi=\KDiracc\varphi,
$$
and that also $\KDirac^2\varphi=\KDiracc^2\varphi$ whenever
$N\geq N_\varphi'$.

Next notice that by Corollary \ref{ZeroInDisguise}
$[1\tensor T(N)](H_q^N)=0$ for all $N\geq q$ so that in particular
$[1\tensor T(N)]\varphi=0$ for these values of $N$.
Thus, we may choose $N_\varphi=\max\set{q,N_\varphi'}$ and apply equation
(\ref{KDirac2WithHiddenZeroOp}).
\end{proof}

\subsection{Operator II -- The spinor Casimir operator of $\hinf^\C$}
The final thing to do in our task of introducing a proper normal ordering for
the square of the Dirac operator is to make sense of and analyze 
the behaviour of
the (unbounded) operator
$$
\sum_{\substack{i,j\in\Z \\ ij>0\\ \abs{i},\abs{j}\leq N}} \Kn{ij}\Kn{ji}-N^3
$$
acting on $\spib_{\lie{p}}^c$ in 
the `limit' $N\to\infty$. We call this hypothetical limit
$\Delta_{\spib_{\lie{p}},\mathrm{ren}}^{\lie{h}}$ the
\emph{spinor Casimir operator of} $\hinf^\C$. Of course since the definition contains the diverging term $N^3$ we would like to show
that for every $N\in\N$ the cut-off
$$
\Delta_{\spib_{\lie{p}},\mathrm{ren}}^{(N)}:=\sum_{\substack{i,j\in\Z \\ ij>0\\ \abs{i},\abs{j}\leq N}} \Kn{ij}\Kn{ji}-N^3
$$
satisfies
\begin{equation}
\stlim{N} \Delta_{\spib_{\lie{p}},\mathrm{ren}}^{(N)}=0.
\end{equation}

Equivalently, we can define the operator
$$
\Delta_{\spib_{\lie{p}}}^{(N)}:=\sum_{\substack{i,j\in\Z \\ ij>0\\ \abs{i},\abs{j}\leq N}} \Kn{ij}\Kn{ji}
$$
for each $N\in\N$, and show that 
for a fixed element $s\in\spib^{\mathrm{pol}}$, each operator
$\Delta_{\spib_{\lie{p}}}^{(N)}$ acts as the constant operator
$$
\Delta_{\spib_{\lie{p}}}^{(N)}\cdot s=N^3\cdot s.
$$
for sufficiently large $N$ depending on the element $s$ chosen.
The reader should compare the above claim with Proposition \ref{KostCasConst}.

\begin{lem}
Consider $\Kn{ij},\KN{ij}:\spib_{\lie{p}}^c\To\spib_{\lie{p}}^c$ as unbounded operators
with dense domain $\spib^{\mathrm{pol}}$. Then
the following identity holds for all pairs $i,j\in\Z$ with $ij>0$:
\begin{eqnarray}\label{VariousSpinCasimirs}
\Kn{ij}\Kn{ji} &=& \KN{ij}\KN{ji}+\delta_{ij}\chi_{\leq 0}(i)\Big(2N\KN{ii}+N^2\cdot 1\Big),
\end{eqnarray}
where $\chi_{\leq 0}$ is the characteristic function of the interval $(-\infty,0]$.
In particular,
\begin{equation}\label{KnSpinorCasimirOnVac}
\Kn{ij}\Kn{ji}\vac_\spib=\delta_{ij}\chi_{\leq 0}(i)N^2\cdot\vac_\spib.
\end{equation}
\end{lem}
\begin{proof}
By equation (\ref{KijHij2}) we may express the operators $\Kn{ij}$ in terms of the
operators $\KN{ij}$ so that

\begin{eqnarray}
\Kn{ij}\Kn{ji} &=& \left\{
\begin{array}{ll}
\KN{ij}\KN{ji} & \text{if } i\neq j \text{ or } i =j > 0 \\
(\KN{ii})^2+2N\KN{ii}+N^2\cdot 1 & \text{if } i = j < 0,
\end{array} \right.
\end{eqnarray}
from which the first claim follows.

The second claim follows from equation (\ref{VariousSpinCasimirs}) and Lemma \ref{Kijlimit} 
which says that $\Kt{ij}\vac_\spib=0$ for all $i,j\in\Z$ with $ij>0$.
\end{proof}

\begin{prop}\label{TheActionOfKijKji}
Fix $M\in\Z_{\geq 0}$ and let $N\in\Z_{>0}$ be any integer satisfying $N\geq M$.
Then for each $i,j\in\Z$ with $ij>0$ the operator 
$\Kn{ij}\Kn{ji}$
acts on $\SPol{M}\subset\spib_{\lie{p}}^c$ so that if
$$
s=\gamma_{m_1 l_1}\gamma_{m_2 l_2}\cdots\gamma_{m_k l_k}\vac_\spib\qquad
(m_h>0,\, l_h<0\,;\abs{m_h},\abs{l_h}\leq M\textrm{ for all } h=1,\ldots k)
$$
is a basis vector for
$\SPol{M}$, then

\begin{eqnarray}\label{SpinorCasIsAdj}
\Kn{ij}\Kn{ji}\cdot s &=& \sum_{h=1}^k
\gamma_{m_1 l_1}\cdots[\Kn{ij}\Kn{ji},\gamma_{m_h l_h}]\cdots\gamma_{m_k l_k}\vac_\spib \nonumber\\
&+& 
\delta_{ij}\chi_{\leq 0}(i)N^2\cdot s.
\end{eqnarray}
\end{prop}
\begin{proof}
By induction on the `length' $k$. 
The case $k=0$ follows from equation (\ref{KnSpinorCasimirOnVac}).

Let $k=1$ so that
$s=\gamma_{ml}\vac_\spib$ where $m>0$ and $l<0$.
Using equation (\ref{KnSpinorCasimirOnVac}) again, one may write

\begin{eqnarray}
\Big[\Kn{ij}\Kn{ji},\gamma_{ml}\Big]\vac_\spib &=&
\Kn{ij}\Kn{ji}\gamma_{ml}\vac_\spib-\gamma_{ml}\Kn{ij}\Kn{ji}\vac_\spib\nonumber\\
&=&
\Kn{ij}\Kn{ji}\gamma_{ml}\vac_\spib-\delta_{ij}\chi_{\leq 0}(i)N^2\cdot\gamma_{ml}\vac_{\spib}\nonumber\\
&=&
\Kn{ij}\Kn{ji}\cdot s-\delta_{ij}\chi_{\leq 0}(i)N^2\cdot s.\nonumber
\end{eqnarray}

We can now make the induction hypothesis for $k=n$ and proceed exactly as in the
proof of Proposition \ref{actioniscommutator}, i.e. one looks at basis elements of the form
$$
s'=\gamma_{ml}\gamma_{m_1 l_1}\gamma_{m_2 l_2}\cdots\gamma_{m_k l_k}\vac_\spib=\gamma_{ml}\cdot s
\quad(m>0,l<0,\,\abs{m},\abs{l}\leq M).
$$
Then one may write
$$
\Kn{ij}\Kn{ji}\cdot s'=\Big(\Kn{ij}\Kn{ji}\Big)(\gamma_{ml}\cdot s)=
[\Kn{ij}\Kn{ji},\gamma_{ml}]s+\gamma_{ml}\cdot
\Big(\Kn{ij}\Kn{ji}\cdot s\Big)
$$
and apply the induction hypothesis to the term $\Kn{ij}\Kn{ji}\cdot s$.
\end{proof}

\begin{cor}\label{ActionOfSpinorCas}
Fix $M\in\Z_{\geq 0}$ and let $N\in\Z_{>0}$ be any integer satisfying $N\geq M$.
Then the operator 
$\Delta_{\spib_{\lie{p}}}^{(N)}$
acts on $\SPol{M}$ so that if
$$
s=\gamma_{m_1 l_1}\gamma_{m_2 l_2}\cdots\gamma_{m_k l_k}\vac_\spib\qquad
(m_h>0,\, l_h<0\, ;\abs{m_h},\abs{l_h}\leq M\textrm{ for all } h=1,\ldots k)
$$
is a basis vector for
$\SPol{M}$, then
\begin{eqnarray}\label{SpinorCasIsAdj}
\Delta_{\spib_{\lie{p}}}^{(N)}\cdot s &=& 
\sum_{\substack{i,j\in\Z\\ ij>0\\ \abs{i},\abs{j}\leq N}}
\sum_{h=1}^k
\gamma_{m_1 l_1}\cdots[\Kn{ij}\Kn{ji},\gamma_{m_h l_h}]\cdots\gamma_{m_k l_k}\vac_\spib \nonumber\\
&+& 
N^3\cdot s\\
&=&
\sum_{h=1}^k
\gamma_{m_1 l_1}\cdots[\Delta_{\spib_{\lie{p}}}^{(N)},\gamma_{m_h l_h}]\cdots\gamma_{m_k l_k}\vac_\spib \nonumber\\
&+& 
N^3\cdot s.
\end{eqnarray}
\end{cor}
\begin{proof}
The first equality follows from Proposition \ref{TheActionOfKijKji} by noticing that
$$
\sum_{\substack{ij>0\\ \abs{i},\abs{j}\leq N}}\delta_{ij}\chi_{\leq 0}(i)N^2\cdot 1=N^2\sum_{i=-N}^{-1} 1=N^3\cdot 1.
$$

The second equation follows from the first one by changing the order of summation and using the (bi)linearity of the Lie bracket $[\, , \,]$.
\end{proof}

\begin{lem}\label{ComputingComSpinCasGa}
Fix $M\in\Z_{\geq 0}$ and let $N\in\Z_{>0}$ be any integer satisfying $N\geq M$.
For $m,l\in\Z$ with $m>0$, $l<0$ and $\abs{m},\abs{l}\leq M$ one can write the commutator $[\Delta_{\spib_{\lie{p}}}^{(N)},\gamma_{ml}]$ as
\begin{equation}
[\Delta_{\spib_{\lie{p}}}^{(N)},\gamma_{ml}]=2\sum_{i=1}^N\gamma_{il}\KN{mi}-2\sum_{i=-N}^{-1}\gamma_{mi}\KN{il}.
\end{equation}
\end{lem}
\begin{proof}
Using the general commutator formula
$$
[AB,C]=A[B,C]+[A,C]B
$$
and the identity
$$
[\Kn{ij},\gamma_{ml}]=\delta_{jm}\gamma_{il}-\delta_{il}\gamma_{mj}
$$
proved earlier, one computes for $ij>0$
\begin{eqnarray}
[\Kn{ij}\Kn{ji},\gamma_{ml}]&=& \Kn{ij}[\Kn{ji},\gamma_{ml}]+[\Kn{ij},\gamma_{ml}]\Kn{ji}\nonumber\\
&=&
\Kn{ij}(\delta_{im}\gamma_{jl}-\delta_{jl}\gamma_{mi})+
(\delta_{jm}\gamma_{il}-\delta_{il}\gamma_{mj})\Kn{ji}\nonumber\\
&=&
\delta_{im}\Kn{ij}\gamma_{jl}-\delta_{jl}\Kn{ij}\gamma_{mi}+\delta_{jm}\gamma_{il}\Kn{ji}-\delta_{il}\gamma_{mj}\Kn{ji}\nonumber
\end{eqnarray}

Now recall that $m>0,\,l<0$ and notice that since $ij>0$ (i.e. $i$ and $j$ have the same sign) we have two cases:
\begin{enumerate}
\item $i,j>0$. In this case since $jl<0$ and $il<0$ necessarily $\delta_{jl},\delta_{il}\equiv 0$ while
$\delta_{im}$ and $\delta_{jm}$ can be equal to $1$.
\item $i,j<0$. Now $\delta_{im},\delta_{jm}\equiv 0$ while $\delta_{jl}$ and $\delta_{il}$ can be equal
to $1$.
\end{enumerate}

Hence using the calculations above one has
\begin{eqnarray}\label{uglylemmaPart1}
[\Delta_{\spib_{\lie{p}}}^{(N)},\gamma_{ml}] &=&
\Big[\sum_{\substack{ij>0 \\ \abs{i},\abs{j}\leq N}}\Kn{ij}\Kn{ji},\gamma_{ml}\Big]=
\sum_{\substack{ij>0 \\ \abs{i},\abs{j}\leq N}}[\Kn{ij}\Kn{ji},\gamma_{ml}]\nonumber\\
&=&
\sum_{\substack{i,j>0\\ \abs{i},\abs{j}\leq N}}[\Kn{ij}\Kn{ji},\gamma_{ml}]+
\sum_{\substack{i,j<0 \\ \abs{i},\abs{j}\leq N}}[\Kn{ij}\Kn{ji},\gamma_{ml}]\nonumber\\
&=&
\sum_{\substack{i,j>0\\ \abs{i},\abs{j}\leq N}}(\delta_{im}\Kn{ij}\gamma_{jl}+\delta_{jm}\gamma_{il}\Kn{ji})\nonumber\\
&-&
\sum_{\substack{i,j<0 \\ \abs{i},\abs{j}\leq N}}(\delta_{jl}\Kn{ij}\gamma_{mi}+\delta_{il}\gamma_{mj}\Kn{ji})\nonumber\\
&=&
\sum_{i=1}^N\{\Kn{mi},\gamma_{il}\}-\sum_{i=-N}^{-1}\{\Kn{il},\gamma_{mi}\}.
\end{eqnarray}

To compute the anti-commutators appearing in (\ref{uglylemmaPart1}) we use the formula relating anti-commutators with commutators
$$
\{A,B\}=[A,B]-2BA.
$$
Keeping in mind that $m>0$ and $l<0$ one computes
\begin{eqnarray}
\{\Kn{mi},\gamma_{il}\} &=&
[\Kn{mi},\gamma_{il}]+2\gamma_{il}\Kn{mi}\nonumber\\
&=&
\delta_{ii}\gamma_{ml}-\underbrace{\delta_{ml}\gamma_{ii}}_{\equiv 0}+2\gamma_{il}\Kn{mi}\nonumber\\
&=&
\gamma_{ml}+2\gamma_{il}\Kn{mi}
\end{eqnarray}
and similarly one obtains
\begin{equation}
\{\Kn{il},\gamma_{mi}\}=-\gamma_{ml}+2\gamma_{mi}\Kn{il}.
\end{equation}

Inserting these into (\ref{uglylemmaPart1})  yields
\begin{eqnarray}
[\Delta_{\spib_{\lie{p}}}^{(N)},\gamma_{ml}] &=&
\sum_{i=1}^N (\gamma_{ml}+2\gamma_{il}\Kn{mi})
-
\sum_{i=-N}^{-1}(-\gamma_{ml}+2\gamma_{mi}\Kn{il}) \nonumber\\
&=&
2N\gamma_{ml}+2\sum_{i=1}^N \gamma_{il}\Kn{mi}-2\sum_{i=-N}^{-1}\gamma_{mi}\Kn{il}
\end{eqnarray}
Finally use equation (\ref{KijHij2}) to obtain
\begin{eqnarray}
[\Delta_{\spib_{\lie{p}}}^{(N)},\gamma_{ml}] &=&
2N\gamma_{ml}+
2\sum_{i=1}^N \gamma_{il}\KN{mi}-
2\sum_{\substack{i=-N \\ i\neq l}}^{-1}\gamma_{mi}\KN{il}-
2(\gamma_{ml}\KN{ll}+N\gamma_{ml})\nonumber\\
&=&
2\sum_{i=1}^N \gamma_{il}\KN{mi}-
2\sum_{i=-N}^{-1}\gamma_{mi}\KN{il}.
\end{eqnarray}
\end{proof}

The following Lemma provides us with a recursive tool that we shall use in order to prove that
$\Delta_{\spib_{\lie{p}}}^{(N)}$ is a constant operator on $\SPol{M}$ whenever $N\geq M$.
\begin{lem}\label{spinoridentities}
Let $M\in\Z_{\geq 0}$ and consider all integers $N\in\Z_{>0}$ such that $N\geq M$. Then
the following identities hold.
\begin{enumerate}
\item $\Delta_{\spib_{\lie{p}}}^{(N)}\vac_\spib=N^3\vac_\spib$.
\item
$\Big[\Delta_{\spib_{\lie{p}}}^{(N)},\gamma_{ml}\Big]
\vac_\spib=0$
so that
$$
\Delta_{\spib_{\lie{p}}}^{(N)}\gamma_{ml}\vac_\spib=
\Big[\Delta_{\spib_{\lie{p}}}^{(N)},\gamma_{ml}\Big]
\vac_\spib+N^3\gamma_{ml}\vac_\spib=N^3\gamma_{ml}\vac_\spib,
$$
where
$m,l\in\Z$ with $m>0$, $l<0$ and $\abs{m},\abs{l}\leq M$
\item The anti-commutator
$$
\Big\{\Big[\Delta_{\spib_{\lie{p}}}^{(N)},\gamma_{kl}\Big],\gamma_{mn}\Big\}=0,
$$
where $k,l,m,n\in\Z$ with $k,m>0$; $\,l,n<0$ and $\abs{k},\abs{l},\abs{m},\abs{n}\leq M$.
\end{enumerate}
\end{lem}
\begin{proof}
\begin{enumerate}
\item Is a special case of Corollary \ref{ActionOfSpinorCas}.
\item Since according to Lemma \ref{Kijlimit} $\KN{ij}\vac_\spib=0$ for all $i,j\in\Z$ with $ij>0$ it follows
from Lemma \ref{ComputingComSpinCasGa}
that
$$
\Big[\Delta_{\spib_{\lie{p}}}^{(N)},\gamma_{ml}\Big]
\vac_\spib=0.
$$
\item According to Lemma \ref{ComputingComSpinCasGa} in order to
compute this anti-commutator we need to compute the anti-commutators
$\{\gamma_{il}\KN{ki},\gamma_{mn}\}$ and  $\{\gamma_{ki}\KN{il},\gamma_{mn}\}$.
To do this we are going to use the general anti-commutator identity
$$
\{AB,C\}=\{C,A\}B-A[C,B]
$$
and carefully keep track of the signs of the various indices so that the various Kronecker deltas will become identically
zero.

Thus we obtain
\begin{eqnarray}\label{CommutatorAnticommutatorI}
\{\gamma_{il}\KN{ki},\gamma_{mn}\} &=& \{\gamma_{mn},\gamma_{il}\}\KN{ki}-\gamma_{il}[\gamma_{mn},\KN{ki}]\nonumber\\
&=&
2\underbrace{\delta_{ml}}_{\equiv 0}\delta_{in}\KN{ki}+\gamma_{il}[\KN{ki},\gamma_{mn}]\nonumber\\
&=&
\gamma_{il}[\KN{ki},\gamma_{mn}]\nonumber\\
&=&\gamma_{il}\Big(\delta_{im}\gamma_{kn}-\underbrace{\delta_{kn}}_{\equiv 0}\gamma_{mi}\Big)\nonumber\\
&=&
\delta_{im}\gamma_{il}\gamma_{kn},
\end{eqnarray}
and
\begin{eqnarray}\label{CommutatorAnticommutatorII}
\{\gamma_{ki}\KN{il},\gamma_{mn}\} &=&
\{\gamma_{mn},\gamma_{ki}\}\KN{il}-\gamma_{ki}[\gamma_{mn},\KN{il}]\nonumber\\
&=&
2\delta_{mi}\underbrace{\delta_{nk}}_{\equiv 0}\KN{il}+\gamma_{ki}[\KN{il},\gamma_{mn}]\nonumber\\
&=&
\gamma_{ki}[\KN{il},\gamma_{mn}]\nonumber\\
&=&\gamma_{ki}\Big(\underbrace{\delta_{lm}}_{\equiv 0}\gamma_{in}-\delta_{in}\gamma_{ml}\Big)\nonumber\\
&=&
-\delta_{in}\gamma_{ki}\gamma_{ml}.
\end{eqnarray}

Hence inserting (\ref{CommutatorAnticommutatorI}) and (\ref{CommutatorAnticommutatorII}) into the expression given by
Lemma \ref{ComputingComSpinCasGa} yields
\begin{eqnarray}
\Big\{\Big[\Delta_{\spib_{\lie{p}}}^{(N)},\gamma_{kl}\Big],\gamma_{mn}\Big\}
&=&
2\sum_{i=1}^N\delta_{im}\gamma_{il}\gamma_{kn}+2\sum_{i=-N}^{-1}\delta_{in}\gamma_{ki}\gamma_{ml}\nonumber\\
&=&
2(\gamma_{ml}\gamma_{kn}+\gamma_{kn}\gamma_{ml})\nonumber\\
&=&
2\{\gamma_{ml},\gamma_{kn}\}\nonumber\\
&=&
2\underbrace{\delta_{mn}}_{\equiv 0}\underbrace{\delta_{kl}}_{\equiv 0}=0,
\end{eqnarray}
as claimed.
\end{enumerate}
\end{proof}

\begin{cor}\label{AdAdCorollary}
Fix an integer $M\in\Z_{\geq 0}$ and let $N\in\Z_{>0}$ be any integer satisfying $N\geq M$. Then the restriction of the operator 
$$
\Delta_{\spib_{\lie{p}}}^{(N)}:\spib_{\lie{p}}^c\To\spib_{\lie{p}}^c
$$
to the subspace $\SPol{M}$
is the 
constant operator $\Delta_{\spib_{\lie{p}}}^{(N)}\Big\vert_{\SPol{M}}=N^3\cdot\id$.
\end{cor}
\begin{proof}
Let
$$
s=\gamma_{m_1 l_1}\gamma_{m_2 l_2}\cdots\gamma_{m_k l_k}\vac_\spib\in\SPol{M}\qquad
(m_h>0,\, l_h<0\textrm{ for all } h=1,\ldots k),
$$
with $\abs{m_h},\abs{l_h}\leq M$ for every $h$,
be a basis vector. 

Then using first Corollary \ref{ActionOfSpinorCas} and after that
part $(3)$ of Lemma \ref{spinoridentities} one obtains
\begin{eqnarray}
\Delta_{\spib_{\lie{p}}}^{(N)}\cdot s &=& 
\sum_{h=1}^k
\gamma_{m_1 l_1}\cdots[\Delta_{\spib_{\lie{p}}}^{(N)},\gamma_{m_h l_h}]\cdots\gamma_{m_k l_k}\vac_\spib \nonumber\\
&+& 
N^3\cdot s\nonumber\\
&=&
\sum_{h=1}^k
(-1)^{k-h}
\gamma_{m_1 l_1}\gamma_{m_2 l_2}\cdots\gamma_{m_k l_k}
[\Delta_{\spib_{\lie{p}}}^{(N)},\gamma_{m_h l_h}]\vac_\spib\nonumber\\
&+& 
N^3\cdot s.
\end{eqnarray}
The result follows now from part $(2)$ of Lemma \ref{spinoridentities}, which
tells us that
$$
[\Delta_{\spib_{\lie{p}}}^{(N)},\gamma_{m_h l_h}]\vac_\spib=0.
$$
\end{proof}
\begin{cor}
Fix an integer $M\in\Z_{\geq 0}$ and let $N\in\Z_{>0}$ be any integer satisfying $N\geq M$. Then the restriction of the unbounded operator 
$$
\Delta_{\spib_{\lie{p}},\mathrm{ren}}^{(N)}:\spib_{\lie{p}}^c\To\spib_{\lie{p}}^c,\qquad
\Delta_{\spib_{\lie{p}},\mathrm{ren}}^{(N)}:=\Delta_{\spib_{\lie{p}}}^{(N)}-N^3\cdot\id
$$
with dense domain $\domain(\Delta_{\spib_{\lie{p}},\mathrm{ren}}^{(N)})=\spib^{\mathrm{pol}}$ to the subspace $\SPol{M}$, is the 
zero operator $\Delta_{\spib_{\lie{p}}}^{(N)}\Big\vert_{\SPol{M}}\equiv 0$.
\end{cor}
\begin{cor}
The strong limit
$$
\stlim{N} \Delta_{\spib_{\lie{p}},\mathrm{ren}}^{(N)}=0.
$$
\end{cor}
Hence taking into account Lemma \ref{OurSpinCasIsZero} we have come up with the following proposition.

\begin{prop}\label{cutoffthatwewant}
Consider the Dirac operator $\KDirac$ as an unbounded operator
$\KDirac:(\F\tensor\spib_{\lie{p}}^c)^{\lie{h}} \To(\F\tensor\spib_{\lie{p}}^c)^{\lie{h}}$, and
similarly consider the cut-offs $\KDiracc$ as unbounded operators
$\KDiracc:(\F\tensor\spib_{\lie{p}}^c)^{\lie{h},N} \To(\F\tensor\spib_{\lie{p}}^c)^{\lie{h},N}$.
Then for each $\varphi\in\domain(\KDirac)$ there exists $N_\varphi\in\N$ such that for all
$N\geq N_\varphi$, $\varphi\in\domain(\KDiracc)$ and
\begin{eqnarray}\label{SecondLastCut-Off}
4\KDirac^2\varphi=4\KDirac^2_{(N)}\varphi &=&
\Big(\Delta_{\lie{g},\mathrm{ren}}^{(N)}\tensor 1+
1\tensor F_{(N)}\Big)\varphi.
\end{eqnarray}

\end{prop}
\subsection{The strong limit $N\to\infty$ of $\KDirac_{(N)}^2$}
Using the methods of Appendix C, one sees that the unbounded operators 
$\Delta_{\lie{g},\mathrm{ren}}^{(N)}:\F\To\F$ have a well-defined strong limit 
$\Delta_{\lie{g}}:\F\To\F$ with dense domain 
$\D(\Delta_{\lie{g}})=\D(\Delta_{\lie{g},\mathrm{ren}}^{(N)})
=\F^{\mathrm{pol}}$,
$$
\Delta_{\lie{g}}:=2\sum_{\substack{i,j\in\Z\setminus\set{0} \\ j<i}}E_{ij}E_{ji}+
\sum_{\substack{i,j\in\Z\setminus\set{0} \\ i<j}}(E_{ii}-E_{jj})
+\sum_{\substack{i\in\Z\setminus\set{0}}}E_{ii}^2
$$
and
\begin{equation}\label{finalstlimdelta}
\stlim{N}\Delta_{\lie{g},\mathrm{ren}}^{(N)}=\Delta_{\lie{g}}.
\end{equation}

Similarly, we saw in section \S\ref{FNOP} that the operators $F_{(N)}:\spib_{\lie{p}}^c\To\spib_{\lie{p}}^c$ have a well defined strong limit as $N\to\infty$,
\begin{equation}\label{finalstlimF}
\stlim{N}F_{(N)}=F.
\end{equation}
This, combining equations (\ref{finalstlimdelta}) and (\ref{finalstlimF}) with Proposition \ref{cutoffthatwewant} yields the
following theorem.
\begin{thm}\label{MainTheoremI}
When restricted to the $\lie{h}$-invariant sector, the square of the Dirac operator $\KDirac:(\F\tensor\spib_{\lie{p}}^c)^{\lie{h}} \To(\F\tensor\spib_{\lie{p}}^c)^{\lie{h}}$ satisfies
\begin{equation}
4\KDirac^2=\Delta_{\lie{g}}\tensor 1 + 1\tensor F.
\end{equation}
\end{thm}
\begin{proof}
We have seen that for sufficiently large $N$
\begin{eqnarray}
4\KDirac^2\varphi=4\KDirac^2_{(N)}\varphi &=&
\Big(\Delta_{\lie{g},\mathrm{ren}}^{(N)}\tensor 1+
1\tensor F_{(N)}\Big)\varphi=(\Delta_{\lie{g}}\tensor 1 + 1\tensor F)\varphi
\end{eqnarray}
for all $\varphi\in(\F\tensor\spib_{\lie{p}}^c)^{\lie{h}}$, i.e. the variable $N$ drops off!
\end{proof}

We have seen in Appendix C and section \S\ref{FNOP}  how to diagonalize and compute the kernels of the operators 
$\Delta_{\lie{g}}:\F_0\To\F_0$ and $F:\spib_{\lie{p}}^c\To\spib_{\lie{p}}^c$, respectively. 
However, as we mentioned in the Introduction, this method of diagonalization doesn't work directly on the
invariant sector $(\F\tensor\spib_{\lie{p}}^c)^{\lie{h}} $ is not closed with respect to the various commutator
operations one needs to apply in order to produce the diagonalization wanted.
This motivates
us to define the auxiliary operator
$$
T_{\KDirac^{\,2}}:\F_0\tensor\spib_{\lie{p}}^c\To\F_0\tensor\spib_{\lie{p}}^c,\quad
T_{\KDirac^{\,2}}:=\Delta_{\lie{g}}\tensor 1 + 1\tensor F
$$
defined on the \emph{whole} tensor product Hilbert space with
dense domain $\D(T_{\KDirac^{\,2}})=\F_0^{\mathrm{pol}}\tensor\spib^{\mathrm{pol}}$. 
Clearly then the
operator
$T_{\KDirac^{\,2}}$ descends to a (well-defined) unbounded operator
$$
T_{\KDirac^{\,2}}^{\lie{h}}:=
T_{\KDirac^{\,2}}\Big\vert_{(\F_0\tensor\spib_{\lie{p}}^c)^{\lie{h}}}
:\Big(\F_0\tensor\spib_{\lie{p}}^c\Big)^{\lie{h}}\To\Big(\F_0\tensor\spib_{\lie{p}}^c\Big)^{\lie{h}}.
$$
with dense domain $\D(T_{\KDirac^{\,2}}^{\lie{h}})=(\F_0^{\mathrm{pol}}\tensor\spib^{\mathrm{pol}})^{\lie{h}}$, that
satisfies
\begin{equation}
\KDirac^2=\frac{1}{4}\cdot T_{\KDirac^{\,2}}^{\lie{h}}.
\end{equation}

\section{Functional analytic properties of the Dirac operator on $\RGr$}
\subsection{Diagonalization of $\KDirac^2$}\label{SectionOfDiagD2}
First in order to get started we need to recall that the inner product $\ip{\cdot}{\cdot}_{\Hilb_1\tensor\Hilb_2}$ of the tensor product of 
two Hilbert spaces (or inner product spaces) $(\Hilb_1,\ip{\cdot}{\cdot}_1)$ and $(\Hilb_2,\ip{\cdot}{\cdot}_2)$ becomes determined by the formula
\begin{equation}\label{ipontensorprod}
\ip{\psi_1\tensor\psi_2}{\varphi_1\tensor\varphi_2}_{\Hilb_1\tensor\Hilb_2}=\ip{\psi_1}{\varphi_1}_1\ip{\psi_2}{\varphi_2}_2.
\end{equation}

We know by Proposition \ref{ConclusiveDeltaG} that $\Delta_{\lie{g}}:\F_0\To\F_0$ is diagonalizable with (non-negative) spectrum equal to
$2\Z_{\geq 0}$. Similarly by Corollary \ref{conclusiveF} the operator $F:\spib_{\lie{p}}^c\To\spib_{\lie{p}}^c$ is diagonalizable with (non-negative) 
spectrum also equal to $2\Z_{\geq 0}$. It follows then that the operator $T_{\KDirac^{\,2}}=\Delta_{\lie{g}}\tensor 1 + 1\tensor F$, introduced
in the previous section, is also diagonalizable on the tensor product $\F_0\tensor\spib_{\lie{p}}^{\lie{h}}$. Indeed we had the 
orthogonal eigenspace
decompositions
$$
\F_0^{\mathrm{pol}}=\bigoplus_{M=0}^\infty \F_0^{\mathrm{pol},M}\quad\textrm{ (algebraic direct sum)}
$$
and
$$
\spib^{\mathrm{pol}}=\bigoplus_{k=0}^\infty\spib^{\mathrm{pol},k}\quad\textrm{ (algebraic direct sum)},
$$
and the claim follows now from the following isomorphism (between inner product spaces)
$$
\F_0^{\mathrm{pol}}\tensor\spib^{\mathrm{pol}}=
\Big(\bigoplus_{M=0}^\infty \F_0^{\mathrm{pol},M}\Big)\bigotimes\Big(\bigoplus_{k=0}^\infty\spib^{\mathrm{pol},k}\Big)
\isom
\bigoplus_{M,k=0}^\infty \Big(\F_0^{\mathrm{pol},M}\tensor \spib^{\mathrm{pol},k}\Big).
$$
Moreover, we see from this that the eigenvalue of $T_{\KDirac^{\,2}}$ on the $(M,k)$-eigenspace
$\F_0^{\mathrm{pol},M}\tensor \spib^{\mathrm{pol},k}$, where $M,k\geq 0$ is equal to $2(M+k)$ and
by looking at equation (\ref{ipontensorprod})  one sees that the above direct sum of $(M,k)$-eigenspaces
is an orthogonal decomposition. Thus the operator 
$$
T_{\KDirac^{\,2}}:\F_0\tensor\spib_{\lie{p}}^c\To\F_0\tensor\spib_{\lie{p}}^c
$$
is diagonalizable non-negative and its spectrum is equal to
$$
\spec(T_{\KDirac^{\,2}})=2\Z_{\geq 0}.
$$

Next, consider the restriction
$$
T_{\KDirac^{\,2}}^{\lie{h}}=
T_{\KDirac^{\,2}}\Big\vert_{(\F_0\tensor\spib_{\lie{p}}^c)^{\lie{h}}}
:\Big(\F_0\tensor\spib_{\lie{p}}^c\Big)^{\lie{h}}\To\Big(\F_0\tensor\spib_{\lie{p}}^c\Big)^{\lie{h}}.
$$
Write now as above
$$
\Big(\F_0^{\mathrm{pol}}\tensor\spib^{\mathrm{pol}}\Big)^{\lie{h}}\isom
\bigoplus_{M,k=0}^\infty \Big(\F_0^{\mathrm{pol},M}\tensor \spib^{\mathrm{pol},k}\Big)^{\lie{h}}
$$
and notice that since 
$$
\Big(\F_0^{\mathrm{pol},M}\tensor \spib^{\mathrm{pol},k}\Big)^{\lie{h}}\subset \F_0^{\mathrm{pol},M}\tensor \spib^{\mathrm{pol},k}
$$
one thinks that the operator $T_{\KDirac^{\,2}}^{\lie{h}}$ has a constant value $2(M+k)$ when it hits this subspace; however in
order for this to be possible
of course we also need to require that
$
(\F_0^{\mathrm{pol},M}\tensor \spib^{\mathrm{pol},k})^{\lie{h}}\neq\emptyset!
$

Thus it follows from the above that the operator
$T_{\KDirac^{\,2}}^{\lie{h}}$
is also diagonalizable, non-negative and its spectrum satisfies
$$
\spec(T_{\KDirac^{\,2}}^{\lie{h}})\subset 2\Z_{\geq 0}.
$$ 

We are now able to deduce from Theorem \ref{MainTheoremI} the following.

\begin{prop}
The square $\KDirac^2:(\F_0\tensor\spib_{\lie{p}}^c)^{\lie{h}}\To(\F_0\tensor\spib_{\lie{p}}^c)^{\lie{h}}$ of the Dirac operator
is diagonalizable, non-negative and its spectrum satisfies
\begin{equation}
\spec(\KDirac^2)\subset\frac{1}{2}\Z_{\geq 0}.
\end{equation}
\end{prop}
\begin{prop}
The square $\KDirac^2:(\F_0\tensor\spib_{\lie{p}}^c)^{\lie{h}}\To(\F_0\tensor\spib_{\lie{p}}^c)^{\lie{h}}$ 
of the Dirac operator
is symmetric.
\end{prop}
\begin{proof}
This follows from the existence of the orthogonal decomposition of $(\F_0\tensor\spib_{\lie{p}}^c)^{\lie{h}}$
into $(M,k)$-eigenspaces
of $T_{\KDirac^{\,2}}^{\lie{h}}=4\KDirac^2$. Namely, if we have $v$ living in the $(M,k)$-eigenspace and $w$ in the
$(N,l)$-eigenspace with the pair $(M,k)\neq(N,l)$
 then because the orthogonality
$$
\ip{\KDirac^2v}{w}=\frac{1}{2}(M+k)\cdot\ip{v}{w}=0=\frac{1}{2}(N+l)\cdot\ip{v}{w}=\ip{v}{\KDirac^2w},
$$
and if $v,w$ are both in the $(M,k)$-eigenspace then
$$
\ip{\KDirac^2v}{w}=\frac{1}{2}(M+k)\cdot\ip{v}{w}=\ip{v}{\KDirac^2w}.
$$
Thus we conclude that
$$
\ip{\KDirac^2v}{w}=\ip{v}{\KDirac^2w}
$$
for all $v,w\in(\F^{\mathrm{pol}}\tensor\spib^{\mathrm{pol}})^{\lie{h}}=\D(\KDirac^2)$.
\end{proof}
\begin{cor}
The square $\KDirac^2:(\F_0\tensor\spib_{\lie{p}}^c)^{\lie{h}}\To(\F_0\tensor\spib_{\lie{p}}^c)^{\lie{h}}$ 
of the Dirac operator
is essentially self-adjoint.
\end{cor}
\begin{proof}
Follows from the previous two Propositions and Proposition \ref{appendixlemma}.
\end{proof}
\subsection{Kernel of $\KDirac^2$}
After having the explicit diagonalization of $\KDirac^2$ into orthogonal eigenspaces
$$
\Big(\F_0^{\mathrm{pol}}\tensor\spib^{\mathrm{pol}}\Big)^{\lie{h}}\isom\bigoplus_{M,k=0}^\infty
\Big(\F_0^{\mathrm{pol},M}\tensor\spib^{\mathrm{pol},k}\Big)^{\lie{h}},
$$
where however some of the $(M,k)$-eigenspaces appearing as summands 
$$
\Big(\F_0^{\mathrm{pol},M}\tensor\spib^{\mathrm{pol},k}\Big)^{\lie{h}}
$$ 
might be empty,
it is now easy to find out what the kernel of the operator $\KDirac^2$ is. Recall
that if non-empty then the square $\KDirac^2$ obtains the eigenvalue $\frac{1}{2}(M+k)$ on
the $(M,k)$-eigenspace. Since $M,k\in\Z_{\geq 0}$, necessarily
$$
\ker(\KDirac^2)=\Big(\F_0^{\mathrm{pol},0}\tensor\spib^{\mathrm{pol},0}\Big)^{\lie{h}}
$$
and we want to be able to say concretely what this space is.
Recalling the definitions, clearly $\spib^{\mathrm{pol},0}=\C\vac_\spib$ and we saw in 
the proof of Proposition \ref{LastOfTheKernels} that $\F_0^{\mathrm{pol},0}=\C\vac_\F$. Hence
$$
\ker(\KDirac^2)=\Big(\C\vac_\F\tensor\vac_\spib\Big)^{\lie{h}}=\C\vac
$$
by Remark \ref{invariantsecnonempt}, since $\dim(\C\vac_\F\tensor\vac_\spib)^{\lie{h}}\leq 1$ and on the
other hand we know by the mentioned Remark that 
$$
\vac=\vac_\F\tensor\vac_\spib\in\Big(\F_0^{\mathrm{pol},0}\tensor\spib^{\mathrm{pol},0}\Big)^{\lie{h}}.
$$

Thus we have proved the following.
\begin{prop}\label{KernelOfD2}
The kernel of the square $\KDirac^2:(\F_0\tensor\spib_{\lie{p}}^c)^{\lie{h}}\To(\F_0\tensor\spib_{\lie{p}}^c)^{\lie{h}}$
is given by
$$
\ker(\KDirac^2)=\C\vac.
$$
In particular $\dim\ker(\KDirac^2)=1<\infty$.
\end{prop}

In order to compute the kernel of the actual Dirac operator itself, first notice that obviously $\ker(\KDirac)\subset\ker(\KDirac^2)$
so that $\ker(\KDirac)\subset\C\vac$. On the other hand by Proposition \ref{vacuumstr} $\C\vac\subset\ker(\KDirac)$ so
that
\begin{equation}
\ker(\KDirac)=\ker(\KDirac^2)=\C\vac.
\end{equation}

Now recalling from Proposition \ref{Diracdefined} that $\KDirac$ is symmetric we have culminated in:
\begin{thm} The Dirac operator 
$\KDirac:(\F_0\tensor\spib_{\lie{p}}^c)^{\lie{h}}\To(\F_0\tensor\spib_{\lie{p}}^c)^{\lie{h}}$ 
on $\RGr$ is an (unbounded) symmetric operator with finite
dimensional kernel given by $\ker(\KDirac)=\C\vac$. The square
$\KDirac^2$ is a diagonalizable, essentially self-adjoint (positive) operator with finite dimensional kernel
$\ker(\KDirac^2)=\C\vac$.
\end{thm}

\bibliographystyle{alpha-loc}

\appendix
\section{Quick reminder on self-adjoint operators}
By an (unbounded) operator $T$ from a complex Hilbert space $\Hilb$ to 
complex Hilbert space $\mathcal{K}$ we mean a linear mapping $T$ defined on 
a dense subspace $\domain(T)\subset\Hilb$ having values in $\mathcal{K}$. 
In particular, when we say that
``$T$ is an operator on $\Hilb$'' we do not mean to imply that
$Tx$ is defined for all $x\in\Hilb$ nor that $T$ is bounded.

We say that $T$ is an \emph{extension} of $S$ if $\domain(S)\subset\domain(T)$
and $Tx=Sx$ for all $x\in\domain(S)$. We shall write $S\subset T$ to indicate
that $T$ is an extension of $S$.

\subsection{Adjoint operators}
Let $T:\Hilb\To\mathcal{K}$ be an operator. Define the set
$$
\domain(T^\ast):=\set{y\in\mathcal{K}\mid x\mapsto\ip{Tx}{y}
\textrm{ is a bounded linear funtional on } \domain(T)}.
$$
Since we assume that $\domain(T)$ is dense in $\Hilb$, the function
$x\mapsto\ip{Tx}{y}$ extends to a bounded linear functional on the entire
$\Hilb$, so by the Riesz representation theorem, there exists a unique
$x^\ast=T^\ast y\in\Hilb$ such that
$$
\ip{Tx}{y}=\ip{x}{T^\ast y}.
$$
By the uniqueness assertion $T^\ast$ is a linear operator with domain
$\domain(T^\ast)$, and is called the \emph{adjoint} of $T$.
\begin{defn}
An operator $T:\Hilb\To\Hilb$ is \emph{self-adjoint} if $\domain(T^\ast)=\domain(T)$ and
$\ip{Tx}{y}=\ip{x}{Ty}$ for all $x\in\domain(T)$ and all $y\in\domain(T^\ast)$.
\end{defn}
\begin{defn}
An operator $T:\Hilb\To\Hilb$ is \emph{symmetric} (or formally self-adjoint) 
if it satisfies $\ip{Tx}{y}=\ip{x}{Ty}$
for all $x,y\in\domain(T)$.
\end{defn}
Please note that because of the dependence on the domain, there is a difference between a symmetric
and a self-adjoint operator. Clearly every self-adjoint operator is symmetric but the
opposite doesn't have to hold.
\begin{prop}
Let $T$ be a (densely defined) operator. Then $T$ is symmetric 
$\iff T\subset T^\ast\iff\ip{Tx}{x}\in\Real$ for all $x\in\domain(T)$.
\end{prop}
\begin{defn}
We say that an operator $T:\Hilb\To\Hilb$ is \emph{diagonalizable} (or discrete)
if there exists an orthonormal basis $\set{e_j\mid j=1,2,\ldots}$ for $\Hilb$
and a 
set $\set{\lambda_j\mid j=1,2,\ldots}$ in $\C$ such that
$$
Tu=\sum_{j=1}^\infty\lambda_j\ip{u}{e_j}e_j.
$$
\end{defn}
If a symmetric operator $T$ 
is diagonalizable one can extend its original domain $\domain(T)\subset\domain(T^\ast)$ to the domain
$$
\domain(\overline{T}):=\Big\{u\in\Hilb\,\Big\vert\,\sum_{j=1}^\infty\abs{\lambda_j}^2
\abs{\ip{u}{e_j}}^2<\infty\Big\}.
$$
by
$$
\overline{T}u=\sum_{j=1}^\infty\lambda_j\ip{u}{e_j}e_j.
$$
Then one can show that $\overline{T}=T^\ast$ . This implies the following proposition:
\begin{prop}\label{appendixlemma}
A symmetric and diagonalizable (unbounded) linear operator $T:\Hilb\To\Hilb$ is essentially self-adjoint, i.e.
it can be extended to a self-adjoint operator.
\end{prop}
\subsection{Strong convergence}
\begin{defn}
A sequence of (densely-defined) operators $T_n:\Hilb\To\Hilb$ is said to \emph{converge strongly} to $T:\Hilb\To\Hilb$,
\begin{eqnarray}
s\!\!-\!\!\!\!\lim_{n\to\infty} T_n=T\quad &:\iff& T_n\varphi\to T\varphi\quad\textrm{ for all } \varphi\in\domain(T)\subset\domain(T_n)\nonumber\\
& \iff & \lim\norm{T_n\varphi-T\varphi}_\Hilb=0,
\end{eqnarray}
where $\norm{\cdot}_\Hilb$ denotes the norm on $\Hilb$ induced by the Hilbert space inner product $\ip{\cdot}{\cdot}_\Hilb$, i.e.
$\norm{\varphi}_\Hilb:=\sqrt{\ip{\varphi}{\varphi}}$ for all $\varphi\in\Hilb$.
\end{defn}
\subsection{Analytic vectors}
\begin{defn}
Let $T$ be an operator on a Hilbert space $\Hilb$. The set $C^{\infty}(T):=\bigcap_{n=1}^\infty\domain(T^n)$
is called the $C^\infty$-\emph{vectors for} $T$. A vector $\varphi\in C^{\infty}(T)$ is called \emph{analytic vector for} $T$
if
$$
\sum_{n=0}^\infty\frac{\norm{T^n\varphi}}{n!}t^n<\infty
$$
for some $t>0$.
\end{defn}

\begin{defn}
Suppose $T$ is symmetric. For each smooth vector $\varphi\in C^{\infty}(T)$ define
$$
\domain_\varphi:=\set{\sum_{n=0}^N\alpha_n T^n\varphi\,\Big\vert\,N=1,2,\ldots,\,\langle\alpha_1,\ldots,\alpha_N\rangle\textrm{ arbitrary}},
$$
where $\langle\alpha_1,\ldots,\alpha_N\rangle$ denotes a (finite) sequence of complex numbers. Let
$\Hilb_\varphi:=\overline{\domain_\varphi}$ be the closure of $\domain_\varphi$ in $\Hilb$ and define
$T_\varphi:\domain_\varphi\To\domain_\varphi$ by $T_\varphi(\sum_{n=0}^N\alpha_n T^n\varphi)=\sum_{n=0}^N\alpha_n T^{n+1}\varphi$.
Then the smooth vector $\varphi$ is called a \emph{vector of uniqueness} if $T_\varphi$ is essentially self-adjoint on
$\domain_\varphi$ (as an operator on $\Hilb_\varphi$).
\end{defn}

\begin{defn}
A subset $S\subset\Hilb$ is called \emph{total} if the set of all finite linear combinations of elements of $S$ is dense in $\Hilb$.
\end{defn}

\begin{thm}[Nelson's analytic vector theorem]\label{Nelson}
Let $T$ be a symmetric operator on a Hilbert space $\Hilb$. If $\domain(T)$ contains a total set of analytic vectors, then
$T$ is essentially self-adjoint.
\end{thm}

\begin{cor}\label{NelsonCor}
Suppose that $T$ is a symmetric operator and let $\domain$ be a dense linear set contained in $\domain(T)$. Then if
$\domain$ contains a dense set of analytic vectors and if $\D$ is invariant under $T$, then $T$ is essentially
self-adjoint.
\end{cor}

\section{Highest weight representations of $\gl_\infty$}
\subsection{The fundamental representations}
Following \cite{KacRa} and \cite{Neeb} the Lie algebra $\gl_\infty$ is defined by
$$
\gl_\infty=\set{(a_{ij})_{i,j\in\Z}\mid a_{ij}\in\C,\, \textrm{ all but a finite number of the }
a_{ij}\textrm{ are zero}
}
$$ 
with the Lie algebra bracket being the ordinary matrix commutator.

Let $E_{ij}$ denote the matrix with $1$ as the $(i,j)$ entry and all other entries $0$. Then the
$E_{ij}\,(i,j\in\Z)$ form a basis for $\gl_\infty$ and the commutation relations of $\gl_\infty$ can be expressed
as the commutation relations
$$
[E_{ij},E_{mn}]=\delta_{jm}E_{in}-\delta_{ni}E_{mj}.
$$

Let $\lie{k}\subset\gl_\infty$ be the subalgebra of diagonal matrices considered as a Cartan subalgebra of
$\gl_\infty$ and define $\varepsilon_j\in\lie{k}^\ast$ by
$$
\varepsilon_j(\diag(a_{ii})):=a_{jj}.
$$
Then the set of roots of $\lie{g}:=\gl_\infty$ with respect to $\lie{k}$ is given by
$$
\Phi=\Phi(\gl_\infty,\lie{k}):=\set{\varepsilon_i-\varepsilon_j\mid i\neq j,\,i,j\in\Z},\quad\textrm{ where }
\lie{g}_{\varepsilon_i-\varepsilon_j}=\C E_{ij}.
$$
For the set of positive roots we may choose (\cite{Neeb}, Proposition II.1, p. 185)
$$
\Phi^+:=\set{\varepsilon_i-\varepsilon_j\mid i<j}.
$$

It is then clear that the Lie algebra $\gl_\infty$ admits a Cartan decomposition
$$
\gl_\infty=\lie{g}_+\oplus\lie{k}\oplus\lie{g}_-,
$$
where
\begin{align}
\lie{g}_+ &=\bigoplus_{\alpha\in\Phi^+}\lie{g}_\alpha= \{\sum_{\mathrm{finite}}a_{ij}E_{ij}\mid i<j\} & &\textrm{ (upper triangular matrices)}\nonumber\\
\lie{k}     &= \{\sum_{\mathrm{finite}}a_{ii}E_{ii}\} & &\textrm{ (diagonal matrices)}\nonumber\\
\lie{g}_- &= \bigoplus_{\alpha\in\Phi^-}\lie{g}_\alpha=
\{\sum_{\mathrm{finite}}a_{ij}E_{ij}\mid i>j\} & &\textrm{ (lower triangular matrices)}.\nonumber
\end{align}

\begin{defn}[\cite{KacRa}] Given a collection of complex numbers $\lambda=\set{\lambda_i\mid i\in\Z}$,
called a \emph{highest weight}, we define the \emph{irreducible highest weight representation} $\pi_\lambda$
of the Lie algebra $\gl_\infty$ as an irreducible representation on a vector space $L(\lambda)$ which admits a non-zero
vector $v_\lambda\in L(\lambda)$, called highest weight vector, such that
\begin{enumerate}
\item $\pi_\lambda(\lie{g}_+)v_\lambda=0$,
\item $\pi_\lambda(E_{ii})v_\lambda=\lambda_i v_\lambda.$
\end{enumerate}
\end{defn}

Next
we let $m\in\Z$ and consider irreducible unitary highest weight representations of $\gl_\infty$ with highest weight
$$
\omega_m=\set{(\lambda_i)_{i\in\Z}\mid\lambda_i=1\textrm{ for } i<m,\,\lambda_i=0
\textrm{ for } i\geq m},
$$
called the \emph{fundamental representations}.

The fundamental representations of $\gl_\infty$ can be realized in the charge $m$-sectors of fermionic Fock spaces,
$r_m:\gl_\infty\To\End(\F_m)$. 
More precisely, let
$\set{w_k}_{k\in\Z}$ be a basis of 
a separable polarized complex Hilbert space $\Hilb=\Hilb_+\oplus\Hilb_-$ obtained as a union of an
orthonormal basis $\set{u_k}$ of $\Hilb_+$ and an orthonormal
basis $\set{v_k}$ of $\Hilb_-$. Then the fundamental representation $r_m:\gl_\infty\To\End(\F_m)$ is given by
$$
r_m(E_{pq})=\psi^\ast_p\psi_q.
$$

By definition it satisfies
\begin{enumerate}
\item $\mathcal{U}(\gl_\infty)\psi_m$ is dense in $\F_m$;
\item $r_m(\lie{g}_+)\psi_m=0$;
\item $r_m(E_{ii})\psi_m=\lambda_i\psi_m$, where
$$
\lambda_i=
\left\{ \begin{array}{ll}
1 & \textrm{if $i<m$}\\
0 & \textrm{if $i\geq m$}
\end{array} \right.
$$
\end{enumerate} 
and where the highest weight vector $v_{\omega_m}=\psi_m\in\F_m$ is given in this concrete realization by (see \cite{Otte} p.154)
$$
\psi_m=
\left\{ \begin{array}{ll}
w_{m-1}\wedge\cdots\wedge w_1\wedge w_0 & \textrm{for $m>0$}\\
w_{-1}\wedge w_{-2}\wedge\cdots\wedge w_m & \textrm{for $m<0$}\\
\vac & \textrm{for $m=0$}.
\end{array} \right.
$$

\subsection{Taking tensor products}
The tensor product of two irreducible unitary highest weight representations with highest weight vectors
$v_\lambda$ and $v_\mu$, respectively, gives an irreducible unitary highest weight representation in the
\emph{highest component} of the tensor product of the representation spaces, i.e. in the vector space
generated by the highest weight vector $v_\lambda\tensor v_\mu$.

Thus, by the above we may construct irreducible unitary highest weight representations
$\pi_\lambda$ of $\gl_\infty$ in a vector space $L(\lambda)$ with highest weight
$$
\lambda=k_1\cdot\omega_{m_1}+\cdots + k_n\cdot\omega_{m_n}
$$
and highest weight vector
$$
v_\lambda=\psi_{m_1}^{\tensor k_1}\tensor\cdots\tensor\psi_{m_n}^{\tensor k_n},
$$
where
$$
\psi_{m_j}^{\tensor k_j}:=\psi_{m_j}\tensor\cdots\tensor\psi_{m_j}\quad (k_j\textrm{ times})
$$
and
$j=1,\ldots n,\, k_1,\ldots k_n\in\N,\, m_1,\ldots m_n\in\Z$ and $n\in\N$.

\subsection{Extensions to $\ainfb$ and $\ainf$}
Every fundamental representation $r_m$ of $\gl_\infty$ can be extended to a representation
$\rh_m$ on the of the larger Lie algebra $\ainfb\supset\gl_\infty$,
$$
\ainfb:=\set{(a_{ij})\mid i,j\in\Z,a_{ij}=0\textrm{ for } \abs{i-j}\gg 0},
$$ 
again equipped with the Lie bracket coming from the matrix commutator,
by defining
\begin{eqnarray}
\rh_m(E_{ij})&=& r_m(E_{ij})\quad\textrm{ if } i\neq j\textrm{ or } i=j\geq 0,\\
\rh_m(E_{ii})&=& r_m(E_{ii})-I\quad\textrm{ if } i<0.
\end{eqnarray}
It is easy to see that $\rh_m$ maps $\F_m$ into itself and it follows from the properties of $r_m$ listed above that
\begin{enumerate}
\item $\mathcal{U}(\ainfb)\psi_m$ is dense in $\F_m$;
\item $\rh_m(\lie{g}_+)\psi_m=0$;
\item $\rh_m(E_{ii})\psi_m=\tilde{\lambda}_i\psi_m$,
where
\begin{equation}\label{weightsforfunrep}
\tilde{\lambda}_i=
\left\{ \begin{array}{lll}
1 & \textrm{if $0\leq i\leq m-1$} &\quad(m>0)\\
-1 & \textrm{if $m\leq i< 0$} &\quad (m<0)\\
0 & \textrm{otherwise}.
\end{array} \right.
\end{equation}
\end{enumerate}

Next we define the Lie algebra
$$
\ainf:=\ainfb\oplus\C c
$$
with $\C c$ in the center and bracket
$$
[a,b]=ab-ba+s(a,b)c,
$$
where $s(\cdot,\cdot)$ is the two-cocyle defined by
\begin{eqnarray}
s(E_{ij},E_{ji})&=&-\alpha(E_{ji},E_{ij})=1,\quad\textrm{ if } i<0, j>0,\nonumber\\
s(E_{ij},E_{mn})&=&0\quad\textrm {in all other cases}.
\end{eqnarray}
Extending $\rh_m$ from $\ainfb$ to $\ainf$ by $\rh(c)=1$, we obtain a \emph{level} $1$ 
unitary irreducible linear representation
of the Lie algebra $\ainf$ in $\F_m$ (i.e. such that the central element $c$ acts by multiplication by $1$).

We denote by $\tilde{\omega}_m=(\tilde{\lambda_i})_{i\in\Z}$ the induced weights of $\ainf$ obtained from the fundamental weights $\omega_m$ of $\gl_\infty$
as described above. Just as before with $\gl_\infty$ we are able to construct unitary irreducible representations $L(\lambda)$
of $\ainf$, where
$$
\lambda=k_1\cdot\tilde{\omega}_{m_1}+\cdots + k_n\cdot\tilde{\omega}_{m_n}
$$
with $k_1,\ldots,k_n\in\N$ and $m_1,\ldots m_n\in\Z$. Since these are obtained by a tensor product construction, the
representation $L(\lambda)$ has level $\ell=\sum_{i=1}^n k_i$.  
\section{Normal ordered Casimir operator for $\ainf$ acting on $\F$}
\subsection{Introducing cut-offs}
The naive second order Casimir element for the infinite-dimensional Lie algebra
$\ainf$ is given by the infinite sum 
$$
\Delta_{\mathrm{naive}}:=\sum_{i,j\in\Z}E_{ij}E_{ji}
$$
Unfortunately,
this sum diverges in general and to make sense of it when acting on 
a highest weight representation of $\ainf$ (e.g. the charge $k$-sectors $\F_k\subset\F$
of a fermionic Fock space $\F=\F(\Hilb,\Hilb_+)$)
one has to introduce
a proper normal ordering prescription.

We start with the cut-off
$$
\Delta_{\mathrm{naive}}^{(N)}:=\sum_{\substack{i,j\in\Z\\ \abs{i},\abs{j}\leq N}}E_{ij}E_{ji}
$$
and write it as
\begin{eqnarray}
\Delta_{\mathrm{naive}}^{(N)}&=&\sum_{\substack{E_{ji}\in\lie{g}_+\\ \abs{i},\abs{j}\leq N}}
E_{ij}E_{ji}+
\sum_{\substack{E_{ji}\in\lie{g}_-\\ \abs{i},\abs{j}\leq N}}
E_{ij}E_{ji}+
\sum_{\substack{i\in\Z\\ \abs{i}\leq N}}E_{ii}^2\nonumber\\
&=&
\sum_{\substack{i,j\in\Z \\ j<i\\ \abs{i},\abs{j}\leq N}}E_{ij}E_{ji}+
\sum_{\substack{i,j\in\Z \\ j>i\\ \abs{i},\abs{j}\leq N}}E_{ij}E_{ji}+
\sum_{\substack{i\in\Z\\ \abs{i}\leq N}}E_{ii}^2\nonumber\\.
\nonumber
\end{eqnarray}
We may ``\,invert'' the order of the product in the second sum by using the commutation relations
$$
[E_{ij},E_{kl}]=\delta_{jk}E_{il}-\delta_{il}E_{kj}+s(E_{ij},E_{kl})
$$
and then move some of the terms inside the first sum. This way one obtains
\begin{eqnarray}
\Delta_{\mathrm{naive}}^{(N)}&=& 
2\sum_{\substack{i,j\in\Z \\ j<i\\ \abs{i},\abs{j}\leq N}}E_{ij} E_{ji}+
\sum_{\substack{i,j\in\Z \\ i<j \\ \abs{i},\abs{j}\leq N}}(E_{ii}-E_{jj})+
\sum_{\substack{i\in\Z \\ \abs{i}\leq N}}E_{ii}^2+
\sum_{\substack{i,j\in\Z \\ i<j,\, ij<0 \\ \abs{i},\abs{j}\leq N}}s(E_{ij},E_{ji}).\nonumber
\end{eqnarray} 
Noticing that the condition ($i<j$ and $ij<0$) is equivalent to the condition ($i<0$ and $j>0$)
the above expression becomes
\begin{eqnarray}
\Delta_{\mathrm{naive}}^{(N)}&=& 
2\sum_{\substack{i,j\in\Z \\ j<i\\ \abs{i},\abs{j}\leq N}}E_{ij} E_{ji}+
\sum_{\substack{i,j\in\Z \\ i<j \\ \abs{i},\abs{j}\leq N}}(E_{ii}-E_{jj})+
\sum_{\substack{i\in\Z \\ \abs{i}\leq N}}E_{ii}^2+
\sum_{\substack{i,j\in\Z \\ i<j,\, ij<0 \\ \abs{i},\abs{j}\leq N}}I.\nonumber
\end{eqnarray} 

Ignoring the diverging term $N^2=\sum_{ij<0,i<j} I$ in the above expression for $\Delta_{\mathrm{naive}}^{(N)}$ one naturally
comes up with the following definition.

\begin{defn} Let $N\in\N$ and define the $N^{\mathrm{th}}$ cut-off normal ordered Casimir operator for $\ainf$ by
$$
\Delta^{(N)}:=2\sum_{\substack{i,j\in\Z \\ j<i\\ \abs{i},\abs{j}\leq N}}E_{ij} E_{ji}+
\sum_{\substack{i,j\in\Z \\ i<j \\\abs{i},\abs{j}\leq N}}(E_{ii}-E_{jj})+
\sum_{\substack{i\in\Z\\ \abs{i}\leq N}}E_{ii}^2.
$$
\end{defn}

\begin{lem} One may write
\begin{equation}
\Delta^{(N)}=2\sum_{\substack{i,j\in\Z \\ j<i\\ \abs{i},\abs{j}\leq N}}E_{ij}E_{ji}+
\sum_{\substack{i\in\Z\\ \abs{i}\leq N}}E_{ii}(E_{ii}-2i).
\end{equation}
\end{lem}
\begin{proof}
Let
$$
S:=\sum_{\substack{i,j\in\Z \\ i<j\\ \abs{i},\abs{j}\leq N}}(E_{ii}-E_{jj})
=\sum_{\substack{i,j\in\Z\\ -N\leq i\leq N-1\\ -N+1\leq j\leq N}}(E_{ii}-E_{jj})
$$
and for a fixed $i\in\Z$ set
$$
S(i):=\sum_{\substack{j\in\Z\\ i<j\leq N}}(E_{ii}-E_{jj})
$$
so that
$$
S=\sum_{i=-N}^{N-1} S(i).
$$
Next write each $S(i)$ in the form
$$
S(i)=\sum_{\substack{j\in\Z\\ i<j\leq N}}E_{ii}-\sum_{\substack{j\in\Z\\ i<j\leq N}}E_{jj}
=(N-i)E_{ii}-\sum_{\substack{j\in\Z\\ i<j\leq N}}E_{jj}=
(N-i)E_{ii}-T(i),
$$
where
$$
T(i):=\sum_{\substack{j\in\Z\\ i<j\leq N}}E_{jj}=\overbrace{\sum_{\substack{j\in\Z \\ i<j\leq N-1}}E_{jj}}^{T_1(i)}+E_{NN}
$$
Thus
\begin{eqnarray}
S &=& \sum_{k=-N}^{N-1}(N-k)E_{kk}-\sum_{i=-N}^{N-1}\sum_{\substack{j\in\Z\\ i<j\leq N}}E_{jj}\nonumber\\
&=&
\sum_{k=-N}^{N-1}(N-k)E_{kk}-\sum_{i=-N}^{N-1} T(i)
=
\sum_{k=-N}^{N-1}(N-k)E_{kk}-\underbrace{\sum_{i=-N}^{N-1} T_1(i)}_{T_1}-\sum_{i=-N}^{N-1}E_{NN}\nonumber\\
&=&
\sum_{k=-N}^{N-1}(N-k)E_{kk}-T_1-2NE_{NN}
\end{eqnarray}

The term $E_{kk}$ $(-N\leq k\leq N-1)$ occurs in $T_1(i)$ iff $i<k$ and then it occurs exactly once.
Hence for a fixed $k$ as above, the term $E_{kk}$ occurs in $T_1$ for
$$
\#\set{i\in\Z\mid i\in[-N,k-1]\subset\Real}=k+N
$$
times. We conclude by combining terms that
\begin{eqnarray}
S &=& \sum_{-N}^{N-1}\Big((N-k)E_{kk}-(k+N)E_{kk}\Big)-2NE_{NN}=-2\sum_{k=-N}^{N-1} kE_{kk}-2NE_{NN}\nonumber\\
&=&
-2\sum_{k=-N}^N kE_{kk}.
\end{eqnarray}
from which the result follows.
\end{proof}
\subsection{Going to the limit}
The above computations motivate us to set the following definition.
\begin{defn} Define the normal ordered Casimir operator for $\ainf$ by
\begin{equation}\label{noCOpF}
\Delta:=2\sum_{\substack{i,j\in\Z \\ j<i}}E_{ij}E_{ji}+
\sum_{\substack{i\in\Z}}E_{ii}(E_{ii}-2i).
\end{equation}
\end{defn}

Our next task is to show that $\Delta$ is well-behaved when realized in highest weight representations $\rh:\ainf\To\End L(\lambda)$, where
$\lambda=\tilde{\omega}_m$ with $m\in\Z$.
More precisely, let $\rh$ be as above and let
$$
\Delta_{\rh}:=2\sum_{\substack{i,j\in\Z \\ j<i}}\rh(E_{ij})\rh(E_{ji})+
\sum_{\substack{i\in\Z}}\rh(E_{ii})\Big(\rh(E_{ii})-2i\Big)
$$
be a formal sum of operators acting on $L(\lambda)$. We need to show that if $v\in L(\lambda)$ then
$\Delta_{\rh}v$ is a well-defined element of $L(\lambda)$.

\begin{lem}\label{CasOperPsi}
Let $\lambda=\tilde{\omega}_m$ so that we are looking at the fundamental representation 
$\rh=\rh_m:\ainf\To\End(\F_m)$ and
let $\psi_m\in\F_m$ be the corresponding highest weight vector given explicitly in Appendix B. Then
$$
\Delta_{\rh}\psi_m=-m(m-2)\psi_m=\Big[1-(m-1)^2\Big]\psi_m.
$$
In particular,
$$
\Delta_{\rh}\psi_m=0\iff m=0,2
$$
and
$$
\Delta_{\rh}\psi_m\in\Z_+\psi_m\iff m=1.
$$
\end{lem}
\begin{proof}
Recalling that 
$\rh_m(\lie{g}_+)\psi_m=0$, the claim follows directly from equation (\ref{weightsforfunrep}) by simple computations using the well-known formula
$$
1+2+\cdots+n=\frac{n(n+1)}{2}.
$$
\end{proof}

Next we want to analyze the commutators $[\Delta_{\rh}, \rh(E_{mn})]$ for the elements $E_{mn}\in\ainf$. 

\begin{lem}\label{aInfCasCom}
Fix an integer $M\in\Z_{\geq 0}$ and consider all integers $N\in\Z_{\geq 0}$ satisfying
$N\geq M$. Then
for all $m,n\in\Z$ with $\abs{m},\abs{n}\leq M$
\begin{equation}
[\Delta^{(N)},E_{mn}]=
\left\{ \begin{array}{ll}
2\sign(m)E_{mn} & \textrm{if $mn<0$}, \\
0 & \textrm{otherwise}.
\end{array} \right.
\end{equation}
\end{lem}

\begin{proof}
Obviously, it is enough to prove only the case where $M=N$ and $N$ obtains all values in $\Z_{\geq 0}$.
Recall that for $N\in\N$
$$
\Delta^{(N)}=\Delta_{\mathrm{naive}}^{(N)}-\sum_{\substack{i,j\in\Z\\ i<j,ij<0\\ \abs{i},\abs{j}\leq N}}I.
$$
which implies we can compute the above commutators using the naive version of the Casimir operator:
$$
[\Delta^{(N)},E_{mn}]=[\Delta_{\mathrm{naive}}^{(N)},E_{mn}]=\Big[\sum_{\substack{i,j\in\Z \\ \abs{i},\abs{j}\leq N}}E_{ij}E_{ji},E_{mn}\Big]=
\sum_{\substack{i,j\in\Z \\ \abs{i},\abs{j}\leq N}}[E_{ij}E_{ji},E_{mn}],
$$
where we have assumed $\abs{m},\abs{n}\leq N$.

To simplify computations that follow we are going to make use of the finite-dimensional Lie algebra $\gl_{2N+1}(\C)$ and realize it as a set consisting of all
$(2N+1)\times (2N+1)$ complex matrices with row and column indices going from $-N$ to $N$.
Let $e_{ij}$ denote the matrix with $1$ as the $(i,j)$ entry and all other entries $0$. Then the
$e_{ij}\,(i,j\in\Z)$ form a basis for $\gl_{(2N+1)}(\C)$ and the commutation relations of $\gl_{(2N+1)}(\C)$ can be expressed
as the commutation relations
$$
[e_{ij},e_{mn}]=\delta_{jm}e_{in}-\delta_{ni}e_{mj},
$$
i.e. these are the same as the commutation relations defining $\ainf$ except that we don't have the two-cocycle $s(\cdot,\cdot)$ appearing.
Then it is known that the Casimir invariant of $\gl_{2N+1}(\C)$ may be given as 
$$
\Delta_{\gl_{2N+1}}=\sum_{\substack{i,j\in\Z\\ \abs{i},\abs{j}\leq N}}e_{ij}e_{ji}
$$
and that
$$
[\Delta_{\gl_{2N+1}},e_{mn}]=0
$$
for all the basis elements $e_{mn}\in\gl_{2N+1}(\C)$.

Going back to the original situation, recall first the abstract commutator formula
$$
[AB,C]=[AC,B]+A[B,C]
$$
to obtain
$$
[\Delta^{(N)},E_{mn}]=\sum_{\substack{i,j\in\Z\\ \abs{i},\abs{j}\leq N}}[E_{ij},E_{mn}]E_{ji}+
\sum_{\substack{i,j\in\Z\\ \abs{i},\abs{j}\leq N}}E_{ij}[E_{ji},E_{mn}]
$$
and
$$
0=[\Delta_{\gl_{2N+1}},e_{mn}]=\sum_{\substack{i,j\in\Z\\ \abs{i},\abs{j}\leq N}}[e_{ij},e_{mn}]e_{ji}+
\sum_{\substack{i,j\in\Z\\ \abs{i},\abs{j}\leq N}}e_{ij}[e_{ji},e_{mn}]
$$
from which we conclude that
$$
[\Delta^{(N)},E_{mn}]=\sum_{\substack{i,j\in\Z\\ \abs{i},\abs{j}\leq N}} s(E_{ij},E_{mn})E_{ji}+
\sum_{\substack{i,j\in\Z\\ \abs{i},\abs{j}\leq N}} s(E_{ji},E_{mn})E_{ij}.
$$

Since
$$
s(E_{ij},E_{mn})=\left\{ \begin{array}{ll}
-\sign(i)\delta_{in}\delta_{jm} & \textrm{if $ij<0$}, \\
0 & \textrm{otherwise}.
\end{array} \right.
$$
we have arrived at the identity
$$
[\Delta^{(N)},E_{mn}]=-\sum_{\substack{i,j\in\Z \\ ij<0\\ \abs{i},\abs{j}\leq N}}\sign(i)\delta_{in}\delta_{jm}E_{ji}
-\sum_{\substack{i,j\in\Z \\ ij<0\\ \abs{i},\abs{j}\leq N}}\sign(j)\delta_{jn}\delta_{im}E_{ij}.
$$
To analyze this we make a case study.
\begin{enumerate}
\item \emph{The case} $mn=0$. Since $ij<0$ one sees by looking at the possible values of the various Kronecker's delta functions
appearing above that necessarliy $[\Delta^{(N)},E_{mn}]=0$. 
\item \emph{The case} $mn>0$. Now $m$ and $n$ are nonzero and are either both positive or both negative, and $i$ and $j$ are nonzero integers with opposite signs.
Thus both conditions $\delta_{in}=1$ and $\delta_{jm}=1$ are \emph{never} valid at the same time implying their
product must always be zero. Similarly one sees that $\delta_{jn}\delta_{im}=0$ so that 
$[\Delta^{(N)},E_{mn}]=0$ in this case, too.
\item \emph{The case} $mn<0$. Now it is always possible to have $\delta_{in}\delta_{jm}=1$ and $\delta_{jn}\delta_{im}=1$ yielding
$$
[\Delta^{(N)},E_{mn}]=-\sign(n)E_{mn}-\sign(n)E_{mn}=2\sign(m)E_{mn}.
$$
\end{enumerate}
\end{proof}

\begin{cor}\label{ExtraCorCas}
For each $i,j\in\Z$, the limit $[\Delta,E_{mn}]$ is a well-defined operator
in each charge $q$-sector $\F_q\,\,(q\in\Z)$ of the fermionic Fock space $\F$. 
In fact, for each pair of integers $i,j\in\Z$, there exists an integer
$N_{i,j}\in\Z_{\geq 0}$ such that
\begin{equation}
[\Delta,E_{ij}]=[\Delta^{(N)},E_{ij}]=
\left\{ \begin{array}{ll}
2\sign(m)E_{mn} & \textrm{if $mn<0$}, \\
0 & \textrm{otherwise}.
\end{array} \right.
\end{equation}
for all $N\geq N_{i,j}$.
In particular, one finds
that for each $k\in\Z,k<0$.
\begin{equation}\label{DeltaSk}
[\Delta,E_{i,i+k}]=
\left\{ \begin{array}{ll}
2E_{i,i+k} & \textrm{if $0<i<-k$}, \\
0 & \textrm{otherwise}.
\end{array} \right.
\end{equation}
\end{cor}

\begin{proof}
Let $M=\max\set{\abs{m},\abs{n}}$. Then it is easy to see from the explicit commutation relations of
$\ainf$ that $[E_{ij},E_{mn}]=0=[E_{ji},E_{mn}]$ whenever $\abs{i}>M$ or $\abs{j}>M$. Thus, 
the commutator in question is actually a finite sum of operators,
$$
[\Delta,E_{mn}]=[\Delta^{(M)},E_{mn}],
$$ 
and therefore well-defined with value given by the Lemma \ref{aInfCasCom}. The same Lemma
also tells us that $[\Delta^{(N)},E_{mn}]=[\Delta^{(M)},E_{mn}]$ whenever $N\geq M$. Thus,
we may choose $N_{i,j}=M$.
\end{proof}

Next, recall from \cite{Otte} \S 4.4. that the operators
$$
s_k:=\sum_{i\in\Z}E_{i,i+k}\quad(k\in\Z)
$$
satisfying
\begin{equation}
[s_n,s_k]=n\delta_{n,-k}\cdot I.
\end{equation}
give a representation of the \emph{Heisenberg Lie algebra} in each charge $q$-sector $\F_q\,\,(q\in\Z)$.
Moreover, the vectors of the form
\begin{equation}
s_{-k_1}\cdot\ldots\cdot s_{-k_n}\psi_q,\quad\textrm{ where } k_1\geq\cdots\geq k_n>0
\end{equation}
yield a basis for $\F_q$. It follows then from equation (\ref{DeltaSk}) that for an integer $k<0$
\begin{equation}
[\Delta,s_k]=\sum_{i\in\Z}[\Delta,E_{i,i+k}]=2\sum_{i=1}^{-k}E_{i,i+k},
\end{equation}
which is a finite sum of well-defined operators acting on each $\F_q\,\,(q\in\Z)$, and is
thus well-defined.
\begin{prop}
The normally ordered Casimir invariant $\Delta$ of $\ainf$ gives
a well-defined (unbounded) operator when realized in the fundamental representations
$\rh=\rh_q:\ainf\To\End(\F_q)$. The operator 
$\Delta_{\rh}$ has a dense domain $\D(\Delta_{\rh})=\F_q^{\mathrm{pol}}$.
\end{prop}
\begin{proof}
It is enough to show that $\Delta$ operates well on the basis vectors
\begin{equation}\label{skBasisForF}
w=s_{-k_1}\cdot\ldots\cdot s_{-k_n}\psi_q,\quad\textrm{ where } k_1\geq\cdots\geq k_n>0
\end{equation}
of $\F_q^{\mathrm{pol}}$, which is clearly dense in $\F_q$. This can be proved by induction
with respect to the length
 $\Length(w):=n$, where the value $n=0$ corresponds to the
vacuum vector $\psi_q$. 

By Lemma \ref{CasOperPsi} we know that the statement is true for
$n=0$. We make the inductive assumption that $\Delta$ is well-defined
on all basis vectors (\ref{skBasisForF}) with $\Length(w)\leq n\in\N$. Let then
$$
w'= s_{-k_1}\cdot\ldots\cdot s_{-k_n}s_{-k_{(n+1)}}\psi_q,\quad\textrm{ where } k_1\geq\cdots\geq k_n>0
$$
be a basis vector of length $n+1$. Using the algebraic identity
$AB=BA+[A,B]$ one may write
$$
\Delta w'=s_{-k_1}\Delta s_{-k_2}\cdot\ldots\cdot s_{-k_n}s_{-k_{(n+1)}}\psi_q+
[\Delta, s_{-k_1}] s_{-k_2}\cdot\ldots\cdot s_{-k_n}s_{-k_{(n+1)}}\psi_q.
$$
The second summand here is well-defined since we noticed that for negative $k\in\Z$ the operator
$[\Delta,s_k]$ makes sense in $\F_q^{\mathrm{pol}}$, and the first summand
is well-defined by our induction assumption.
\end{proof}

\subsection{Finding eigenvalues}
Next we consider $\rh:\ainf\To\End(L_\lambda)$ as a \emph{projective}
highest weight 
representation of $\gl_\infty$ via restriction. If we denote by
$v_\lambda\in L(\lambda)$ the corresponding highest weight vector, then we know
that the set
$$
L(\lambda)_{\mathrm{pol}}:=\rh(U(\gl_\infty))v_\lambda\subset L(\lambda)
$$
is dense in $L(\lambda)$.

By the Poincar\'e--Birkhoff--Witt theorem any ordered monomial in
$U(\gl_\infty)$ can be written as a finite product of the form
$$
u=\prod_{E_{ij}\in\lie{g}_-} E_{ij}^{r_{ij}}
\prod_{E_{kl}\in\lie{h}} E_{kl}^{s_{kl}}
\prod_{E_{mn}\in\lie{g}_+} E_{mn}^{p_{mn}},\quad
r_{ij},\, s_{kl},\, p_{mn}\in\N,
$$
i.e.
$$
u=\prod_{\substack{i,j\in\Z \\ i>j}} E_{ij}^{r_{ij}}
\prod_{k\in\Z} E_{kk}^{s_{kk}}
\prod_{\substack{m,n\in\Z \\ m<n}} E_{mn}^{p_{mn}},\quad
r_{ij},\, s_{kk},\, p_{mn}\in\N.
$$
It follows from this and the definition of a highest weight representation
that any element of $L(\lambda)_{\mathrm{pol}}$ can be written as a linear
combination of vectors of the form
$$
w=\Big(\prod_{\substack{i,j\in\Z \\ i>j}} E_{ij}^{r_{ij}}\Big)v_\lambda,
$$
with only finitely many terms $r_{ij}\neq 0$,
i.e. $L(\lambda)_{\mathrm{pol}}$ is generated by all finite $\C$-linear combinations of elements of
the form $U(\lie{g}_-)v_\lambda$.

\begin{prop}\label{CasimirInOfAinfWD}
Let 
$\rh=\rh_m:\ainf\To\End(\F_m)$ be the $m$-th fundamental representations ($m\in\Z$)
of $\ainf$. The operator 
$\Delta_{\rh}:\F_m\To\F_m$ has a dense domain $\D(\Delta_{\rh})=\F_m^{\mathrm{pol}}=U(\lie{g}_-)\psi_m$ and
if
\begin{equation}\label{wdirectly}
w=E_{i_1 j_1}\cdots E_{i_n j_n}\psi_m\in\F_m^{\mathrm{pol}}\qquad(i_k>j_k\textrm{ for all }k=1,\ldots n)
\end{equation} 
with the different $E_{i_j j_k}$'s in this expression not necessarily distinct 
then the image vector
$$
\Delta_{\rh}w=\Big[2\mathrm{Num}(w)+1-(m-1)^2\Big]w,
$$
where $\mathrm{Num}$ denotes the number
$$
\mathrm{Num}(w):=\#\set{\textrm{those }E_{i_k j_k}\textrm{ in (\ref{wdirectly}) with } i_k>0\textrm{ and } j_k<0}.
$$
\end{prop}

\begin{proof}
We first look at the cut-off $\Delta^{(N)}$ for $N\in\N$ such that $N>\abs{m}$ so that the result of
Lemma \ref{CasOperPsi} applies as well to $\Delta^{(N)}$. Moreover, we assume that
all indices $i,j$ in the expression
\begin{equation}\label{w}
w=\prod_{\substack{i,j\in\Z \\ i>j}} E_{ij}^{r_{ij}}\psi_m
\end{equation}
satisfy $\abs{i},\abs{j}\leq M$, where $M\in\Z_{\geq 0}$ is a fixed integer satisfying $M\leq N$. 
Notice that the set
\begin{eqnarray}
\mathfrak{S}^{(M)}&:=&\lie{g}_-\cap\set{E_{ij}\mid i,j\in\Z,\, ij<0,\,\abs{i},\abs{j}\leq M} \nonumber\\
&=&\set{E_{ij}\mid i,j\in\Z,\, i>j,\, ij<0,\, \abs{i},\abs{j}\leq M}\nonumber\\
&=&
\set{E_{ij}\mid i>0,\,j<0, \abs{i},\abs{j}\leq M}\nonumber
\end{eqnarray}
so that according to Lemma \ref{aInfCasCom} for the $E_{ij}$'s appearing in equation (\ref{w}) it holds that
\begin{eqnarray}
[\Delta^{(N)}_{\rh}, \rh(E_{ij})] &=&
\left\{ \begin{array}{ll}
2\sign(i)\rh(E_{ij}) & \textrm{if $E_{ij}\in\mathfrak{S}^{(M)}$}, \\
0 & \textrm{otherwise}
\end{array} \right.\nonumber \\
&=& 
\left\{ \begin{array}{ll}
2\rh(E_{ij}) & \textrm{if $i>0$ and $j<0$}, \\
0 & \textrm{otherwise}.
\end{array} \right.
\nonumber
\end{eqnarray}

Thus, if $E_{i_k j_k}$ appears in (\ref{w}), and $n\in\N$ with $n>0$, we write
$\Delta^{(N)}_{\rh}=[\Delta^{(N)}_{\rh},E_{i_k j_k}]+E_{i_k j_k}\Delta^{(N)}_{\rh}$ and if furthermore we write $w$ in the form
$$
w=E_{i_1 j_1}\cdots E_{i_n j_n}\psi_m
$$
then
\begin{eqnarray}
& & E_{i_1 j_1}\cdots\Delta^{(N)}_{\rh}E_{i_k j_k}E_{i_{k+1} j_{k+1}}\cdots E_{i_n j_n}\psi_m\nonumber\\
&=&
\left\{ \begin{array}{ll}
2w+
E_{i_1 j_1}\cdots E_{i_k j_k}\Delta^{(N)}_{\rh}E_{i_{k+1} j_{k+1}}\cdots E_{i_n j_n}\psi_m & \textrm{if $i>0$ and $j<0$}, \nonumber\\
E_{i_1 j_1}\cdots E_{i_k j_k}\Delta^{(N)}_{\rh}E_{i_{k+1} j_{k+1}}\cdots E_{i_n j_n}\psi_m
& \textrm{otherwise}.\nonumber
\end{array} \right.
\end{eqnarray}
It follows then by induction that
$$
\Delta^{(N)}_{\rh}w=2\mathrm{Num}(w)+E_{i_1 j_1}\cdots E_{i_n j_n}\Delta^{(N)}_{\rh}\psi_m.
$$
Lemma \ref{CasOperPsi} then gives
$$
\Delta^{(N)}_{\rh}w=\Big[2\mathrm{Num}(w)+1-(m-1)^2\Big]w.
$$
 
Noticing that in our situation at hand $\Delta^{(N)}_{\rh}w=\Delta_{\rh}w$,  and taking into account
Corollary \ref{ExtraCorCas}, one sees that
for sufficiently large $N$, we are allowed to replace $\Delta^{(N)}_{\rh}$ with $\Delta_{\rh}$ in the above computations.
The claim then follows by noticing that each vector $w$ of the form (\ref{wdirectly}) necessarily belongs to some
$\mathfrak{S}^{(M)}$.
\end{proof}

\begin{cor}\label{DiagonalizingCasOfG}
Let $m,M\in\Z$ with $M\geq 0$ and let $\F_m^{\mathrm{pol},M}\subset\F_m^{\mathrm{pol}}$ be the
subspace of $\F_m^{\mathrm{pol}}$ generated by all finite linear combinations of elements $w$
of the form (\ref{wdirectly}) with $\mathrm{Num}(w)=M$. Then
$$
\Delta_{{\rh}_m}\Big\vert_{\F_m^{\mathrm{pol},M}}=\Big[2M+1-(m-1)^2\Big]\cdot\id_{\F_m^{\mathrm{pol},M}}.
$$
\end{cor}

Clearly we have
$$
\F_m^{\mathrm{pol}}=\sum_{M=0}^\infty \F_m^{\mathrm{pol},M}\quad\textrm{(algebraic sum)},
$$
and taking into account Corollary \ref{DiagonalizingCasOfG}, one sees that this sum is actually direct,
$$
\F_m^{\mathrm{pol}}=\bigoplus_{M=0}^\infty \F_m^{\mathrm{pol},M}\quad\textrm{(algebraic direct sum)}.
$$

Next we are going to look at more carefully the case $m=0$. Then according to Corollary \ref{DiagonalizingCasOfG}
\begin{equation}\label{DiagCasGF0}
\Delta_{{\rh}_0}\Big\vert_{\F_0^{\mathrm{pol},M}}=2M\cdot\id_{\F_0^{\mathrm{pol},M}}.
\end{equation}

\begin{prop}[Diagonalization of $\Delta_{{\rh}_0}$]\label{LastOfTheKernels}
On $\F_0^{\mathrm{pol}}$ the operator $\Delta_{{\rh}_0}$ is besides diagonalizable it is also
non-negative with kernel equal to
\begin{equation}
\ker(\Delta_{{\rh}_0})=\F_0^{\mathrm{pol},0}=\C\vac_\F.
\end{equation}
In particular, $\dim\ker(\Delta_{{\rh}_0})=1<\infty$.
\end{prop}
\begin{proof}
By Corollary \ref{DiagonalizingCasOfG} the only thing left to do is to show that
in the case of $\F_0$ the direct sum decomposition
$$
\F_0^{\mathrm{pol}}=\bigoplus_{M=0}^\infty \F_0^{\mathrm{pol},M}\quad\textrm{(algebraic direct sum)}
$$
is an \emph{orthogonal decomposition}, i.e. that
$$
\F_0^{\mathrm{pol},M}\perp\F_0^{\mathrm{pol},N}
$$
when $M\neq N$.

Now the first thing to notice is that the Poincar\'e--Birkhoff-Witt argument given right below Lemma \ref{aInfCasCom} can be
taken one step further in the case of $\F_0$ since we know that $E_{ij}\vac_\F=0$ unless $i>0$ and $j<0$. Namely,
we express $\lie{g}_-$ as a direct sum of
$\{\sum_{\textrm{finite}}a_{ij}E_{ij}\mid a_{ij}\in\C,\, i>j,\, ij<0\}$ and its complement in $\lie{g}_-$. After this
we run through the same argument used before to conclude that $\F_0^{\mathrm{pol}}$ is generated
by all finite $\C$-linear combinations of the vectors of the form
$$
w=\prod_{\substack{i,j\in\Z \\ i>0,j<0}}E_{ij}^{r_{ij}}\vac_\F,
$$
where only finite number of the $r_{ij}$'s are different from $0$. 

In particular, we may
always assume that $i\neq j$ in the above expression so that
$$
\rh(E_{ij})=r(E_{ij})=\psi_{i}^\ast\psi_{j}\qquad(i>0,\,j<0).
$$
Moreover, notice that the anti-commutator $\{\psi_{i}^\ast,\psi_j\}=\delta_{ij}=0$ when $i$ and $j$ have different signs.
It follows from this and the other anti-commutations relations of the $CAR$ algebra $\CAR(\Hilb,\Hilb_+)$,
namely that all the other defining anti-commutation relations are zero,  
 that we may write the element
\begin{equation}\label{wInF0}
w=E_{i_1 j_1}\cdots E_{i_n j_n}\vac_\F\qquad  (i_k>0, j_k<0\textrm{ for all } k=1,\ldots n)
\end{equation}
in the form
\begin{equation}\label{wInThen}
w=\pm \psi_{i_1}^\ast\cdots\psi_{i_n}^\ast\psi_{j_1}\cdots\psi_{j_n}\vac_\F \qquad  (i_k>0, j_k<0\textrm{ for all } k=1,\ldots n),
\end{equation}
i.e. modulo sign an element of the standard orthonormal basis of $\F$. Also notice that since in $\CAR(\Hilb,\Hilb_+)$
$$
2(\psi_i^\ast)^2=\{\psi_i^\ast,\psi_i^\ast\}=0\qquad\textrm{ and }\qquad 2\psi_i^2=\{\psi_i,\psi_i\}=0
$$ 
all of the indices $i_1,\ldots, i_n$ and $j_1,\ldots j_n$ in (\ref{wInF0}) are pairwise distinct for nonzero $w$. Also
notice that if $w\neq 0$ we have
$$
\mathrm{Num}(w)= n\in\Z_{\geq 0},
$$
i.e essentially the `length' of $w$ when measured as the number of 
$\psi_{i_k}$'s (resp. $\psi_{i_k}^\ast$'s) appearing in the
expression (\ref{wInThen}), and the space
$\F^{\mathrm{pol},n}$ is generated by all finite $\C$-linear combinations of the standard orthonormal basis vectors of $\F$ of the form
(\ref{wInThen}) with the $+$ sign and all the indices being pairwise distinct. Clearly then
$$
\F_0^{\mathrm{pol},M}\perp\F_0^{\mathrm{pol},N}
$$
if $M\neq N$. It also follows from this that $\F_0^{\mathrm{pol},0}=\C\vac_\F$.
\end{proof}

\begin{prop}\label{stronglimitCasimir}
The unbounded operators $\Delta_{\rh_0},\Delta_{\rh_0}^{(N)}:\F_0\To\F_0$ with common dense
domain $\F_0^{\mathrm{pol}}$ satisfy the strong limit condition
\begin{equation}
\stlim{N} \Delta_{\rh_0}^{(N)}=\Delta_{\rh_0}.
\end{equation}
\end{prop}
\begin{proof}
Let $\F_0^{\mathrm{pol},(M)}$ denote the $\C$-subvector space of $\F_0^{\mathrm{pol}}$ consisting of all finite $\C$-linear combinations of
those basis vectors
\begin{equation}
w=E_{i_1 j_1}\cdots E_{i_n j_n}\vac_\F\quad (i_k>0, j_k<0\textrm{ for all } k=1,\ldots n)
\end{equation}
of $\F_0^{\mathrm{pol}}$ satisfying $\abs{i_k},\abs{j_k}\leq M$ for all $k=1,\ldots n$ (we interpret the case $n=0$ to mean the vacuum vector $\vac_\F$, which is
included in this subset). Of course, those basis vectors then yield a basis
of $\F_0^{\mathrm{pol},(M)}$. On the other hand, the union $\bigcup_{m\in\Z_{\geq 0}}\F_0^{\mathrm{pol},(M)}=\F_0^{\mathrm{pol}}$, so
that it is enough by linearity to show that for each element
$$
w=E_{i_1 j_1}\cdots E_{i_n j_n}\vac_\F\quad (i_k>0, j_k<0,\,\abs{i_k},\abs{j_k}\leq M\textrm{ for all } k=1,\ldots n)
$$
there exists an integer $N_w\in\Z_{\geq 0}$ such that
$\Delta_{\rh_0}^{(N)}w=\Delta_{\rh_0}w$ whenever $N\geq N_w$. But this follows from
the proof of Proposition \ref{CasimirInOfAinfWD}.
\end{proof}


\begin{thebibliography}{100000}
\bibitem [B]{B} D.~Borthwick, The Pfaffian line bundle, \emph{Comm.~Math.~Phys.}  \textbf{149} (1992), 463--493.
\bibitem [Fri]{Fri} T.~Friedrich, \emph{Dirac operators in Riemannian geometry}. Graduate Studies in Mathematics, Vol
\textbf{25}. American Mathematical Society. Providence, Rhode Island (2000).
\bibitem [Ga]{Ga} S.~A.~Gaal, \emph{Linear analysis and representation theory}.
Springer, Berlin (1973).
\bibitem [He]{He} S. ~Helgason, \emph{Differential geometry, Lie groups and symmetric spaces}. Academic Press, New York (1978).
\bibitem [HuaPan]{HuaPan} J.~S.~Huang, P.~Pand\v{z}i\'c, \emph{Dirac operators in representation theory}. Birkhäuser. Boston (2006).
\bibitem [KacRa]{KacRa} V.~G.~ Kac, A.~K.~ Raina, \emph{Highest weight representations of infinite dimensional Lie algebras}.
World Scientific (1987).
\bibitem [Ko]{Ko} B.~Kostant. A cubic Dirac operator and the emergence of Euler number multiplets of representations for equal rank subgroups,
\emph{Duke Math.~J.} \textbf{100} (1999), 447--501.
\bibitem [La]{La} G.~D.~Landweber, \emph{Dirac operators on Loop Spaces}. PhD Thesis,
Harvard University, Cambridge, Massachusetts (1999). Electronically available at
http://math.bard.edu/greg/LoopDirac.pdf
\bibitem [LaMich]{LaMich} H.~B.~ Lawson, M--L.~ Michelsohn, \emph{Spin geometry}.
Princeton University Press. Princeton, New Jersey (1989).
\bibitem [Mic]{Mic} J.~Mickelsson, \emph{Current algebras and groups}.
Plenum Monographs in Nonlinear Physics. Plenum Press, New York (1989).
\bibitem[Neeb]{Neeb} K-H.~Neeb, Holomorphic highest weight representations of infinite dimensional complex classical groups,
\emph{J.~Reine angew. Math.} \textbf{497} (1998), 171-222.
\bibitem[Neeb2]{Neeb2} K-H.~Neeb, Infinite-dimensional Lie groups. In
\emph{Infinite dimensional Kähler manifolds} (Oberwolfach, 1995), 287--375, DMV Sem., \textbf{31},
Birkhäuser, Basel, 2001.
\bibitem [Otte]{Otte} J.~T.~Ottesen, \emph{Infinite dimensional groups and algebras in quantum physics}. Lecture Notes in Physics.
Springer--Verlag, Berlin--Heidelberg (1995).
\bibitem [Pi1]{Pi1} D.~Pickrell, Measures on infinite dimensional Grassmann manifolds, 
\emph{J.~Funct.~Anal.} \textbf{70}, no. 2 (1987), 323--356.
\bibitem [Pi2]{Pi2} D.~Pickrell, On the support of quasi-invariant measures on infinite-dimensional Grassman manifolds,
\emph{Proc.~Amer.~Soc.} \textbf{100}, no. 1 (1987), 111--116.
\bibitem [PrSe]{PrSe} A.~Pressley, G.~Segal, \emph{Loop groups}, Oxford University Press, New York, 1986.
\bibitem [SpeWu]{SpeWu} M.~Spera, T.~Wurzbacher, Differential geometry of Grassmannian embeddings
of based loop groups, \emph{Differential geometry and its Applications} \textbf{13} (2000), 43--75.
\bibitem [Wu]{HuWu} T.~Wurzbacher, 
Fermionic second quantization and the geometry of the restricted Grassmannian. In
\emph{Infinite dimensional Kähler manifolds} (Oberwolfach, 1995), 287--375, DMV Sem., \textbf{31},
Birkhäuser, Basel, 2001.
\end{thebibliography}
\end{document}